\documentclass[12pt]{article}

\usepackage{graphicx}
\usepackage{latexsym,amssymb}
\usepackage{amsthm}
\usepackage{indentfirst}
\usepackage{amsmath}
\usepackage{color}
\usepackage{xcolor}
\usepackage{fourier}
\usepackage[colorlinks=true,backref=page]{hyperref}
\usepackage{enumitem}
\usepackage{float}

\newtheorem{theoremalph}{Theorem}

\newtheorem*{Main Theorem}{Main Theorem}
\newtheorem{Coro}[theoremalph]{Corollary}
\newtheorem{Theorem}{Theorem}[section]
\newtheorem*{Theorem A}{Theorem A}
\newtheorem*{Theorem A'}{Theorem A'}
\newtheorem*{Theorem B'}{Theorem B'}

\newtheorem{Definition}[Theorem]{Definition}
\newtheorem{Proposition}[Theorem]{Proposition}
\newtheorem{Lemma}[Theorem]{Lemma}
\newtheorem{Example}[Theorem]{Example}

\newtheorem{Question}{Question}

\newtheorem{Remark}[Theorem]{Remark}
\newtheorem*{Remark*}{Remark}

\newtheorem{Remark-numbered}[Theorem]{Remark}

\newtheorem*{Claim}{Claim}
\newtheorem{Claim-numbered}[Theorem]{Claim}

\begin{document}

	\title{Entropy formula for surface diffeomorphisms}

	\author{Yuntao Zang\\ \small School of Mathematical Sciences, Soochow University, Suzhou, P.R. China\\
		\small \texttt{ytzang@suda.edu.cn}}
\date{}

	\maketitle
	\begin{abstract}

		\medskip
	Let $f$ be a $C^r$ ($r>1$) diffeomorphism on a compact surface $M$ with $h_{\rm top}(f)\geq\frac{\lambda^{+}(f)}{r}$ where $\lambda^{+}(f):=\lim_{n\to+\infty}\frac{1}{n}\max_{x\in M}\log \left\|Df^{n}_{x}\right\|$. We establish an equivalent formula for the topological entropy: $$h_{\rm top}(f)=\lim_{n\to+\infty}\frac{1}{n}\log\int_{M}\left\|Df^{n}_{x}\right\|\,dx.$$ We also characterize the topological entropy via the volume growth of curves and present several applications. Our approach builds on the key ideas developed in the works of Buzzi-Crovisier-Sarig (\emph{Invent. Math.}, 2022) and Burguet (\emph{Ann. Henri Poincar\'e}, 2024) concerning the continuity of the Lyapunov exponents.
	
	\medskip
	\noindent\textbf{\normalsize MSC Class:} 37A05; 37B40; 37C40; 37D25; 37E30

	\end{abstract}
	
	\tableofcontents
	\section{Introduction}
	\subsection{Statement of the main results}
	
	Let $f$ be a diffeomorphism on a compact surface $M$. Write $$\lambda^{+}(f):=\lim_{n\to+\infty}\frac{1}{n}\max_{x\in M}\log \left\|Df^{n}_{x}\right\|.$$ The sequence $\left\{\max_{x\in M}\log \left\|Df^n_x\right\|\right\}_{n\ge 1}$ is sub-additive, and hence the above limit exists.

	 For surface diffeomorphisms with large entropy, we establish an equivalent formula for the topological entropy in terms of the \emph{volume growth of the tangent cocycle} :
	\begin{theoremalph}\label{main result}
		For any $C^r$ ($r>1$) diffeomorphism $f$ on a compact surface $M$ with $h_{\rm top}(f)\geq\frac{\lambda^{+}(f)}{r}$, we have $$h_{\rm top}(f)=\lim_{n\to+\infty}\frac{1}{n}\log\int_{M}\left\|Df^{n}_{x}\right\|\,dx.$$

	\end{theoremalph}
Let us first make some remarks about Theorem \ref{main result}.
	\begin{itemize}
		\item Note that the integral is taken with respect to the Lebesgue measure \footnote{Throughout, we assume that the Lebesgue measure on $M$ is normalized to be a probability measure.}. The Lyapunov regular set (that is, the set of points with well-defined Lyapunov exponents) usually has zero Lebesgue measure and therefore does not directly contribute to the integral above. Even for ergodic volume-preserving Anosov systems, the positive Lyapunov exponent equals the entropy of the Lebesgue measure almost everywhere, yet the integral in Theorem~\ref{main result} does not converge to this entropy. Instead, it converges to the topological entropy. In this sense, for sufficiently large time $n$, it is the \emph{exceptional} points (those for which the derivative has not yet grown at the rate of the positive Lyapunov exponent up to time $n$) that dominate the integral. The order of the operations $\log$ and $\int$ is crucial (see Corollary~\ref{order of log and integral}).

		\item In the special case $r=\infty$, no additional term is added to the topological entropy, and our result thus implies the main results of Kozlovski \cite{Koz98} for surface diffeomorphisms. Also note that Theorem~\ref{main result} is formulated directly in terms of the derivative $Df^{n}_{x}$, rather than the induced map $(Df^{n}_{x})^{\wedge}$ between exterior algebras of the tangent spaces which is considered by Kozlovski \cite{Koz98}. Although the equivalence of these two quantities for evaluating the topological entropy is not apparent from their respective forms, it is nonetheless natural to work directly with $Df^{n}_{x}$ in the two-dimensional setting, after a more careful analysis. The use of the derivative $Df^{n}_{x}$ in Theorem~\ref{main result} might provide a more effective approach for numerically estimating the topological entropy.

		The extension of the formula in Theorem~\ref{main result} from the $C^{\infty}$ to the $C^{r}$ setting is far from routine and requires significant improvements in Yomdin theory. The key obstruction lies in the presence of local complexity (indicated by the Yomdin term $\frac{\lambda^{+}(f)}{r}$). In the $C^{\infty}$ case, this obstruction disappears, which drastically simplifies the problem \footnote{This is also the case for Anosov systems, where no local complexity is present, and the formula in Theorem~\ref{main result} can be proved in a relatively straightforward way, without any assumption on the topological entropy or the regularity (see, e.g., Kifer-Newhouse \cite{KiN91}, Yang-Zang \cite{YaZ23}).}. In contrast, in the $C^{r}$ setting, the persistence of a nontrivial local complexity term makes the analysis substantially more delicate. We refer to the background in Section~\ref{Backgrounds and applications} for a detailed comparison between the two settings.

		\item The threshold $\frac{\lambda^{+}(f)}{r}$ for the topological entropy is typically not expected to be attained in Yomdin-type results. For instance:
		\begin{itemize}
			\item For $C^r$ ($r>1$) surface diffeomorphisms with $h_{\rm top}(f)>\frac{\lambda^{+}(f)}{r}$, Buzzi-Crovisier-Sarig \cite{BCS22'} (see also subsequent extensions by Buzzi-Luo-Yang \cite{BLY25}, Burguet-Luo-Yang \cite{BLY25'} and Zang \cite{Zan25})  showed that there are at most finitely many ergodic measures of maximal entropy.
			\item For $C^r$ ($r>1$) surface diffeomorphisms with $h_{\rm top}(f)>\frac{\lambda^{+}(f)}{r}$, Burguet \cite{Bur24} showed that if $\mu_{n}$ is a sequence of ergodic measures with $h(f,\mu_{n})\to h_{\rm top}(f)$ and $\mu_{n}\to \mu$ for some invariant measure $\mu$, then $\mu$ is a measure of maximal entropy and the positive Lyapunov exponent of $\mu_{n}$ converges to that of $\mu$.
			\item For $C^r$ ($r>1$) surface diffeomorphisms, Burguet \cite{Bur24'} also showed that if Lebesgue almost every point has maximal Lyapunov exponent strictly larger than $\frac{\lambda^{+}(f)}{r}$, then the system admits an SRB measure.
		\end{itemize}
		All these results rely on the assumption that certain invariants-such as the topological entropy or Lyapunov exponents-are \emph{strictly} larger than the threshold $\frac{\lambda^{+}(f)}{r}$. The critical value $\frac{\lambda^{+}(f)}{r}$ itself is delicate and remains poorly understood. In contrast, Theorem~\ref{main result} does not suffer from this limitation, as it concerns a formula for the topological entropy itself.
		\item \begin{itemize}
			\item The case of systems with low regularity ($r=1$) is not understood by us.
			\item The condition $h_{\rm top}(f)\geq\frac{\lambda^{+}(f)}{r}$ appears to be critical. There are examples for interval maps (see Kozlovski \cite{Koz98}) with
			$h_{\rm top}(f)<\frac{\lambda^{+}(f)}{r}$, in which the topological entropy vanishes while the volume growth of the tangent cocycle appearing in Theorem~\ref{main result} is positive. Although this does not immediately yield counterexamples in our setting of surface diffeomorphisms, we do not expect a fundamental gap between the situations of interval maps and surface diffeomorphisms \footnote{We also note that Kozlovski \cite[Page 2]{Koz98} mentioned, without providing a precise reference, that Misiurewicz constructed counterexamples for surface diffeomorphisms where the topological entropy is strictly smaller than the volume growth.}. 
			\item Recently, S. Ben Ovadia and Burguet \cite{BeB25} studied the Viana conjecture in arbitrary dimensions using a refined high-dimensional version of Yomdin theory. This work may shed light on possible extensions of Theorem~\ref{main result} to higher dimensions.
		\end{itemize}

	\end{itemize}
	
		We also provide a formula expressing the topological entropy in terms of the \emph{volume growth of the submanifolds} which extends the classical results in Yomdin theory \cite{Yom87} in the $C^{\infty}$ setting.

	The \emph{volume} (i.e., arc length) of a curve $\sigma: [0,1]\to M$ is denoted by ${\rm Vol}\left(\sigma\right)$.
	
	\begin{theoremalph}
		\label{entropy coincides with volume growth of submanifolds}
		For any $C^r$ ($r>1$) diffeomorphism $f$ on a compact surface $M$ with $h_{\rm top}(f)\geq\frac{\lambda^{+}(f)}{r}$, we have
		
		$$\begin{aligned}
			h_{\rm top}(f)
			&=\lim_{n\to+\infty}\frac{1}{n}\log \sup_{\sigma}{\rm Vol}\left(f^{n}\left(\sigma\right)\right)\\
			&=\sup_{\sigma}\limsup_{n\to+\infty}\frac{1}{n}\log{\rm Vol}\left(f^{n}\left(\sigma\right)\right)\\
			&=\sup_{\sigma}\liminf_{n\to+\infty}\frac{1}{n}\log{\rm Vol}\left(f^{n}\left(\sigma\right)\right)
		\end{aligned}$$	where, for any $\varepsilon\in (0,1)$, the supremum can be taken over the family of all $C^{r}$ embedded curves  $\sigma: [0,1]\to M$ with $\left\|\sigma\right\|_{C^{r}}\leq 1$ and $\left\|d_{t}\sigma\right\| \geq \varepsilon$ for all $t\in[0,1]$ \footnote{These conditions ensure that the curves are neither degenerate nor highly oscillatory and are imposed only for technical convenience. One may allow more general families. For example, we believe that the single condition $\left\|\sigma\right\|_{C^{r}}\leq 1$ is sufficient. This is indeed the case for the second and third equalities in Theorem \ref{entropy coincides with volume growth of submanifolds}, which can be deduced from the present theorem.}.
	\end{theoremalph}

	We make some comments.
	\begin{itemize}
		\item In general, beyond our $C^{r}$ setting together with the condition $h_{\rm top}(f)\geq \frac{\lambda^{+}(f)}{r}$, it is quite subtle to establish an equivalence between the two kinds of volume growth of the submanifolds in Theorem \ref{entropy coincides with volume growth of submanifolds} : $$\limsup_{n\to+\infty}\frac{1}{n}\log \sup_{\sigma}{\rm Vol}\left(f^{n}\left(\sigma\right)\right)\stackrel{?}{ = }\sup_{\sigma}\limsup_{n\to+\infty}\frac{1}{n}\log{\rm Vol}\left(f^{n}\left(\sigma\right)\right).$$ This question was implicitly raised in an unpublished manuscript of Buzzi \cite[Remark 1.2]{Buz96}.
		\item If the topological entropy is \emph{strictly} larger than $\frac{\lambda^{+}(f)}{r}$, then there exists an invariant measure $\mu$ of maximal entropy (\cite{Bur24,BCS25}). Moreover, by a result of Cogswell \cite{Cog00} (see also Zang \cite{Zan22}), the supremum in Theorem~\ref{entropy coincides with volume growth of submanifolds} is attained by the local unstable manifold $W^{u}_{\rm loc}(x)$ for $\mu$-almost every point $x$. In fact, Corollary~\ref{volume growth w.r.t. any measure with large entropy} shows that the supremum is attained for any ergodic measure $\mu$ with $h(f,\mu)>\frac{\lambda^{+}(f)}{r}$, and not only for measures of maximal entropy as one might expect. 
		\item It is worth noting that Theorem~\ref{entropy coincides with volume growth of submanifolds} actually implies Theorem~\ref{main result}. Indeed, by a Fubini-type argument (see the proof of Theorem~\ref{main theorem of decomposition of volume growth of derivatives} in Section~\ref{Upper bound of the volume growth rate}), there exists a constant $K>0$ such that for any $n\ge 1$, $$\int_{M}\left\|Df^{n}_{x}\right\|\,dx\leq K \cdot \sup_{\sigma}{\rm Vol}(f^n\sigma).$$ As a consequence, by Theorem \ref{entropy coincides with volume growth of submanifolds}, we have
		
		$$\limsup_{n\to+\infty}\frac{1}{n}\log\int_{M}\left\|Df^{n}_{x}\right\|\,dx\leq\limsup_{n\to+\infty}\frac{1}{n}\log \sup_{\sigma}{\rm Vol}(f^n\sigma)=h_{\rm top}(f).$$
	
		The reverse inequality ($\geq$) follows from pure Pesin theory (see Theorem~\ref{volume of tangent cocycles bounds entropy}).
		
		However, we will directly prove Theorem~\ref{main result} using a more fundamental maximum bound, namely Theorem~\ref{main theorem of decomposition of volume growth of derivatives}. This approach offers a deeper insight into the mechanism governing the volume growth of the tangent cocycle. But there is no essential difference between these two approaches, since both ultimately follow from Theorem~\ref{main theorem of decomposition of volume growth of submanifolds}.
		
	\end{itemize}

	\subsection{Background and applications}\label{Backgrounds and applications}
	\medskip
	The study of the volume growth of the tangent cocycle:
	$$\frac{1}{n}\log\int\left\|(Df^{n}_{x})^{\wedge}\right\|\,d x$$ appears to go back to Sacksteder-Shub \cite{SaS78}. Here $(Df^{n}_{x})^{\wedge}$ denotes the linear map induced by $Df^{n}_{x}$ on the exterior algebras of the tangent spaces $T_{x}M$ and $T_{f^{n}x}M$, and $\left\|\cdot\right\|$ is the operator norm induced by the Riemannian metric.

Based on Pesin theory, an inequality was obtained by Przytycki \cite{Prz80} for $C^r$ ($r>1$) diffeomorphisms and by Newhouse \cite{New88} for $C^r$ ($r>1$) maps on a compact manifold $M$ of arbitrary dimension:
\begin{equation}\label{Przytycki upper bound of entropy}
	\liminf_{n\to+\infty}\footnote{Przytycki \cite{Prz80} established the inequality in terms of a $\limsup$ and remarked at the end of the paper that it remains valid with $\liminf$.}\frac{1}{n}\log\int\left\|(Df^{n}_{x})^{\wedge}\right\|\,d x\geq h_{\rm top}(f).
\end{equation}
Subsequently, Kozlovski \cite{Koz98} (see also Klapper-Young \cite{KlY95}) showed that, for $C^\infty$ diffeomorphisms, by Yomdin theory \cite{Yom87}, a lower bound for the topological entropy is provided, i.e.,
\begin{equation}\label{Kozlovski lower bound of entropy}
	\limsup_{n\to+\infty}\frac{1}{n}\log\int\left\|(Df^{n}_{x})^{\wedge}\right\|\,d x\leq h_{\rm top}(f).
\end{equation}
 Combining (\ref{Przytycki upper bound of entropy}) and (\ref{Kozlovski lower bound of entropy}), we obtain an equality for the topological entropy. One might wonder what the $C^{r}$ version of the above formula is for finite $r$. Indeed, for $C^{r}$ surface diffeomorphisms, one may expect an inequality that combines the topological entropy with the Yomdin term:  \begin{equation}\label{Kozlovski lower bound of entropy with Yomdin term}
 	\limsup_{n\to+\infty}\frac{1}{n}\log\int\left\|(Df^{n}_{x})^{\wedge}\right\|\,d x\leq\footnote{Although this formula is not stated explicitly in Kozlovski \cite{Koz98} or elsewhere (as far as we know), it can be deduced from the arguments in \cite{Koz98} as an application of Yomdin theory in the $C^{r}$ case.} h_{\rm top}(f)+\frac{\lambda^{+}(f)}{r}.
 \end{equation} In the $C^{\infty}$ case, the Yomdin term $\frac{\lambda^{+}(f)}{r}$ vanishes, which yields the inequality (\ref{Kozlovski lower bound of entropy}). In contrast, for finite $r$, the inequality (\ref{Kozlovski lower bound of entropy with Yomdin term}) cannot be reduced to inequality (\ref{Kozlovski lower bound of entropy}), as the Yomdin term does not vanish in general.  Rather than the \emph{additive} bound of the topological entropy and the Yomdin term appearing in (\ref{Kozlovski lower bound of entropy with Yomdin term}), we obtain a \emph{maximum} bound that yields a substantial improvement (see Theorem~\ref{main theorem of decomposition of volume growth of derivatives}):

 \medskip	\begin{equation}\label{Our lower bound of entropy convex with Yomdin term}
 \limsup_{n\to+\infty}\frac{1}{n}\log\int_{M}\left\|Df^{n}_{x}\right\|\,dx\leq \max \left\{h_{\rm top}(f),\,\, \frac{\lambda^{+}(f)}{r}\right\}.
 \end{equation}\medskip

 If $h_{\rm top}(f)\geq\frac{\lambda^{+}(f)}{r}$, inequality (\ref{Our lower bound of entropy convex with Yomdin term}) immediately implies inequality (\ref{Kozlovski lower bound of entropy}). Together with inequality (\ref{Przytycki upper bound of entropy}) which holds for any $C^r$ ($r>1$) diffeomorphism, this yields Theorem~\ref{main result}.

There are also related results in the partially hyperbolic setting (see the works of Saghin \cite{Sag14}, Yang-Zang \cite{YaZ23} and Guo-Liao-Sun-Yang \cite{GLS18}), as well as in the setting of random dynamical systems (see Ma \cite{MaX20}, Ma-Zhang \cite{MaZ23}).

\medskip
\medskip

	Next, we present two applications (Corollary \ref{order of log and integral}, Corollary \ref{exponent less than topological entropy}) of Theorem \ref{main result}, as well as two applications (Corollary \ref{volume growth w.r.t. any measure with large entropy}, Corollary \ref{full entropy implies full volume growth}) of Theorem \ref{entropy coincides with volume growth of submanifolds}.
 
	The order of the operations $\log$ and $\int$ in Theorem \ref{main result} is crucial in the following sense.
	\begin{Coro}\label{order of log and integral}
		For any $C^r$ ($r>1$) diffeomorphism $f$ on a compact surface $M$ with $h_{\rm top}(f)>\frac{\lambda^{+}(f)}{r}$ \footnote{Note that here we have to assume a strict inequality. Otherwise, Burguet's result \cite{Bur24'} does not guarantee the existence of an SRB measure.}, the following two properties are equivalent: 
		
		\begin{enumerate}[label=(\Roman*)]
			\item $$\limsup_{n\to+\infty}\frac{1}{n}\int_{M} \log\left\|Df^{n}_{x}\right\|\,dx=\limsup_{n\to+\infty}\frac{1}{n}\log\int_{M} \left\|Df^{n}_{x}\right\|\,dx. $$
			\item there exist finitely many ergodic SRB measures, each of which is a measure of maximal entropy, and the union of their basins \footnote{The basin of an invariant measure $\mu$ is defined by $B(\mu):=\left\{x\in M:\frac{1}{n}\sum_{j=0}^{n-1}\delta_{f^j(x)}\xrightarrow{\text{weak $\ast$}}\mu\right\}$.} has full Lebesgue measure.
		\end{enumerate}
	\end{Coro}
	\begin{proof}
		\textbf{Part:} (I) $\Longrightarrow$ (II).
		
		By Fatou Lemma (see Lemma \ref{Fatou Lemma}) and Theorem \ref{main result},

			$$\begin{aligned}
		\int_{M} \limsup_{n\to+\infty}\frac{1}{n}\log\left\|Df^{n}_{x}\right\|\,dx
			&\geq \limsup_{n\to+\infty}\frac{1}{n}\int_{M} \log\left\|Df^{n}_{x}\right\|\,dx\\
			&=\limsup_{n\to+\infty}\frac{1}{n}\log\int_{M} \left\|Df^{n}_{x}\right\|\,dx\\
			&=\footnotemark h_{\rm top}(f)\\
			&>\frac{\lambda^{+}(f)}{r}.
		\end{aligned}$$\footnotetext{Indeed, the argument here only requires one direction (see Theorem \ref{volume of tangent cocycles bounds entropy}) of the equality in Theorem \ref{main result}.}
 We claim that for Lebesgue almost every $x$, \begin{equation}\label{leb almost every point has exponent equal to topo entropy}
		\limsup_{n\to+\infty}\frac{1}{n}\log\left\|Df^{n}_{x}\right\|= h_{\rm top}(f).
		\end{equation} Because otherwise, since the integral of the left side is no less than $h_{\rm top}(f)$, there must exist some set $A$ with positive Lebesgue measure such that $$\limsup_{n\to+\infty}\frac{1}{n}\log\left\|Df^{n}_{x}\right\| > h_{\rm top}(f)>\frac{\lambda^{+}(f)}{r},\quad\forall x\in A.$$ By Burguet's result \cite{Bur24'}, there is an ergodic SRB measure $\mu$ whose positive Lyapunov exponent is strictly larger than $h_{\rm top}(f)$. This is a contradiction since the positive Lyapunov exponent of an ergodic SRB measure coincides with its entropy.

		Again, by Burguet's result \cite{Bur24'}, under condition (\ref{leb almost every point has exponent equal to topo entropy}), there are at most countably many ergodic SRB measures $\{\mu_{i}\}_{i\in \mathbb{N}}$ such that Lebesgue almost every $x$ lies in the basin of some $\mu_{i}$ and each $\mu_{i}$ carries maximal entropy.  Indeed, in our setting there are only finitely many such ergodic SRB measures, as a direct consequence of the main results of Buzzi-Crovisier-Sarig \cite{BCS22'}, which assert that there are at most finitely many ergodic measures of maximal entropy.
		
		\textbf{Part:} (II) $\Longrightarrow$ (I).
		
	The direction $\cdots\leq \cdots$ in (I) follows directly from Jensen’s inequality. We next prove the reverse inequality.
		
		By Burguet's result \cite[Main Theorem]{Bur24'}, since the union of the basins has full Lebesgue measure,  for Lebesgue almost every $x$, 
		\begin{equation}
			\lim_{n\to+\infty}\frac{1}{n}\log\left\|Df^{n}_{x}\right\|= h_{\rm top}(f) \footnote{Burguet’s Main Theorem \cite[Main Theorem]{Bur24'} establishes this equality with "$\limsup$" in place of "$\lim$". However, the proof of \cite[Lemma 13]{Bur24'} yields a slightly stronger conclusion: for Lebesgue almost every $x$ satisfying $$\limsup_{n\to+\infty}\frac{1}{n}\log\left\|Df^{n}_{x}\right\|> \frac{\lambda^{+}(f)}{r},$$ the limit actually exists.}. 
		\end{equation} Then, by Fatou Lemma (see Lemma \ref{Fatou Lemma}), 	$$h_{\rm top}(f)=\int_{M} \liminf_{n\to+\infty}\frac{1}{n}\log\left\|Df^{n}_{x}\right\|dx
		\leq \liminf_{n\to+\infty}\frac{1}{n}\int_{M} \log\left\|Df^{n}_{x}\right\|dx\leq \limsup_{n\to+\infty}\frac{1}{n}\int_{M} \log\left\|Df^{n}_{x}\right\|dx.$$ By Theorem \ref{main result}, we have the reverse inequality.		
	\end{proof}

	The proof of Corollary \ref{order of log and integral} also shows that equality (I) therein is equivalent to $$\lim_{n\to+\infty}\frac{1}{n}\int_{M} \log\left\|Df^{n}_{x}\right\|\,dx=\lim_{n\to+\infty}\frac{1}{n}\log\int_{M} \left\|Df^{n}_{x}\right\|\,dx. $$

	\medskip
	\medskip
	
	The following result, while not new \footnote{Although this result does not appear in this exact form, Corollary \ref{exponent less than topological entropy} can also be deduced from Burguet's result \cite{Bur24'} by considering SRB measures. Here we give a technically lighter proof using Theorem \ref{main result}.}, is another interesting consequence of Theorem~\ref{main result}. Denote the usual \emph{maximal Lyapunov exponent} of $x$ by $$\lambda^{+}(f,x):=\limsup_{n\to+\infty}\frac{1}{n}\log\left\|Df^{n}_{x}\right\| .$$
	\begin{Coro}\label{exponent less than topological entropy}
		For any $C^r$ ($r>1$) diffeomorphism $f$ on a compact surface $M$ with $h_{\rm top}(f)\geq\frac{\lambda^{+}(f)}{r}$, $$\lambda^{+}(f,x)\leq h_{\rm top}(f)$$ for Lebesgue almost every point $x$. 
	\end{Coro}
	\begin{proof} By a general result in measure theory (see Lemma \ref{exponent less than topological entropy abstract version}) and Theorem \ref{main result}, $$\lambda^{+}(f,x)\leq \lim_{n\to+\infty}\frac{1}{n}\log\int_{M}\left\|Df^{n}_{x}\right\|\,dx=h_{\rm top}(f)$$ for Lebesgue almost every point $x$.	
	\end{proof}
	\medskip
	\medskip
	The following Corollary \ref{volume growth w.r.t. any measure with large entropy} is an application of Theorem~\ref{entropy coincides with volume growth of submanifolds}, showing that the maximal volume growth is attained w.r.t. any ergodic measure $\mu$ satisfying $h(f,\mu)>\frac{\lambda^{+}(f)}{r}$, rather than being restricted to measures of maximal entropy. This conclusion is somewhat unexpected. For instance, consider an SRB measure $\mu$: although its unstable manifolds carry $\mu$-typical points with full induced Lebesgue measure, their volume growth does not converge to the positive Lyapunov exponent (which coincides with the entropy of $\mu$), but instead to the topological entropy.
	
	It is instructive to compare this phenomenon with the so-called \emph{Almost Anosov} systems constructed by Hu-Young \cite{HuY95}, where Lebesgue almost every point has zero Lyapunov exponents, yet a measure of maximal entropy still exists. In particular, a curve along which Lebesgue almost every point has zero Lyapunov exponents may nevertheless exhibit maximal volume growth, equal to the topological entropy. On the other hand, it is not hard to see: a curve along which Lebesgue almost every point has positive Lyapunov exponents cannot have zero volume growth.
	
	The proof of Corollary~\ref{volume growth w.r.t. any measure with large entropy} relies on background from homoclinic relations (see \cite{BCS22'}) and Pesin theory (see \cite{BPe07}). Here we only provide an outline of the argument, omitting the technical details.
	
		\begin{Coro}\label{volume growth w.r.t. any measure with large entropy}
		For any $C^r$ ($r>1$) transitive diffeomorphism $f$ on a compact surface $M$ and any ergodic measure $\mu$ with $h(f,\mu)>\frac{\lambda^{+}(f)}{r}$, we have for $\mu$-almost every $x$, $$h_{\rm top}(f)=\lim_{n\to+\infty}\frac{1}{n}\log {\rm Vol}\left(f^{n}\left(W^{u}_{\rm loc}(x)\right)\right)$$ where $W^{u}_{\rm loc}(x)$ denotes the local unstable manifold at $x$.
	\end{Coro}
	\begin{proof}[Outline of the proof]
		Since $f$ is transitive, by \cite[Theorem 6.12]{BCS22'}, there is a unique homoclinic class that supports all ergodic measures $\mu$ with $h(f,\mu)>\frac{\lambda^{+}(f)}{r}$. Moreover, this homoclinic class  supports a unique ergodic measure $\nu$ of maximal entropy \footnote{The existence of measures of maximal entropy for the $C^{r}$ case is established by Burguet \cite[Corollary 1]{Bur24} and the uniqueness is established by Buzzi-Crovisier-Sarig \cite[Main Theorem Revisited]{BCS22'}.}. Hence, $\mu$ must be homoclinically related to $\nu$, i.e., there are two subsets $R_{\mu}, R_{\nu}$ with $\mu(R_{\mu}), \nu(R_{\nu})>0$ such that for any $x\in R_{\mu}, y\in R_{\nu}$, the unstable manifold $W^{u}(x)$ intersects the stable manifold $W^{s}(y)$ transversely at some points, and likewise $W^{s}(x)$ intersects $W^{u}(y)$ transversely at some points. Note that for $\nu$-almost every $y\in R_{\nu}$, $$\nu_{y}\left(R_{\nu}\right)>0$$ where $\nu_{y}$ denotes the conditional measure \footnote{In fact, conditional measures are defined with respect to \emph{subordinate partitions} and we disregard this point here. See Ledrappier-Young \cite{LeY85} for further details.} of $\nu$ on $W^{u}_{\rm loc}(y)$. We fix any such $y\in R_{\nu}$. Let $z$ be a transverse point of $W^{u}(x)$ and  $W^{s}(y)$. Let $m_{1}\in\mathbb{N}$ be large enough such that $$z\in f^{m_{1}}\left(W^{u}_{\rm loc}\left(f^{-m_{1}}\left(x\right)\right)\right).$$ By Inclination Lemma \footnote{The Inclination Lemma \cite[Lemma 2.7]{BCS22'} requires $y$ to be \emph{recurrent}. As $\nu$-almost every point in $R_{\nu}$ is recurrent (see \cite[Section 2.3(e),(f)]{BCS22'}), we may assume without loss of generality that $y$ is recurrent.}, there are $m_{2}\in\mathbb{N}$ and some curve $D$ with $$z\in f^{-m_{2}}\left(D\right)\subset f^{m_{1}}\left(W^{u}_{\rm loc}\left(f^{-m_{1}}\left(x\right)\right)\right)$$ such that $D$ is sufficiently $C^{1}$-close to $W^{u}_{\rm loc}(y)$. Consequently, the stable holonomy map $\pi_{s}:\Lambda\to D$ is well defined on some Pesin set \footnote{Here $\Lambda$ refers to the standard Pesin set, where points admit stable and unstable manifolds with uniform size, and the angle between the stable and unstable directions is uniformly bounded away from zero. See \cite{BPe07} for details on Pesin theory.} $\Lambda\subset W^{u}_{\rm loc}(y)$ with $\nu_{y}(\Lambda)>0$. See Figure \ref{holonomyonD}.
		
			\begin{figure}[htbp]
			\centering
			\includegraphics[width=0.85\textwidth]{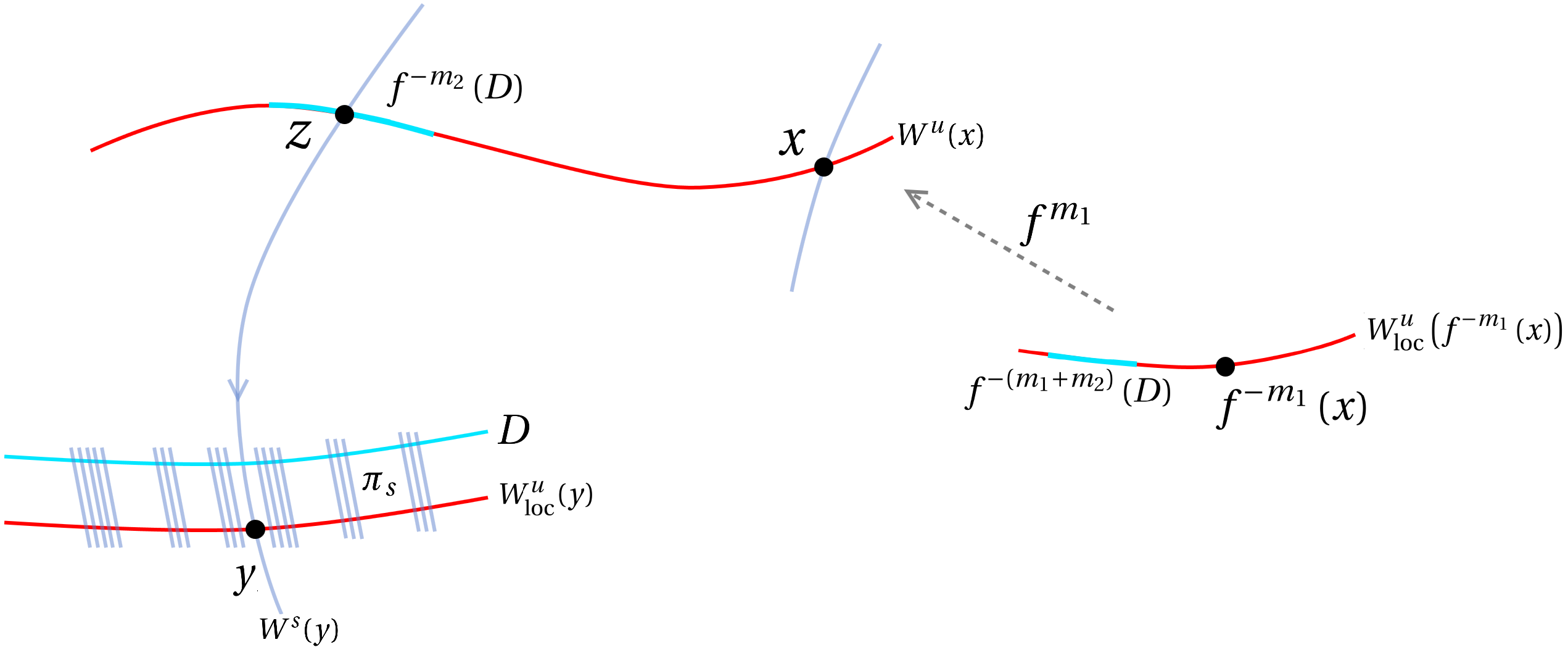}
			\caption{The stable holonomy map $\pi_{s}$}
			\label{holonomyonD}
		\end{figure}
		
		Let us first briefly recall the argument in \cite{Zan22} (see also \cite{Cog00}, \cite{New88} for similar ideas) establishing 
		\begin{equation}\label{result of zang}
			h(f,\nu)\leq \liminf_{n\to+\infty}\frac{1}{n}\log {\rm Vol}\left(f^{n}\left(W^{u}_{\rm loc}(y)\right)\right).
		\end{equation} The fact that $\nu_{y}(\Lambda)>0$ is what ensures sufficient volume growth. Given any small $\varepsilon>0$ and any large $n\in\mathbb{N}$, since $\nu_{y}\left(\Lambda\right)>0$, by a fiber-version of Katok's entropy formula (see \cite[Theorem A (1)]{Zan22}), there is a maximal $(n,\varepsilon)$-separated subset $\Omega_{n,\varepsilon}$ of $\Lambda$ with $$\#\Omega_{n,\varepsilon}\gtrsim e^{n\cdot \left(h(f,\nu)-\varepsilon\right)}$$ such that for any $w\in\Omega_{n,\varepsilon}$, $${\rm Vol}\left(f^{n}\left(B^{u}(w,n,\varepsilon)\right)\right)\gtrsim e^{-n\varepsilon}$$ where $B^{u}(w,n,\varepsilon)$ denotes the unstable dynamical ball, $$B^{u}(w,n,\varepsilon):=\left\{w'\in W^{u}_{\rm loc}(w):\ d\left(f^{k}(w'),f^{k}(w)\right)<\varepsilon \ \text{ for all } 0\leq k<n\right\}.$$ As a consequence, $${\rm Vol}\left(f^{n}\left(W^{u}_{\rm loc}(y)\right)\right)\gtrsim \#\Omega_{n,\varepsilon}\cdot \min_{w\in\Omega_{n,\varepsilon}} {\rm Vol}\left(f^{n}\left(B^{u}(w,n,\varepsilon)\right)\right)\gtrsim e^{n\cdot \left(h(f,\nu)-2\varepsilon\right)}$$ which proves (\ref{result of zang}) by the arbitrariness of $\varepsilon$. The key observation is that the same argument applies with $D$ in place of $W^{u}_{\rm loc}(y)$. Indeed, since $D$ is connected to $W^{u}_{\rm loc}(y)$ through the stable holonomy on $\Lambda$, it follows that, for any large $n$ and any $w\in\Lambda$, the point $f^{n}\left(\pi_{s}(w)\right)$ remains inside the Pesin chart centered at $f^{n}(w)$. As a consequence, $${\rm Vol}\left(f^{n}\left(B^{u}\left(\pi_{s}(w),n,\varepsilon\right)\right)\right)\approx{\rm Vol}\left(f^{n}\left(B^{u}(w,n,\varepsilon)\right)\right)\gtrsim e^{-n\varepsilon}.$$ Hence, $${\rm Vol}\left(f^{n}\left(D\right)\right)\gtrsim \#\Omega_{n,\varepsilon}\cdot \min_{w\in\Omega_{n,\varepsilon}} {\rm Vol}\left(f^{n}\left(B^{u}\left(\pi_{s}(w),n,\varepsilon\right)\right)\right)\gtrsim e^{n\cdot \left(h(f,\nu)-2\varepsilon\right)}.$$ This shows that for any $x\in R_{\mu}$, $$h_{\rm top}(f)=h(f,\nu)\leq\liminf_{n\to+\infty}\frac{1}{n}\log {\rm Vol}\left(f^{n}\left(D\right)\right)\leq\liminf_{n\to+\infty}\frac{1}{n}\log {\rm Vol}\left(f^{n}\left(W^{u}_{\rm loc}\left(f^{-m_{1}}\left(x\right)\right)\right)\right).$$ By Theorem \ref{entropy coincides with volume growth of submanifolds} \footnote{Any local unstable manifold with small size is contained in a curve from the family in Theorem~\ref{entropy coincides with volume growth of submanifolds}, as local unstable manifolds are nearly straight at small scales.}, $$h_{\rm top}(f)=\lim_{n\to+\infty}\frac{1}{n}\log {\rm Vol}\left(f^{n}\left(W^{u}_{\rm loc}\left(f^{-m_{1}}\left(x\right)\right)\right)\right).$$ By ergodicity and the backward contraction of the local unstable manifolds, this equality above in fact holds for $\mu$-almost every $x$ without passing to backward iterates (see Lemma~\ref{simple property of ergodicity}): $$h_{\rm top}(f)=\lim_{n\to+\infty}\frac{1}{n}\log {\rm Vol}\left(f^{n}\left(W^{u}_{\rm loc}(x)\right)\right).$$
		
	\end{proof}
	
The proof of Corollary~\ref{volume growth w.r.t. any measure with large entropy} in fact yields the following more general statement: if two hyperbolic ergodic measures $\mu$ and $\nu$ are homoclinically related, then for $\mu$-almost every point $x$, $$h(f,\nu)\leq\liminf_{n\to+\infty}\frac{1}{n}\log {\rm Vol}\left(f^{n}\left(W^{u}_{\rm loc}\left(x\right)\right)\right).$$ As a consequence of the variational principle, for any hyperbolic ergodic measure $\mu$ supported on a homoclinic class $H(\mathcal{O})$ and for $\mu$-almost every $x$, $$h_{\rm top}\left(f, H(\mathcal{O})\right)\leq\liminf_{n\to+\infty}\frac{1}{n}\log {\rm Vol}\left(f^{n}\left(W^{u}_{\rm loc}(x)\right)\right)$$ where $h_{\rm top}\left(f, H(\mathcal{O})\right)$ denotes the topological entropy of the homoclinic class.

	\medskip
	\medskip
	
	In general, the topological entropy of a curve is bounded above by its volume growth (see Lemma \ref{volume growth geq entropy}), and the inequality can be strict. Indeed, topological entropy only reflects the number of dynamical balls, capturing the global complexity, whereas volume growth depends not only on this number but also on the geometric distortion within each dynamical ball, such as folding and bending, which represent local complexity. However, the following corollary shows that when a curve carries full topological entropy, its volume growth attains the maximal rate which coincides with the topological entropy. In this case, the contribution of local complexity inside dynamical balls becomes negligible.
	
	Write $\log^{+} x := \max\{\log x,\, 0\}$.
	\begin{Coro}\label{full entropy implies full volume growth}
		For any $C^r$ ($r>1$) diffeomorphism $f$ on a compact surface $M$ with $h_{\rm top}(f)\geq\frac{\lambda^{+}(f)}{r}$ and any $C^{r}$ embedded curve $\sigma: [0,1]\to M$ with $h_{\rm top}\left(f,\sigma\right)\footnote{$h_{\rm top}(f,\Omega)$ denotes the topological entropy of $f$ restricted to a general subset $\Omega$, defined in the usual way by $$h_{\rm top}\left(f,\Omega\right):=\lim_{\varepsilon\to 0}\limsup_{n \to +\infty}\frac{1}{n}\log\# S(n,\varepsilon,\Omega)$$ where $S(n,\varepsilon,\Omega)$ denotes a maximal $(n,\varepsilon)$-separated subset of $\Omega$ (see the book \cite{Wal82} for background).}=h_{\rm top}(f)$,  we have $$h_{\rm top}(f)=\limsup_{n\to+\infty}\frac{1}{n}\log^{+} {\rm Vol}\left(f^{n}\left(\sigma\right)\right).$$
	\end{Coro}
\begin{proof}
We may cover $\sigma $ by finitely many $C^r$ embedded curves $\{\sigma_{i}\}$ satisfying
$$\|\sigma_i\|_{C^r}\leq 1\qquad
\text{and}\qquad
\|d_t\sigma_i\|\geq \varepsilon, \,t\in[0,1]
$$ for some sufficiently small $\varepsilon>0$. By Lemma \ref{volume growth geq entropy}, we have $$h_{\rm top}(f)=h_{\rm top}\left(f,\sigma\right)\leq\limsup_{n\to+\infty}\frac{1}{n}\log^{+} {\rm Vol}\left(f^{n}\left(\sigma\right)\right)\leq \max_{i}\limsup_{n\to+\infty}\frac{1}{n}\log^{+} {\rm Vol}\left(f^{n}\left(\sigma_{i}\right)\right).$$ On the other hand, by Theorem \ref{entropy coincides with volume growth of submanifolds}, the right-hand side is bounded above by $h_{\rm top}(f)$. Therefore, equality holds.
\end{proof}

\subsection{Some further questions}
We believe that the ideas of Theorem~\ref{main result} and Theorem~\ref{entropy coincides with volume growth of submanifolds} might extend to other topological invariants (see Llibre-Saghin \cite{LlS09} for a survey of various invariants measuring dynamical complexity), such as the growth rate of hyperbolic periodic points. This has been established by Burguet~\cite{Bur20} in the $C^{\infty}$ setting and it was conjectured by Burguet in the $C^{r}$ case:

\begin{Question}[\cite{Bur20}]
	For any $C^r$ ($r>1$) diffeomorphism $f$ on a compact surface $M$ with $h_{\rm top}(f)>\frac{\lambda^{+}(f)}{r}$ and any $0<\delta< h_{\rm top}(f)$, $$h_{\rm top}(f)=\limsup_{n\to+\infty}\frac{1}{n}\log \#\mathcal{P}^{n}_{\delta}$$ where $\mathcal{P}^{n}_{\delta}$ is the set of $n$-periodic points with Lyapunov
	exponents $\delta$-away from zero.
\end{Question}

We also propose some questions related to the formula in Theorem~\ref{main result}: $$h_{\rm top}(f)=\lim_{n\to+\infty}a_{n}:=\frac{1}{n}\log\int_{M}\left\|Df^{n}_{x}\right\|\,dx.$$

\begin{Question}
	In general, does $a_{n}$ always converge ? If so, what is its limit if it is not the topological entropy ?
\end{Question}

\begin{Remark*}
	Note that $a_{n}$ might NOT be a sub-additive sequence, even when $f$ preserves the volume (i.e., when the Lebesgue measure is an invariant measure).
    
    We believe that a deeper formula underlies Theorem~\ref{main result}.
\end{Remark*}

\vspace{3mm}

\begin{Question}
Is it reasonable to ask about the rate of convergence in Theorem~\ref{main result} ?
\end{Question}

\begin{Remark*}
Note that the definition of $a_n$ already contains the factor "$\frac{1}{n}\log$". It might therefore be too ambitious to expect a specific rate of convergence. However, we do not know this convincingly.  We also note that Buzzi-Crovisier-Sarig \cite{BCS25} have recently shown that measures of maximal entropy are exponentially mixing.
\end{Remark*}

\vspace{3mm}

Denote by ${\rm Diff}^{r}(M)$ the space of all $C^{r}$ diffeomorphisms on $M$ equipped with the $C^{r}$ topology.  Define the functions $\varphi_{n}, \varphi: {\rm Diff}^{r}(M)\to \mathbb{R}$: $$\varphi_{n}\left(g\right):=\frac{1}{n}\log\int_{M}\left\|Dg^{n}_{x}\right\|\,dx,\qquad \varphi:=\limsup_{n\to+\infty} \varphi_{n}.$$
\begin{Question}
	Is $\varphi$ upper semi-continuous ?
\end{Question}
\begin{Remark*}
		Burguet \cite[Corollary 2]{Bur24} shows that the topological entropy function  $h_{\rm top}(\cdot): {\rm Diff}^{r}(M)\to \mathbb{R}$ is upper semi-continuous \footnote{$h_{\rm top}(\cdot): {\rm Diff}^{r}(M)\to \mathbb{R}$ is in fact continuous. Katok's horseshoe theorem \cite{Kat80} implies that $h_{\rm top}(\cdot)$ is lower semi-continuous with respect to the $C^{1}$ topology, and together with the above upper semi-continuity, this yields continuity.} at any $f$ with $h_{\rm top}(f)\geq\frac{\lambda^{+}(f)}{r}$. By Theorem~\ref{main result}, this also proves, in a rather indirect way, that $\varphi$ is upper semi-continuous at any $f$ with $h_{\rm top}(f)\geq \frac{\lambda^{+}(f)}{r}$. We wonder whether $\varphi$ is upper semi-continuous in general, and whether this can be established by a direct argument.
\end{Remark*}

\vspace{3mm}

\begin{Question}
Does an analogue of Theorem~\ref{main result} hold for interval maps ?
\end{Question}

\begin{Remark*}
This is known to hold for $C^{\infty}$ interval maps by the work of Kozlovski \cite{Koz98}. It remains unclear whether new obstructions arise when extending it to $C^{r}$ interval maps. In this paper, we only consider diffeomorphisms.

It is also worth noting that there have been several recent works on the continuity of Lyapunov exponents and the existence of SRB measures for interval maps (see, for example, Delplanque \cite{Del24}, Delplanque-Li \cite{DeL25}, and Li \cite{LiH24}).
\end{Remark*}

	\section{Maximum bounds for volume growth}
	In this section, we present the main results used to prove Theorem~\ref{main result} and Theorem~\ref{entropy coincides with volume growth of submanifolds}.

\subsection{Mechanism behind the maximum bounds for volume growth}

	\begin{Theorem}\label{main theorem of decomposition of volume growth of submanifolds}
		Let $f$ be a $C^r$ ($r>1$) diffeomorphism on a compact surface $M$. We have $$\limsup_{n\to+\infty}\frac{1}{n}\log \sup_{\sigma}{\rm Vol}(f^n\sigma)\leq\max \left\{h_{\rm top}(f),\,\, \frac{\lambda^{+}(f)}{r}\right\}$$ where, for any $\varepsilon\in (0,1)$, the supremum can be taken over the family of all $C^{r}$ embedded curves  $\sigma: [0,1]\to M$ with $\left\|\sigma\right\|_{C^{r}}\leq 1$ and $\left\|d_{t}\sigma\right\| \geq \varepsilon$ for all $t\in[0,1]$.
	\end{Theorem}
\begin{Remark*} This \emph{maximum bound} improves the classical estimate in Yomdin theory \cite{Yom87}, where only the following \emph{sum bound} is obtained:
$$\limsup_{n\to+\infty}\frac{1}{n}\log \sup_{\sigma}{\rm Vol}(f^n\sigma)\leq h_{\rm top}(f)+\frac{\lambda^{+}(f)}{r}.$$
\end{Remark*}	
	\begin{Theorem}\label{main theorem of decomposition of volume growth of derivatives}
		Let $f$ be a $C^r$ ($r>1$) diffeomorphism on a compact surface $M$. We have $$\limsup_{n\to+\infty}\frac{1}{n}\log\int_{M}\left\|Df^{n}_{x}\right\|\,dx\leq\max \left\{h_{\rm top}(f),\,\, \frac{\lambda^{+}(f)}{r}\right\}.$$
	\end{Theorem}

The inequalities in Theorems~\ref{main theorem of decomposition of volume growth of submanifolds} and \ref{main theorem of decomposition of volume growth of derivatives} can be equivalently expressed in a convex combination form: there is some $\alpha\in[0,1]$ such that \begin{equation}\label{convex form}
	\begin{aligned}
	\limsup_{n\to+\infty}\frac{1}{n}\log \sup_{\sigma}{\rm Vol}(f^n\sigma)\leq \alpha\cdot h_{\rm top}(f) + (1-\alpha)\cdot\frac{\lambda^{+}(f)}{r},\\
		\limsup_{n\to+\infty}\frac{1}{n}\log\int_{M}\left\|Df^{n}_{x}\right\|\,dx\leq \alpha\cdot h_{\rm top}(f) + (1-\alpha)\cdot\frac{\lambda^{+}(f)}{r}.
	\end{aligned} 
\end{equation}

The reason for presenting this convex-combination form is that our proof actually yields a more refined upper bound, in which the coefficient $\alpha\in[0,1]$ reflects how the volume growth is distributed between the entropic contribution and the Lyapunov-type contribution (the Yomdin term). In particular, the argument implicitly determines an admissible choice of $\alpha$ and thus contains strictly more information than the crude bound given by the maximum. However, making this choice explicit would require introducing a substantial amount of additional technical machinery. For this reason, we do not pursue an explicit description of $\alpha$ here in the statements. As a consequence, the statements of the theorems may be equivalently weakened to the simpler form involving $\max\{h_{\rm top}(f),\lambda^{+}(f)/r\}$.

	These convex-type upper bounds originate from the idea that dynamical invariants such as volume growth measure complexity from two different perspectives: a global one and a local one. To illustrate this idea, let us consider the volume growth of a curve $\sigma$ and write $${\rm Vol}(f,\sigma):=\limsup_{n\to+\infty}\frac{1}{n}\log {\rm Vol}(f^n\sigma).$$

	At time $n$ and at scale $\varepsilon$, the volume of $f^{n}(\sigma)$ can be estimated by decomposing it into two factors:

	$${\rm Vol}(f^n\sigma)\leq \overbrace{\left(\text{number of $(n,\varepsilon)$ dynamical balls to cover $\sigma$}\right)}^{\text{Global complexity}}\times \overbrace{\max_{x}{\rm Vol}\left(f^n\left(\sigma\cap \left(B\left(x,n,\varepsilon\right)\right)\right)\right)}^{\text{Local complexity}}$$ where $B(x,n,\varepsilon)$ denotes the usual \emph{$(n,\varepsilon)$ dynamical ball} at $x$ defined by $$B(x,n,\varepsilon):=\left\{y\in M:\ d\left(f^{k}(y),f^{k}(x)\right)<\varepsilon \ \text{ for all } 0\leq k<n\right\}.$$

	The global complexity is controlled by the topological entropy, while the classical Yomdin theory \cite{Yom87} shows that the local complexity grows at most at the rate $\lambda^{+}(f)/r$. Combining these two estimates yields the classical Yomdin bound:
	\begin{equation}\label{classical Yomdin bound}
		{\rm Vol}(f,\sigma)\leq h_{\rm top}(f) + \frac{\lambda^{+}(f)}{r}.
	\end{equation}

	However, this estimate is in general quite loose, since it implicitly assumes that the global and the local complexities are simultaneously present at all times.

	The key insight behind the convex bound in (\ref{convex form}) is that these two sources of complexity need not act at the same time. Instead, the time interval $[0,n)$ can be decomposed into two major types of periods:
	\begin{itemize}
		\item Geometric times (with proportion $\alpha$): during which the curve stretches across the manifold in a globally expanding and essentially straight manner;
		\item Neutral times (with proportion $1-\alpha$): during which the curve remains confined inside a single dynamical ball and develops complexity only through local folding.
	\end{itemize}

	\begin{figure}[htbp]
		\centering
		\includegraphics[width=1\textwidth]{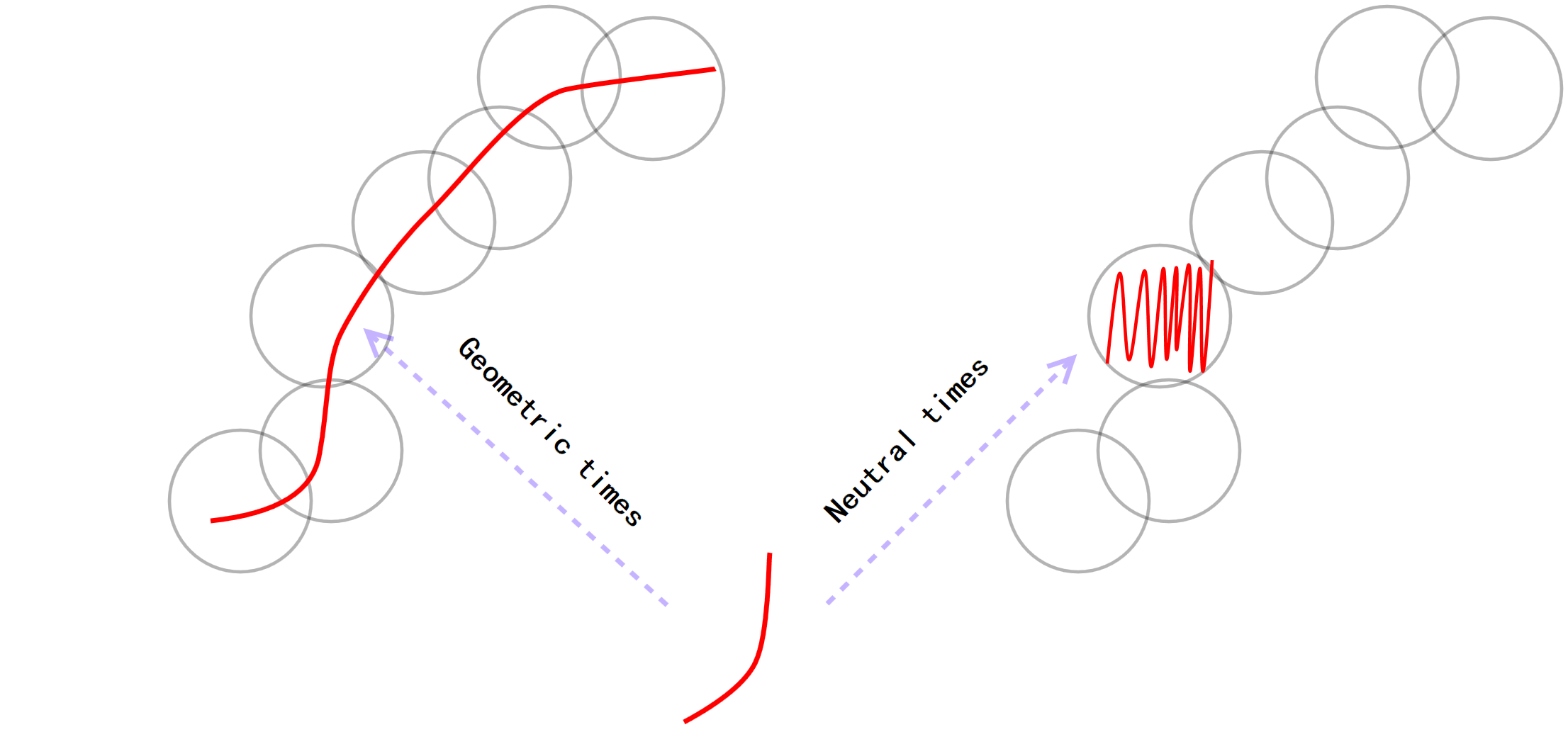}
		\caption{Growth at geometric and neutral times}
		\label{Pic expansion at Geometric and Neutral times}
	\end{figure}

	During geometric times, one only needs to count the number of dynamical balls that the curve visits, which contributes an entropy-type growth, while no significant local complexity is generated inside each dynamical ball.
	In contrast, during neutral times, the curve stays inside a single dynamical ball, so the global complexity vanishes, and the growth is entirely governed by local effects, captured by the Yomdin term $\lambda^{+}(f)/r$. See Figure \ref{Pic expansion at Geometric and Neutral times}.

	This separation of time scales explains the convex structure of the resulting bound and clarifies the mechanism responsible for the improvement over the classical Yomdin estimate. As a consequence, the classical Yomdin bound \eqref{classical Yomdin bound} can be refined to the following inequality:  $${\rm Vol}(f,\sigma)\leq \alpha\cdot h_{\rm top}(f) + (1-\alpha)\cdot\frac{\lambda^{+}(f)}{r}.$$

The separation of time into geometric and neutral components was first introduced by Buzzi-Crovisier-Sarig \cite{BCS22} and was further developed and systematized by Burguet \cite{Bur24}. This idea has since been deepened and applied in several related contexts; see, for instance, \cite{BLY25', BLY25, LuY25}.

We emphasize that all these works are formulated in the presence of a reference invariant measure, and the quantities under consideration are measure-theoretic invariants, such as metric entropy, Lyapunov exponents. In contrast, our approach focuses on purely topological quantities, namely the volume growth of the submanifolds and of the tangent cocycle. The absence of a reference measure leads to several substantial technical difficulties. We mention here two of the main ones.

\begin{itemize}
	\item \emph{The presence of periodic sources \footnote{A periodic point $p$ is called a \emph{periodic source} if both Lyapunov exponents of $p$ are positive.}.}
	When a reference hyperbolic measure $\mu$ is given, the analysis may be restricted to $\mu$-typical points, which are not influenced by periodic sources. In our setting, however, when considering, for example, the volume growth of a curve $\sigma$, a large portion of points in $\sigma$ may spend a long time near periodic sources. During such times, which still correspond to geometric times, the curve needs time to escape from the "trapping region" of these sources. As a consequence, one must carefully control the amount of topological entropy produced during these geometric times, avoiding an unnecessary decomposition of the curve into exponentially many pieces before it has moved away from periodic sources. To address this issue, we introduce a finer decomposition of time and define the notion of \emph{trapping times}, which allows us to prevent an excessive creation of entropy (see Proposition \ref{grow rate for major component in finite step n}).

	\item \emph{Lack of uniform control on the local complexity at geometric times}.
	Without a reference measure, it becomes more delicate to justify that no local complexity is generated during geometric times in a uniform way. We overcome this difficulty by exploiting several partitions of the curve, as introduced in Proposition~\ref{volume growth for time n}, together with an upper semi-continuity property for sub-additive sequences, which yields the required control.
\end{itemize}

	\subsection{Proofs of Theorem \ref{main result} and Theorem \ref{entropy coincides with volume growth of submanifolds}}\label{proof of main theorems}

	We now directly apply the two bounds established above in Theorem \ref{main theorem of decomposition of volume growth of submanifolds} and \ref{main theorem of decomposition of volume growth of derivatives} to prove Theorem \ref{main result} and Theorem \ref{entropy coincides with volume growth of submanifolds}, in order to provide the reader with a clear overview of the global framework.

	\begin{proof}[Proof of Theorem \ref{main result}]

	The inequality
	$$ h_{\rm top}(f)\leq\liminf_{n\to+\infty}\frac{1}{n}\log\int_{M}\left\|Df^{n}_{x}\right\|\,d x$$ is, in a certain sense, already well known (see Przytycki \cite{Prz80}, Kozlovski \cite{Koz98}). Its proof relies only on the $C^r$ ($r>1$) regularity and follows as a consequence of Pesin theory. The only distinction from the classical formulations is that we work directly with the derivative $Df^{n}_{x}$, rather than with the induced map $(Df^{n}_{x})^{\wedge}$ acting on the exterior algebras of the tangent spaces. This difference is natural and expected in our setting, since we restrict ourselves to dimension two. For the reader's convenience and completeness, we reproduce this part of the argument in Theorem~\ref{volume of tangent cocycles bounds entropy}.

	The reverse inequality is provided by Theorem \ref{main theorem of decomposition of volume growth of derivatives}: $$\limsup_{n\to+\infty}\frac{1}{n}\log\int_{M}\left\|Df^{n}_{x}\right\|\,dx\leq \max \left\{h_{\rm top}(f),\,\, \frac{\lambda^{+}(f)}{r}\right\}=h_{\rm top}(f).$$

	The proof of Theorem \ref{main result} is complete.
	\end{proof}

	\begin{proof}[Proof of Theorem \ref{entropy coincides with volume growth of submanifolds}]

		In the following arguments, the notation $\sup_{\sigma}$ refers to the supremum taken over the family of curves in Theorem \ref{entropy coincides with volume growth of submanifolds}. 
		
		For the direction $h_{\rm top}(f)\leq \cdots$, we may assume $h_{\rm top}(f)>0$, since the result is otherwise trivial. By the work of Cogswell \cite{Cog00} (see also Zang \cite{Zan22}), based purely on Pesin theory, for any hyperbolic ergodic measure $\mu$ and for $\mu$-almost every point $x$, $$h(f,\mu)\leq \liminf_{n\to+\infty}\frac{1}{n}\log {\rm Vol}\left(f^{n}\left(W^{u}_{\rm loc}(x)\right)\right)\leq \footnote{A cautious reader might worry that, for a fixed scale $\varepsilon$, the local unstable manifold $W^u_{\rm loc}(x)$ may not be contained in any curve in the family considered in Theorem \ref{entropy coincides with volume growth of submanifolds}, for instance, if it has large curvature. This causes no difficulty, however, since we need not use $W^u_{\rm loc}(x)$ itself to estimate the metric entropy. Instead, we can use a curve $\sigma$ sufficiently close to $W^u_{\rm loc}(x)$ and transverse to the local stable manifold $W^s_{\rm loc}(x)$, so that $\sigma$ and $W^u_{\rm loc}(x)$ are related by local stable holonomy. The volume-growth rate of such a curve $\sigma$ is also bounded below by $h(f,\mu)$. This is precisely the idea used in the proof of Corollary~\ref{volume growth w.r.t. any measure with large entropy}, where the curve $D$ plays the same role.} \sup_{\sigma}\liminf_{n\to+\infty}\frac{1}{n}\log {\rm Vol}\left(f^{n}\left(\sigma\right)\right)$$ where $W^{u}_{\rm loc}(x)$ denotes the local unstable manifold at $x$. Hence, by the variational principle,

		$$h_{\rm top}(f)\leq\sup_{\sigma}\liminf_{n\to+\infty}\frac{1}{n}\log {\rm Vol}\left(f^{n}\left(\sigma\right)\right)\leq\liminf_{n\to+\infty}\frac{1}{n}\log\sup_{\sigma}{\rm Vol}\left(f^n(\sigma)\right).$$

		The reverse inequality is provided by Theorem \ref{main theorem of decomposition of volume growth of submanifolds},

		 $$\begin{aligned}
		\sup_{\sigma}\limsup_{n\to+\infty}\frac{1}{n}\log {\rm Vol}\left(f^{n}\left(\sigma\right)\right)
			&\leq \limsup_{n\to+\infty}\frac{1}{n}\log\sup_{\sigma}{\rm Vol}\left(f^n(\sigma)\right)\\
			&\leq \max \left\{h_{\rm top}(f),\,\, \frac{\lambda^{+}(f)}{r}\right\}\\
			&=h_{\rm top}(f).
		\end{aligned}$$

		The proof of Theorem \ref{entropy coincides with volume growth of submanifolds} is complete.

	\end{proof}

%

	\medskip
	\medskip

	It remains to prove Theorem \ref{main theorem of decomposition of volume growth of submanifolds} (Section \ref{Volume growth of the submanifolds}) and Theorem \ref{main theorem of decomposition of volume growth of derivatives} (Section \ref{Volume growth of the tangent cocycle}).

\section{Preliminaries}
In this section, we collect several notions and basic facts that will be used throughout the paper.
We briefly recall Lyapunov exponents and their associated invariant splittings, and then introduce the induced dynamics on the projective tangent bundle, which provides a convenient framework for describing the measure-theoretic invariants.
\subsection{Lyapunov exponents}
The topological entropy of $f$ is denoted by $h_{\rm top}(f)$ and the metric entropy of an invariant measure $\mu$ is denoted by $h(f, \mu)$. By Oseledets Multiplicative Ergodic Theorem, there is an invariant set $\mathcal{R}$ (called the \emph{Lyapunov regular set}) with total measure (i.e., $\mu(\mathcal{R})=1$ for any invariant measure $\mu$) such that for any $x\in\mathcal{R}$, there are a splitting (called the  \emph{Oseledets splitting}) $T_{x}M=E^{1}\oplus E^{2}\oplus\cdots\oplus E^{l}$ and finitely many numbers (called the \emph{Lyapunov exponents}) $\lambda_{1}(x)>\lambda_{2}(x)>\cdots>\lambda_{l}(x)$ such that for any nonzero vector $v\in E^{j}$, we have $$\lim_{n\to \pm\infty}\frac{1}{n}\log\left\|Df^{n}_{x}(v)\right\|=\lambda_{j}.$$

In our special setting (surface diffeomorphism), given an invariant measure $\mu$, there are at most two well defined Lyapunov exponents for $\mu$-almost every $x$. The larger Lyapunov exponent $\lambda_{1}(x)$ is denoted by $\lambda^{+}(f,x)$ and it can be also defined as $$\lambda^{+}(f,x)=\lim_{n\to+\infty}\frac{1}{n}\log\left\|Df^{n}_{x}\right\|.$$ Similarly, the smaller Lyapunov exponent is denoted by $\lambda^{-}(f,x)$. If $\lambda^{+}(f,x)>\lambda^{-}(f,x)$, we rewrite the Oseledets splitting at $x$ as $$T_{x}M=E^{+}_{x}\oplus E_{x}^{-}.$$ An invariant measure $\mu$ is called \emph{hyperbolic} if for $\mu$-a.e. $x$, one Lyapunov exponent of $x$ is positive and the other is negative. For an ergodic measure $\mu$ with $h(f,\mu)>0$, $\mu$ is hyperbolic by Ruelle's inequality.  We define $$\lambda^{+}(f, \mu):=\int\lambda^{+}(f,x) d\,\mu(x),\quad \lambda^{-}(f, \mu):=\int\lambda^{-}(f,x) d\,\mu(x).$$ If $\mu$ is ergodic, then $\lambda^{\pm}(f, \mu)=\lambda^{\pm}(f,x)$ for $\mu$-a.e. $x$.

\subsection{Tangent dynamics}\label{tangent dynamics}
With the Riemannian structure inherited from the tangent bundle $TM$, we denote the projective tangent bundle of $M$ by  $$\hat{M}:=\left\{(x, [v]):~x\in M, \,[v] \text{ is the linear subspace of $T_{x}M$ generated by the unit vector $v$}\right\} .$$ Let $\pi:\hat{M}\to M$ be the natural projection. Let $\hat{f}: \hat{M}\to \hat{M}$ be the induced map of $f$ (also called the \emph{canonical lift}) defined by $$\hat{f}(x, [v]):=\left(f(x), [Df_{x}(v)]\right).$$ If $f$ is of class $C^{r}$, then $\hat{f}$ is of class $C^{r-1}$.

Let $\rho:\hat{M}\to\mathbb{R}$ be the continuous function defined by $$\rho(x, [v]):=\log\left\|Df_{x}\left(v\right)\right\|.$$

Given an $\hat{f}$-invariant measure $\hat{\mu}$, the \emph{projection} of $\hat{\mu}$ is defined by $\mu:=\hat{\mu}\circ\pi^{-1}$ and $\hat{\mu}$ is called the \emph{lift} of $\mu$. $\hat{\mu}$ is called an \emph{unstable lift} of $\mu$ if for $\hat{\mu}$-almost every point $\hat{x}$,
\begin{equation}\label{Birkhoff average positive for hat mu}
	\lim_{n\to +\infty}\frac{1}{n}\sum_{i=0}^{n-1}\rho\left(\hat{f}^{i}\left(\hat{x}\right)\right)>0.
\end{equation}

We define $$\lambda(\hat{f},\hat{\mu}):=\int_{\hat{M}}\rho d\,\hat{\mu}.$$ We list some basic properties.

\begin{itemize}
	\item Given an $\hat{f}$-invariant measure $\hat{\mu}$, if (\ref{Birkhoff average positive for hat mu}) holds for $\hat{\mu}$-almost every point $\hat{x}$, then the Lyapunov exponent $\lambda^{+}(f,x)$ is positive for $\mu$-almost every $x$, where $\mu$ is the projection of $\hat{\mu}$. This follows from the fact that for $\hat{\mu}$-almost every point $\hat{x} = (x, [v]) \in \hat{M}$ with $\|v\| = 1$, $$ \frac{1}{n}\sum_{i=0}^{n-1}\rho\left(\hat{f}^{i}(\hat{x})\right) = \frac{1}{n}\log\left\|Df_{x}^{n}(v)\right\|. $$
	\item If $\hat{\mu}_{n}\xrightarrow{\text{weak $\ast$}}\hat{\mu}$, then $\lambda(\hat{f},\hat{\mu}_{n})\to \lambda(\hat{f},\hat{\mu})$. This is a consequence of the continuity of the function $\rho$.
	\item Let $\mu$ be an $f$-invariant measure with $\lambda^{+}(f,x)>0\geq\lambda^{-}(f,x)$ for $\mu$-almost every $x$. Then $\mu$ has a unique unstable lift $\hat{\mu}^{+}$ which can be defined by: $$\hat{\mu}^{+}:=\int_{\hat{M}}\delta_{(x, E^{+}_{x})} d\,\mu(x).$$ Moreover (refer to \cite[Proposition 3.8]{BCS25}), $$\lambda^{+}(f,\mu)=\lambda(\hat{f},\hat{\mu})\quad\text{if and only if}\quad\hat{\mu}=\hat{\mu}^{+}.$$

	For more detailed discussions on the lifted dynamics, we refer to \cite[Section 3.4]{BCS22} and the references therein.

\end{itemize}

\subsection{Notations and conventions}
\begin{itemize}
	\item Whenever we say that $\sigma:[0,1]\to M$ is a $C^{r}$ curve, we always assume that it is an embedded curve.
	\item We identify $\sigma$ with its image whenever there is no ambiguity; in particular, we freely write $x\in\sigma$.
	\item Once a curve $\sigma$ is fixed, for any $x\in\sigma$, we denote by $\hat{x}=(x,[v])\in\hat{M}$  where $v$ is the unit tangent vector of $\sigma$ at $x$.
	\item Given $\hat{x}:=(x,[v])\in \hat{M}$ where $v$ is a unit vector in $T_{x}M$, we shall slightly abuse notation and write $$Df(\hat{x}):=Df_{x}(v)\in T_{f(x)}M.$$
	\item Write $$\left\|Df\right\|:=\max_{x\in M}\left\|Df_{x}\right\|, \quad \left\|Df^{-1}\right\|:=\max_{x\in M}\left\|Df^{-1}_{x}\right\|.$$ Note that $\left\|Df\right\|, \left\|Df^{-1}\right\| \geq 1$.
\end{itemize}

	\section{Volume growth of the submanifolds}\label{Volume growth of the submanifolds}
	The purpose of this section is to establish Theorem~\ref{main theorem of decomposition of volume growth of submanifolds} via two key technical results: Proposition~\ref{grow rate for major component in finite step n} and Proposition~\ref{volume growth for time n}.
	\begin{itemize}
		\item Proposition~\ref{grow rate for major component in finite step n} separates the volume growth into the entropic contribution and the Lyapunov-type contribution (the Yomdin term) via Burguet's reparametrizations, a variant of Yomdin's reparametrizations.
		\item Proposition~\ref{volume growth for time n} refines Proposition~\ref{grow rate for major component in finite step n} by controlling the error term arising from entropic contributions, more precisely, the local volume growth at geometric times in Proposition~\ref{grow rate for major component in finite step n}. It shows that this error term is negligible in the asymptotic regime.
	\end{itemize}

	\subsection{Reparametrization}

	Let $\sigma: [0,1]\to M$ be a $C^{r}$ curve. For an integer $s\ge 1$, the quantity $\left\|d^{s}\sigma\right\|$ denotes the supremum norm of the $s$-th derivative of $\sigma$, namely
	$$
	\left\|d^{s}\sigma\right\|
	:=
	\sup_{t\in[0,1]} \left\|d^{s}_{t}\sigma\right\|,
	$$
	where the norm is taken with respect to the Riemannian metric on $M$. We define the $C^{r}$ norm by $$
	\left\|\sigma\right\|_{C^{r}}
	:=
	\max_{1\leq s\leq r}\,
	\left\|d^{s}\sigma\right\|.$$ We say that $\sigma$ is \emph{bounded} if $$\max_{2\leq s\leq r}\,\left\|d^{s}\sigma\right\|\leq\frac{1}{6}\cdot\left\|d\sigma\right\|.$$ Note, consequently, we have
\begin{equation}\label{Property of boundedness}
	\left\|d\sigma\right\|\leq 2\cdot\left\|d_t\sigma\right\|,\,\forall t\in[0,1].
\end{equation} Moreover, the curve $\sigma$ is said to be \emph{$\varepsilon$-bounded} if $$\left\|d\sigma\right\|\leq \varepsilon.$$ Roughly speaking, a bounded curve is one that is nearly straight, whereas an $\varepsilon$-bounded curve is nearly straight and, in addition, has small length.

We often call an affine map from $[0,1]$ to itself a \emph{reparametrization}.

The following proposition is a reparametrization result in the spirit of Yomdin theory, developed by Burguet \cite[Lemma 12]{Bur24}. It shows that an $\varepsilon$-bounded curve can be decomposed, via affine reparametrizations, into finitely many pieces on which the image under $f$ remains bounded. Moreover, the number of such pieces and the sizes of the reparametrizations are quantitatively controlled by the expansion rates along the curve.

The ceiling function $\left\lceil x\right\rceil$ assigns to each real number $x$ the smallest integer $n$ such that $n\ge x$. Recall that, given a curve $\sigma$, $\hat{x}:=(x,[v])\in\hat{M}$  where $v$ is the unit tangent vector of $\sigma$ at $x\in \sigma$.

	\begin{Proposition}\cite[Lemma 12]{Bur24}\label{Burguet-Yomdin reparametrization}
		Let $f$ be a $C^r$ ($r>1$) diffeomorphism on a compact manifold $M$. There is a small constant $\varepsilon_{f}>0$ such that for any $\varepsilon_{f}$-bounded $C^r$ curve $\sigma :[0,1]\to M$ and any two integers $a,b\in\mathbb{Z}$, there is a family $\Theta$ of affine maps from $[0,1]$ to itself with the following properties:
		\begin{enumerate}
			\item $\left\{ x \in\sigma :~ \left\lceil\log\left\|Df_{x}\right\|\right\rceil=a, \,\left\lceil\log\left\|Df\left(\hat{x}\right)\right\|\right\rceil=b \right \}\subset \cup_{\theta\in \Theta} \sigma\circ\theta$,
			\item $f\circ \sigma\circ\theta$ is bounded for any $\theta\in\Theta$,
			\item $|\theta'|\leq e^{-\frac{a-b+3r}{r-1}}$ for any $\theta\in\Theta$,
			\item $\log\#\Theta\leq  A_{r} + \frac{a-b}{r-1}$ where $A_{r}$ is a universal constant depending only on $r$.
		\end{enumerate}

	\end{Proposition}
	\begin{Remark}
	 Compared with \cite[Lemma 12]{Bur24}, we slightly modify the third property by adding the term " $3r$ " in the exponent. This modification is acceptable, since the contraction factor is only required to depend on $r$, and any adjustment by a uniform constant depending solely on $r$ does not affect the validity of the result. The role of the additional term " $3r$ " is purely technical (see the remarks following Proposition~\ref{grow rate for major component in finite step n} for details).
	\end{Remark}

The following result, due to Burguet \cite[Lemma 8]{Bur24'}, can be regarded as an $n$-step and local version of the reparametrization result stated in Proposition~\ref{Burguet-Yomdin reparametrization}.

For $x\in M$, $n\in\mathbb{N}$ and $\varepsilon>0$, recall that the usual \emph{$(n,\varepsilon)$ dynamical ball} at $x$ is defined by $$B(x,n,\varepsilon):=\left\{y\in M:\ d\left(f^{k}(y),f^{k}(x)\right)<\varepsilon \ \text{ for all } 0\leq k<n\right\}.$$
		\begin{Proposition}\cite[Lemma 8]{Bur24'}\label{Corollary 2 Burguet-Yomdin reparametrization}
		Let $f$ be a $C^r$ ($r>1$) diffeomorphism on a compact manifold $M$. For any integer $p\geq 1$, there are a small constant $\varepsilon_{f^{p}}>0$ and a large constant $B_{p}$ such that for any  $\varepsilon_{f^{p}}$-bounded $C^r$ curve $\sigma :[0,1]\to M$, any $x\in\sigma$ and any $n\in\mathbb{N}$, there is a family $\Theta$ of affine maps from $[0,1]$ to itself with the following properties:
		\begin{enumerate}
			\item $B\left(x,n,\varepsilon_{f^{p}}\right)\cap\sigma\subset \cup_{\theta\in \Theta} \sigma\circ\theta$,
			\item $f^{k}\circ \sigma\circ\theta$ is $\varepsilon_{f^{p}}$-bounded for any $\theta\in\Theta$ and $0\leq k\leq n$,

			\item $$\log\#\Theta\leq B_{p}+n\cdot A_{r}+\frac{1}{r-1}\cdot\frac{1}{p}\sum_{k=0}^{n-1}\left(\left\lceil\log\left\|Df^{p}_{f^{k}(x)}\right\|\right\rceil-\left\lceil\log\left\|Df^{p}\left(\hat{f}^{k}\left(\hat{x}\right)\right)\right\|\right\rceil\right)$$ where $A_{r}$ is a universal constant depending only on $r$.
		\end{enumerate}

	\end{Proposition}
\begin{Remark}
Compared with \cite[Lemma 8]{Bur24'}, we insert ceiling operations in the third property. This modification is acceptable, since the resulting discrepancy is at most linear in $n$ and can therefore be absorbed into the constant $A_r$. Also, to simplify notation, we assume that the constant $A_{r}$ appearing here is the same as the one in Proposition~\ref{Burguet-Yomdin reparametrization}.
\end{Remark}

	\subsection{Decomposition of the volume growth}
	Recall that $\rho:\hat{M}\to\mathbb{R}$ is the continuous function $$\rho(x, [v]):=\log\left\|Df_{x}\left(v\right)\right\|$$ where $v$ is a unit vector in $T_{x}M$. Define another continuous function $\rho':\hat{M}\to\mathbb{R}$ (introduced by Burguet \cite{Bur24}) by $$\rho'(x,[v]):=\rho(x, [v])-\frac{1}{r}\log^{+}\left\|Df_{x}\right\|$$ where $\log^{+} x := \max\{\log x,\, 0\}$.
	\begin{Definition}\label{Definition of geometric time}
		We call $n\geq 1$   a \emph{geometric time} of $\hat{x}=(x,[v])$  if for any $0\leq k\leq n-1$,
	\begin{equation}\label{definition of geometric time inequality}
		\sum_{i=k}^{n-1}\rho'\left(\hat{f}^{i}(\hat{x})\right)\geq n-k.
	\end{equation}
	We let $E(\hat{x})$ denote the set of all geometric times of $\hat{x}$.
	\end{Definition}
	Roughly speaking, the existence of infinitely many geometric times for a point $\hat{x}$ suggests that its projection $x=\pi(\hat{x})$ has a positive maximal Lyapunov exponent larger than $1$. At first glance, this may appear to be an overly strong requirement in the definition of geometric times. However, the precise expansion rate $1$ used in this definition is in fact irrelevant: since our arguments ultimately concern the dynamics of large iterates of $\hat{f}$ (see the proof of Theorem~\ref{main theorem of decomposition of volume growth of submanifolds}), any fixed normalization of the expansion threshold leads to an equivalent notion for our purposes.

	\begin{Lemma}\label{property of geometric time}
		Let $\hat x\in\hat M$. Suppose that $m\ge1$ is not a geometric time of $\hat x$, and let
		$n<m$ be the largest geometric time smaller than $m$ (set $n=0$ if there is no geometric time before $m$). Then
		$$\sum_{i=n}^{m-1}\rho'\left(\hat f^{i}(\hat{x})\right)<m-n.$$
	\end{Lemma}

	\begin{proof}
		Since $m$ is not a geometric time, there exists $k\in\{0,\dots,m-1\}$ such that
		\begin{equation}\label{nongeometric time at k}
			\sum_{i=k}^{m-1}\rho'\left(\hat f^{i}(\hat{x})\right)<m-k.
		\end{equation}
		Choose $k$ minimal with this property.
		  If $k>n$, since $k$ is not a geometric time, there exists
			$\ell\in\{0,\dots,k-1\}$ such that
			$$\sum_{i=\ell}^{k-1}\rho'\left(\hat f^{i}(\hat{x})\right)<k-\ell.$$
			Adding this to inequality (\ref{nongeometric time at k}) yields with $\ell<k$
			$$\sum_{i=\ell}^{m-1}\rho'\left(\hat f^{i}(\hat{x})\right)<m-\ell,$$ contradicting the minimality of $k$. Therefore, we must have $k\leq n$. If $k=n$, the conclusion holds by inequality (\ref{nongeometric time at k}) (this remains true even in the case $n=0$, where $k$ must also be $0$). If $k<n$, since $n$ is a geometric time,
			$$\sum_{i=k}^{n-1}\rho'\left(\hat f^{i}(\hat{x})\right)\ge n-k.$$ Subtracting this from the inequality (\ref{nongeometric time at k}) yields
			$$\sum_{i=n}^{m-1}\rho'\left(\hat f^{i}(\hat{x})\right)<(m-k)-(n-k)=m-n,$$
			which is exactly the desired conclusion.
			
	\end{proof}

	Given a subset $E\subset \mathbb{N}$ and a number $L\in\mathbb{N}$, we define
	$$E^{L}:=\bigcup [i,j]$$ where the union runs over all $i,j\in E$ with $0\leq j-i\leq L$. Any element of $E^{L}(\hat{x}):=\left(E(\hat{x})\right)^{L}$ lies at distance at most $L$ from a geometric time. Since this discrepancy is uniformly bounded by $L$, it does not affect the asymptotic estimates as $n\to+\infty$. We can view $E^{L}(\hat{x})$ as the set of \emph{$L$-extended geometric times} of $\hat{x}$.

Let $\varepsilon_{f}\gg \varepsilon_{f^{p}}$ be the small constants in Proposition \ref{Burguet-Yomdin reparametrization} w.r.t. $f$ and $f^{p}$.

\begin{Definition}\label{Definition of periodic sources}
A periodic point $z$ of $f$ is called a \emph{periodic source} if both Lyapunov exponents of $z$ are positive. 

Let $T_{z}$ denote the period of $z$ and let $a:=\min_{0<i<j\leq T_{z}}\,d\left(f^{i}(z), f^{j}(z)\right)$ \footnote{If $z$ is a fixed point (i.e., $T_{z}=1$), we set $a:=1$.}.  We fix a small number $0<r_{z}<\min\{\frac{a}{100}, \frac{\varepsilon_{f}}{100\left\|Df\right\|}\}$ \footnote{The purpose of introducing the small number $r_{z}$ will be revealed in Proposition~\ref{volume growth for time n}.} such that for any $0\leq k \leq T_{z}-1$,
\begin{enumerate}
	\item  $r_{z}=r_{f^{k}\left(z\right)}$,
	\item 
	$$f^{-T_{z}}\left(B\left(f^{k}(z),r_{z}\right)\right)\subset B\left(f^{k}(z),r_{z}\right),$$
	\item $$f\left(B\left(f^{k}(z),r_{z}\right)\right)\subset B\left(f^{k+1}(z),\frac{a}{100}\right).$$
\end{enumerate}
We list all periodic sources by $z_1, z_{2}, z_{3},\cdots$ (since there are at most countably many of them). We may choose $r_z$ sufficiently small so that the associated balls $B(z_{i}, r_{z_{i}})$ are pairwise disjoint (since periodic sources form a discrete set).
\end{Definition}

Given a $C^{r}$ curve $\sigma:[0,1]\to M$ and a point $x\in\sigma$, recall that we denote by $\hat{x}=(x,[v])\in\hat{M}$  where $v$ is the unit tangent vector of $\sigma$ at $x$. Given any $L,m,n, p\in \mathbb{N}$ with $m\leq n$, any $E\subset [0,n)$ and a family of integers $\beta:=\left\{\left(\beta'_{i},\, \beta''_{i}\right)\right\}_{i\in [0,n)}$, we define

$$\begin{aligned}
	O^{p,m,L}_{n, E,\beta}:=\Big\{x\in\sigma:~&  \left(E(\hat{x})\cap [0,n)\right)^{L}=E,\,\,\exists 1\leq s\leq n, \text{ s.t. } x\in B(z_{s}, m, r_{z_{s}}) \text{ and }\\ &  \frac{1}{p}\left\lceil\log\left\|Df^{p}_{f^{i}(x)}\right\|\right\rceil= \beta'_{i}, \,\frac{1}{p}\left\lceil\log\left\|Df^{p}\left(\hat{f}^{i}\left(\hat{x}\right)\right)\right\|\right\rceil=\beta''_{i},\, i\in [m,n)\cap E,\\ & \left\lceil\log\left\|Df_{f^{i}(x)}\right\|\right\rceil=\beta'_{i}, \,\left\lceil\log\left\|Df\left(\hat{f}^{i}\left(\hat{x}\right)\right)\right\|\right\rceil=\beta''_{i},\, i\in [0,n)\setminus\left([m,n)\cap E\right)\Big \}.
\end{aligned}$$
If $m=0$, we define
$$\begin{aligned}
	O^{p,0,L}_{n, E,\beta}:=\Big\{x\in\sigma:~& \left(E(\hat{x})\cap [0,n)\right)^{L}=E,\,\,\forall 1\leq s\leq n, x\notin B(z_{s},  r_{z_{s}}) \text{ and }\\&  \frac{1}{p}\left\lceil\log\left\|Df^{p}_{f^{i}(x)}\right\|\right\rceil= \beta'_{i}, \,\frac{1}{p}\left\lceil\log\left\|Df^{p}\left(\hat{f}^{i}\left(\hat{x}\right)\right)\right\|\right\rceil=\beta''_{i},\, i\in [m,n)\cap E,\\ & \left\lceil\log\left\|Df_{f^{i}(x)}\right\|\right\rceil=\beta'_{i}, \,\left\lceil\log\left\|Df\left(\hat{f}^{i}\left(\hat{x}\right)\right)\right\|\right\rceil=\beta''_{i},\, i\in [0,n)\setminus\left([m,n)\cap E\right)\Big \}.
\end{aligned}$$

	\begin{Proposition}\label{grow rate for major component in finite step n}
		Let $f$ be a $C^r$ ($r>1$) diffeomorphism on a compact surface $M$. Consider the following setting:
		\begin{itemize}
			\item any $L,m,n, p\in\mathbb{N}$ with $m\leq n$,
			\item any $\varepsilon_{f}$-bounded $C^{r}$ curve $\sigma:[0,1]\to M$,

			\item any  $O^{p,m,L}_{n, E,\beta}$ with $\beta:=\left\{\left(\beta'_{i}, \beta''_{i}\right)\right\}_{i\in [0,n)}$.
		\end{itemize}
		 Then we have
		 $$\begin{aligned}
		 	\log{\rm Vol}\left(f^{n}\left(O^{p,m,L}_{n, E,\beta}\right)\right)
		 	&\leq \#\left([m,n)\cap E\right)\cdot h_{\rm top}(f)+\left(m + \#\left([m,n)\setminus E\right)\right)\cdot\frac{\log\left\|Df\right\|}{r}\\
		 	&+\sum_{k\in [m,n)\cap E}\frac{\beta'_{k}-\beta''_{k}}{r-1}+ C_{n,L}.
		 \end{aligned}$$ where the error term $C_{n,L}$ satisfies $$\limsup_{L \to +\infty}\limsup_{n \to +\infty}\frac{C_{n,L}}{n}=C_{r}$$ for some universal constant $C_{r}$ depending only on $r$.

	\end{Proposition}
The proof of Proposition~\ref{grow rate for major component in finite step n} is rather involved. We therefore first explain the underlying idea.

The set $O^{p,m,L}_{n,E,\beta}$ represents a portion of the curve $\sigma$ on which all points share the same collection of (L-extended) geometric times $E$ (we call the complementary times of $E$ the \emph{neutral times}) and the same derivative data along their orbits. For simplicity, we assume that $\sigma$ intersects only one periodic source, namely a repeller $z$.

We decompose the time interval $[0,n)$ into three parts \footnote{In the proof of Proposition~\ref{grow rate for major component in finite step n}, an additional part, referred to as \emph{tail times}, is introduced for technical reasons.}
: trapping times, (partial) geometric times, and (partial) neutral times, $$[0,m),\quad [m,n)\cap E,\quad [m,n)\setminus E.$$

	\begin{itemize}
		\item Trapping times $[0,m)$. During this period, the set $O^{p,m,L}_{n,E,\beta}$ remains confined in a small neighborhood of the repeller $z$. Consequently, its growth is purely local \footnote{Positivity of the local volume growth is not excluded even near a repeller; it may occur when oscillatory effects dominate the geometric repelling of the region. See Example~\ref{example repeller positive volume growth}.}.
		 The corresponding growth rate is bounded by $\frac{\log\left\|Df\right\|}{r}$ (which gives rise to the Yomdin term $\frac{\lambda^{+}(f)}{r}$).
		\item Geometric times $[m,n)\cap E$.  At these times, $O^{p,m,L}_{n,E,\beta}$ exhibits both a \emph{global} growth, bounded by $\#\left([m,n)\cap E\right)\cdot h_{\rm top}(f)$, and a \emph{local} growth, bounded by

		$$\sum_{k\in [m,n)\cap E}\frac{\beta'_{k}-\beta''_{k}}{r-1}.$$

		In Proposition~\ref{grow rate for major component in finite step n}, we retain the latter term temporarily and later show, in Proposition~\ref{volume growth for time n}, that its contribution is negligible.
		\item Neutral times $[m,n)\setminus E$. At neutral times, the growth of $O^{p,m,L}_{n,E,\beta}$ is again purely local, with growth rate bounded by $\frac{\log\left\|Df\right\|}{r}$.
	\end{itemize}

It is worth emphasizing that, at neutral times, the curve does not necessarily fail to grow. Rather, by definition of geometric times, the growth rate is modest-specifically, it is bounded by a constant strictly less than $1$ (see the inequality (\ref{definition of geometric time inequality}) in Definition \ref{Definition of geometric time} ). The precise threshold is irrelevant for our purposes; what matters is that this bound is uniform and does not depend on the dynamics of $f$.

\begin{proof}[Proof of Proposition \ref{grow rate for major component in finite step n}]
	We divide the proof into three steps.
		\medskip
		\medskip

\noindent\textbf{\large Step 1: An inductive construction of reparametrizations.}
	\medskip

 We first define some numbers:
\begin{itemize}
	\item Let $n_E$ be the maximal geometric time in $[0,n)$, that is, the largest element of $E\cap[0,n)$. If no such geometric time exists, we set $n_{E}=0$.

	\item Let $\mathcal{A}$ be the subset of $[0,n_{E}]$ obtained by removing all the integers $k$ such that $$k, \, k+1\in [m,n_{E}]\cap E.$$
	We enumerate the elements of $\mathcal{A}$ as $$0=a_{1}<a_{2}<\cdots<a_{l_{0}}=n_{E}.$$
	\item We consider several parts  \footnote{$n_E$ is regarded as both a geometric time and a tail time, purely for technical convenience.} in $[0,n)$. Write $\mathcal{A}_{i}:=\{a_{1},\cdots,a_{i}\}\subset\mathcal{A}$.
	\begin{itemize}
		\item {\bf Trapping times}: $\Gamma^{1}_{i}:=\mathcal{A}_{i}\cap [0,m)$,
		\item {\bf Geometric times} \footnote{We emphasize that the set $\Gamma^{2}_{i}$ is not meant to represent all geometric times. For technical reasons, it only records the right endpoints of the geometric intervals in $[m,n_{E}]\cap E$.}: $\Gamma^{2}_{i}:=\left(\mathcal{A}_{i}\cap E\right)\setminus [0,m)$,
		\item {\bf Neutral times}: $\Gamma^{3}_{i}=\mathcal{A}_{i}\setminus\left(\Gamma^{1}_{i}\cup\Gamma^{2}_{i}\right)$,
		\item {\bf Tail times}: $[n_{E},n)$.
	\end{itemize}
\end{itemize}

Recall that $1\gg\varepsilon_{f}\gg\varepsilon_{f^{p}}>0$ and $A_{r}$ are the constants in Proposition \ref{Burguet-Yomdin reparametrization} w.r.t. $f$ and  $f^{p}$. Let $B_{p}$ be the constant in Proposition \ref{Corollary 2 Burguet-Yomdin reparametrization}. We may further assume that $\varepsilon_{f^{p}}$ is sufficiently small and $B_{p}$ is sufficiently large such that
\begin{itemize}
	\item for any $k\in\mathbb{N}$, the minimal cardinality of a $(k,\varepsilon_{f^{p}})$-spanning set of $M$ is bounded above: $$\log r(f,k,\varepsilon_{f^{p}})  \footnote{ Recall (see \cite{Wal82} for background in ergodic theory) that a set $\Omega\subset M$ is called \emph{$(k,\varepsilon)$-spanning} if for every $x\in M$,
		there exists $y\in \Omega$ such that
		$$
		d\bigl(f^i(x),f^i(y)\bigr)<\varepsilon
		\quad\text{for all } i=0,1,\dots,k-1.$$ Denote by $r(f,k,\varepsilon)$ the minimal cardinality of a $(k,\varepsilon)$-spanning set of $M$.}\leq B_{p}+k\left(h_{\rm top }(f)+A_{r}\right),$$
	\item there is a family $\hat{\Theta}$ \footnote{Such a family $\hat{\Theta}$ can be obtained by adapting the argument in Lemma~\ref{Decomposition into varepsilon bounded curves}, yielding the bound $\#\hat{\Theta}\lesssim \frac{\varepsilon_{f}}{\varepsilon_{f^{p}}}$.}
	 of affine maps from $[0,1]$ into itself with $$\#\hat{\Theta}\leq B_{p},\quad\bigcup_{\theta\in\hat{\Theta}}\theta\left([0,1]\right)=[0,1]$$ such that for any $\varepsilon_{f}$-bounded curve $\widetilde{\sigma}$ and any $\theta\in\hat{\Theta}$, $\widetilde{\sigma}\circ\theta$ is an $\varepsilon_{f^{p}}$-bounded curve.
\end{itemize}

Define $$\Delta:=\left\{a_{k}:~a_{k+1}\in\Gamma^{2}_{l_{0}}\right\}.$$

	Given $1\leq s\leq n$ and a family of integers $\gamma:=\left\{\left(\gamma'_{i},\, \gamma''_{i}\right)\right\}_{i\in [0,n)}$, we define \footnote{In the case $m=0$, we define $\Omega^{p,0,L,s}_{n, E,\gamma}$ by replacing the condition "$x\in B(z_{s}, m, r_{z_{s}})$" in $\Omega^{p,m,L,s}_{n, E,\gamma}$ with "$x\notin B(z_{s}, r_{z_{s}})$", as before.} :
$$\begin{aligned}
	\Omega^{p,m,L,s}_{n, E,\gamma}:=\Big\{&x\in\sigma:~  \left(E(\hat{x})\cap [0,n)\right)^{L}=E,\,x\in B(z_{s}, m, r_{z_{s}}) \text{ and }\\ &  \frac{1}{p}\left\lceil\log\left\|Df^{p}_{f^{i}(x)}\right\|\right\rceil= \gamma'_{i}, \,\frac{1}{p}\left\lceil\log\left\|Df^{p}\left(\hat{f}^{i}\left(\hat{x}\right)\right)\right\|\right\rceil=\gamma''_{i},\, i\in \left(\left([m,n)\cap E\right)\setminus \Gamma^{2}_{l_{0}}\right)\cup \Delta,\\ & \left\lceil\log\left\|Df_{f^{i}(x)}\right\|\right\rceil=\gamma'_{i}, \,\left\lceil\log\left\|Df\left(\hat{f}^{i}\left(\hat{x}\right)\right)\right\|\right\rceil=\gamma''_{i},\, i\in \left(\left([0,n)\setminus\left([m,n)\cap E\right)\right)\cup\Gamma^{2}_{l_{0}}\right)\setminus \Delta\Big \}.
\end{aligned}$$

The definition of $\Omega^{p,m,L,s}_{n,E,\gamma}$ differs from that of $O^{p,m,L}_{n,E,\beta}$ in two respects: first, it fixes a periodic source $z_s$; second, it modifies the definition at times belonging to $\Delta$ and $\Gamma^{2}_{l_0}$ for technical convenience. These times constitute only a small portion of $[0,n)$ (at most $2n/L$) and therefore do not affect the asymptotic estimates.

Next we carry out an inductive construction of a family of reparametrizations to cover $\Omega^{p,m,L,s}_{n, E,\gamma}$.

\begin{Claim}
	There are families $\{\Theta_{i}\}_{1\leq i\leq l_{0}}$ of affine maps from $[0,1]$ to itself such that for any $1\leq i\leq l_{0}$,

		\begin{enumerate}
		\item  $\Omega^{p,m,L,s}_{n, E,\gamma}\subset \cup_{\theta\in \Theta_{i}} \sigma\circ\theta$,
		\item for each $\theta_{i}\in\Theta_{i}$ ($i>1$) with $a_{i}\notin\Gamma^{2}_{i}$, there are some $\theta_{i-1}\in \Theta_{i-1}$ and some affine map $\theta:[0,1]\to[0,1]$ with $|\theta'|\leq e^{-\frac{\gamma'_{a_{i-1}}-\gamma''_{a_{i-1}}+3r}{r-1}}$ such that $\theta_{i}=\theta_{i-1}\circ\theta$,
		\item for any $\theta\in\Theta_{i}$, $f^{a_{i}}\circ\sigma\circ\theta$ is $\varepsilon_{f}$-bounded,
		\item we have \footnote{By convention, we set $[a,b)=\emptyset$ if $a>b$.}
		$$\log\# \Theta_{i}\leq  \#\left([m,a_{i}]\cap E\right)\cdot h_{\rm top}(f)+ \sum_{k<a_{i}}\frac{\gamma'_{k}-\gamma''_{k}}{r-1}+ \#\Gamma^{2}_{i}\cdot 3B_{p}+ a_{i}\cdot 2A_{r}$$	where $A_{r}$ is the constant given in Proposition~\ref{Burguet-Yomdin reparametrization}, which coincides with that in Proposition~\ref{Corollary 2 Burguet-Yomdin reparametrization}.

		\end{enumerate}

\end{Claim}
\begin{proof}[Proof of Claim]

	We start an inductive construction to build $\Theta_{0\leq i\leq l_{0}}$.
	\paragraph{\it{\bfseries\small  Initial step:}\\}
	For $i=1$ ($a_{1}=0$), the curve $\sigma$ is $\varepsilon_{f}$-bounded by our assumption. We define $\Theta_{1}:=\{{\rm Id:[0,1]\to[0,1]}\}$. The properties in the claim are trivially satisfied.
	\paragraph{\it{\bfseries \small Inductive step:}\\}
	Assume that we have families $\{\Theta_{k}\}_{1\leq k\leq i}$ of affine maps satisfying the four properties in the claim. We consider three cases:

	\begin{enumerate}
		\item $a_{i+1}\in \Gamma^{1}_{i+1}$.

		Note that in this case, $$a_{i+1}=a_{i}+1=i\in[0,m).$$ As a consequence, given any $\theta_{i}\in\Theta_{i}$, after composing with a suitable reparametrization, we may assume \footnote{This follows directly from the definition of $\Omega^{p,m,L,s}_{n, E,\gamma}$, since we only consider points lying in the dynamical ball $B(z_{s}, m, r_{z_{s}})$ where $r_{z_{s}}<\frac{\varepsilon_{f}}{100\left\|Df\right\|}$.} \begin{equation}\label{trapping case}
			f^{a_{i}}\circ\sigma\circ\theta_{i}\subset B\left(f^{a_{i}}\left(z_{s}\right), \frac{\varepsilon_{f}}{10\left\|Df\right\|}\right).
		\end{equation} Applying Proposition \ref{Burguet-Yomdin reparametrization} with $a=\gamma'_{a_{i}}, b=\gamma''_{a_{i}}$, we get a family $\Theta$ of affine maps from $[0,1]$ to itself with $$\log \#\Theta\leq \frac{\gamma'_{a_{i}}-\gamma''_{a_{i}}}{r-1}+A_{r} $$ such that  each $f^{a_{i+1}} \circ \sigma \circ \theta_{i} \circ \theta, \,\theta \in \Theta$ is bounded and $$|\theta'|\leq e^{-\frac{\gamma'_{a_{i}}-\gamma''_{a_{i}}+3r}{r-1}}.$$ By (\ref{trapping case}), $$f^{a_{i+1}} \circ \sigma \circ \theta_{i} \circ \theta\subset B\left(f^{a_{i+1}}\left(z_{s}\right), \frac{\varepsilon_{f}}{10}\right).$$ Hence, we may assume, after composing $\theta$ with a suitable reparametrization \footnote{This can be shown directly from the definitions; see \cite[Lemma~13]{Bur24} and the references therein for a precise argument.},  that $f^{a_{i+1}} \circ \sigma \circ \theta_{i} \circ \theta$ is $\varepsilon_{f}$-bounded.  We then define $$\Theta_{i+1}:=\{\theta_{i}\circ\theta:~  \,\theta\in\Theta, \,\theta_{i}\in \Theta_{i}\}.$$ Note in this case,  $$\Gamma^{1}_{i+1}=\Gamma^{1}_{i}\cup\{a_{i+1}\},\quad \Gamma^{2}_{i+1}=\Gamma^{2}_{i},\quad \Gamma^{3}_{i+1}=\Gamma^{3}_{i}, \quad a_{i+1}-a_{i}=1,$$ and $$\#\left([m,a_{i}]\cap E\right)=\#\left([m,a_{i+1}]\cap E\right)=0.$$ Hence, we have the desired bound on $\log \# \Theta_{i+1}$:
		$$\begin{aligned}
			\log \# \Theta_{i+1}
			&\leq \log \# \Theta_{i}+\log \# \Theta\\
			&\leq\#\left([m,a_{i+1}]\cap E\right)\cdot h_{\rm top}(f)+ \sum_{k<a_{i+1}}\frac{\gamma'_{k}-\gamma''_{k}}{r-1}+\#\Gamma^{2}_{i+1}\cdot 3B_{p}+ a_{i+1}\cdot 2A_{r}.
		\end{aligned}$$

		The remaining properties asserted in the claim are obvious.

	\item $a_{i+1}\in \Gamma^{2}_{i+1}$.

	In this case, $a_{i}\notin \Gamma^{2}_{i}$ and $(a_{i},a_{i+1}]\subset E$.

	Recall the definitions of $\varepsilon_{f^{p}}$, $B_p$ and $\hat{\Theta}$ ($\#\hat{\Theta}\leq B_{p}$) introduced before the claim.

	For any $\theta_{i}\in \Theta_{i}$ and $\hat{\theta}\in\hat{\Theta}$, $f^{a_{i}}\circ\sigma\circ\theta_{i}\circ\hat{\theta}$ is an $\varepsilon_{f^{p}}$-bounded curve. Next, we decompose each $f^{a_i}\circ \sigma \circ \theta_i \circ \hat{\theta}$ into smaller pieces, with the decomposition taken up to time $a_{i+1}-a_{i}$. We choose some $(a_{i+1}-a_{i}, \varepsilon_{f^{p}})$ spanning set $\Omega\subset f^{a_{i}}\circ\sigma\circ\theta_{i}\circ\hat{\theta}$ with $$\log\#\Omega=\log r(f,a_{i+1}-a_{i}, \varepsilon_{f^{p}})\leq B_{p}+\left(a_{i+1}-a_{i}\right)\cdot\left(h_{\rm top}(f)+A_{r}\right).$$ Note that we are only concerned with points $z\in f^{a_{i}}\circ\sigma\circ\theta_{i}\circ\hat{\theta}$ satisfying $$\frac{1}{p}\left\lceil\log\left\|Df^{p}_{f^{j}(z)}\right\|\right\rceil= \gamma'_{a_{i}+j}, \quad\frac{1}{p}\left\lceil\log\left\|Df^{p}\left(\hat{f}^{j}\left(\hat{z}\right)\right)\right\|\right\rceil=\gamma''_{a_{i}+j},\quad j\in [0, a_{i+1}-a_{i}).$$ where $\hat{z}:=(z,[v])$, with $v$ denoting the unit tangent vector to the curve
	$f^{a_{i}}\circ\sigma\circ\theta_{i}\circ\hat{\theta}$ at $z$. Hence, without loss of generality, we may assume that all points in $\Omega$ satisfy these relations.

	Applying Proposition \ref{Corollary 2 Burguet-Yomdin reparametrization} to each point $z\in \Omega$ with $n=a_{i+1}-a_{i}$, we have a family $\Theta_{z}$ of affine maps from $[0,1]$ to itself with the following properties:
	\begin{enumerate}
		\item $B(z,a_{i+1}-a_{i},\varepsilon_{f^{p}})\cap f^{a_{i}}\circ\sigma\circ\theta_{i}\circ\hat{\theta}\subset \cup_{\theta\in \Theta_{z}} f^{a_{i}}\circ\sigma\circ\theta_{i}\circ\hat{\theta}\circ\theta$,
		\item $f^{k}\circ f^{a_{i}}\circ\sigma\circ\theta_{i}\circ\hat{\theta}\circ\theta$ is $\varepsilon_{f^{p}}$-bounded for any $\theta\in\Theta_{z}$ and $0\leq k\leq a_{i+1}-a_{i}$,

		\item We have $$\log\#\Theta_{z}\leq B_{p}+\left(a_{i+1}-a_{i}\right)\cdot A_{r}+\frac{1}{r-1}\cdot\sum_{k=a_{i}}^{a_{i+1}-1}\left(\gamma'_{k}-\gamma''_{k}\right).$$

	\end{enumerate}
	Note that  $$\Gamma^{1}_{i+1}=\Gamma^{1}_{i},\quad \Gamma^{2}_{i+1}=\Gamma^{2}_{i}\cup\{a_{i+1}\},\quad \Gamma^{3}_{i+1}=\Gamma^{3}_{i}.$$ Since
	the half-open interval $(a_i,a_{i+1}]\subset E$, we have $$\#\left([m,a_{i+1}]\cap E\right)=\#\left([m,a_{i}]\cap E\right)+a_{i+1}-a_{i}.$$

We then define $$\Theta_{i+1}:=\left\{\theta_{i}\circ\hat{\theta}\circ\theta:~ \theta_{i}\in \Theta_{i},\,\hat{\theta}\in\hat{\Theta},\, \theta\in\Theta_{z}, \, z\in\Omega\right\}$$ and we have
$$\begin{aligned}
	\log \# \Theta_{i+1}
	&\leq \log \# \Theta_{i}+ \log \#\hat{\Theta}+\max_{z\in\Omega}\log\#\Theta_{z}+\log\#\Omega\\
	&\leq\#\left([m,a_{i+1}]\cap E\right)\cdot h_{\rm top}(f)+ \sum_{k<a_{i+1}}\frac{\gamma'_{k}-\gamma''_{k}}{r-1} + \#\Gamma^{2}_{i+1}\cdot 3B_{p}+ a_{i+1}\cdot 2A_{r}.
\end{aligned}$$

The remaining properties asserted in the claim are obvious. The terms "$3B_{p}, 2A_{r}$" arise from this case.

		\item $a_{i+1}\in \Gamma^{3}_{i+1}$.

		In this case, $a_{i+1}-a_{i}=1$.

		Given any $\theta_{i}\in\Theta_{i}$, for the $\varepsilon_{f}$-bounded curve $f^{a_{i}}\circ\sigma\circ\theta_{i}$, we apply Proposition \ref{Burguet-Yomdin reparametrization} with $a=\gamma'_{a_{i}}, b=\gamma''_{a_{i}}$ and we get a family $\Theta$ of affine maps from $[0,1]$ to itself with $$\log \#\Theta\leq \frac{\gamma'_{a_{i}}-\gamma''_{a_{i}}}{r-1}+ A_{r}$$ such that  each $f^{a_{i+1}} \circ \sigma \circ \theta_{i} \circ \theta, \,\theta \in \Theta$ is bounded and $$|\theta'|\leq e^{-\frac{\gamma'_{a_{i}}-\gamma''_{a_{i}}+3r}{r-1}}.$$  Next we show that $f^{a_{i+1}}\circ \sigma\circ\theta_{i}\circ\theta,\, \theta\in\Theta$ is actually $\varepsilon_{f}$-bounded. Define $$j=\max\left\{k:~  a_{k}\in  E, \,\,a_{k}<a_{i+1}\right\}.$$ If the set in the right side is empty, we set $j=1$ which gives $a_{j}=a_{1}=0$.	We also note that $(a_{j}, a_{i+1}]\cap E=\emptyset$ by the maximality of $j$.  Consequently, $$a_{i+1}-a_{j}=i+1-j.$$ For any $x\in \Omega^{p,m,L,s}_{n, E,\gamma}$, since $a_{i+1}$ is not a geometric time while $a_{j}$ is the largest geometric time before $a_{i+1}$ \footnote{$a_{j}$ is the right endpoint of a geometric interval and is not necessarily larger than $m$.}, by Lemma \ref{property of geometric time}, $$\sum_{k=a_{j}}^{a_{i+1}-1}\rho'\left(\hat{f}^{k}(\hat{x})\right)=\sum_{k=a_{j}}^{a_{i}}\rho'\left(\hat{f}^{k}(\hat{x})\right)=\sum_{k=j}^{i}\rho'\left(\hat{f}^{a_{k}}(\hat{x})\right)< a_{i+1}-a_{j}=i+1-j.$$ Taking the ceiling operation into consideration, we then have
		\begin{equation}\label{ceiling operation}
		\sum_{k=j}^{i}\gamma''_{a_{k}}-\frac{\gamma'_{a_{k}}}{r}<2\left(a_{i+1}-a_{j}\right)=2\left(i+1-j\right).
		\end{equation}

		Let $t\in[0,1]$ be such that $\sigma\circ\theta_{i}\circ\theta(t)=x\in \Omega^{p,m,L,s}_{n, E,\gamma}$ for some $x$. By the second assumption in the claim, there are some $\theta_{j}\in \Theta_{j}$ and some affine maps $\widetilde{\theta}_{j},\widetilde{\theta}_{j+1},\cdots, \widetilde{\theta}_{i-1}$ with  $$|\widetilde{\theta}'_{k}|\leq e^{-\frac{\gamma'_{a_{k}}-\gamma''_{a_{k}}+3r}{r-1}},\quad j\leq k\leq i-1$$ such that $$\theta_{i}=\theta_{j}\circ \widetilde{\theta}_{j}\circ\widetilde{\theta}_{j+1}\circ\cdots\circ\widetilde{\theta}_{i-1}.$$ We have
		$$\begin{aligned}
		\left\|d\left(f^{a_{i+1}}\circ \sigma\circ\theta_{i}\circ\theta\right)\right\|
			&\leq 2\left\|d_{t}\left(f^{a_{i+1}}\circ \sigma\circ\theta_{i}\circ\theta\right)\right\|\text{\quad by (\ref{Property of boundedness})}\\
			&\leq 2\left\|Df^{a_{i+1}-a_{j}}\left(\hat{f}^{a_{j}}(\hat{x})\right)\right\|\cdot \left\|d_{t}\left(f^{a_{j}}\circ \sigma\circ\theta_{i}\circ\theta\right)\right\|\\
			&\leq 2\left\|Df^{a_{i+1}-a_{j}}\left(\hat{f}^{a_{j}}(\hat{x})\right)\right\|\cdot \left\|d\left(f^{a_{j}}\circ \sigma\circ\theta_{i}\right)\right\|\cdot|\theta'|\\
			&\leq 2\left\|Df^{a_{i+1}-a_{j}}\left(\hat{f}^{a_{j}}(\hat{x})\right)\right\|\cdot \left\|d\left(f^{a_{j}}\circ \sigma\circ\theta_{j}\right)\right\|\cdot\left(\prod_{k=j}^{i-1}|\widetilde{\theta}'_{k}|\right)\cdot|\theta'|\\
			&\leq 2\exp{\left(\sum_{k=j}^{i}\gamma''_{a_{k}}\right)}\cdot\varepsilon_{f}\cdot \exp{\left(\sum_{k=j}^{i}-\frac{\gamma'_{a_{k}}-\gamma''_{a_{k}}+3r}{r-1}\right)}\\
			&=2\varepsilon_{f}\cdot\exp{\left(\sum_{k=j}^{i}\gamma''_{a_{k}}-\frac{\gamma'_{a_{k}}-\gamma''_{a_{k}}+3r}{r-1}\right)}\\
			&=2\varepsilon_{f}\cdot\exp{\left(\sum_{k=j}^{i}\frac{r\cdot\left(\gamma''_{a_{k}}-\frac{\gamma'_{a_{k}}}{r}\right)-3r}{r-1}\right)}\\
			&=2\varepsilon_{f}\cdot\exp{\left(\frac{r}{r-1}\cdot\left(\sum_{k=j}^{i}\gamma''_{a_{k}}-\frac{\gamma'_{a_{k}}}{r}\right)-\frac{3r\left(i-j+1\right)}{r-1}\right)}\\
			&\leq2\varepsilon_{f}\cdot\exp{\left(-\frac{r\cdot\left(i-j+1\right)}{r-1}\right)}\text{\quad by (\ref{ceiling operation})}\\
			&\leq2\varepsilon_{f}\cdot e^{-1}\\
			&\leq\varepsilon_{f}.
		\end{aligned}$$ This shows that $f^{a_{i+1}}\circ \sigma\circ\theta_{i}\circ\theta$ is not only bounded, but $\varepsilon_{f}$-bounded. One can also see from the above estimates the technical role played by the term " $3r$ " in Proposition~\ref{Burguet-Yomdin reparametrization}.

		We define $\Theta_{i+1}:=\{\theta_{i}\circ\theta:~\theta_{i}\in \Theta_{i}, \theta\in \Theta\}$. Then we have  $$\Gamma^{1}_{i+1}=\Gamma^{1}_{i},\quad \Gamma^{2}_{i+1}=\Gamma^{2}_{i},\quad \Gamma^{3}_{i+1}=\Gamma^{3}_{i}\cup\{a_{i+1}\},\quad a_{i+1}-a_{i}=1.$$ Note $$\#\left([m,a_{i+1}]\cap E\right)=\#\left([m,a_{i}]\cap E\right).$$  Hence, we have the desired bound on $\log \# \Theta_{i+1}$:
		$$\begin{aligned}
			\log \# \Theta_{i+1}
			&\leq \log \# \Theta_{i}+\log \# \Theta\\
			&\leq\#\left([m,a_{i+1}]\cap E\right)\cdot h_{\rm top}(f)+ \sum_{k<a_{i+1}}\frac{\gamma'_{k}-\gamma''_{k}}{r-1} + \#\Gamma^{2}_{i+1}\cdot 3B_{p}+ a_{i+1}\cdot 2A_{r}.
		\end{aligned}$$
		The remaining properties asserted in the claim are obvious.
	\end{enumerate}
The proof of the claim is complete.
\end{proof}

We define a subset of $O^{p,m,L}_{n,E,\beta}$ associated with a fixed periodic source $z_s$ by\footnote{In the case $m=0$, we define $O^{p,0,L,s}_{n, E,\gamma}$ by replacing the condition "$x\in B(z_{s}, m, r_{z_{s}})$" in $O^{p,m,L,s}_{n, E,\gamma}$ with "$x\notin B(z_{s}, r_{z_{s}})$", as before.}
$$\begin{aligned}
	O^{p,m,L,s}_{n, E,\beta}:=\Big\{x\in\sigma:~&  \left(E(\hat{x})\cap [0,n)\right)^{L}=E,\,\,x\in B(z_{s}, m, r_{z_{s}}) \text{ and }\\ &  \frac{1}{p}\left\lceil\log\left\|Df^{p}_{f^{i}(x)}\right\|\right\rceil= \beta'_{i}, \,\frac{1}{p}\left\lceil\log\left\|Df^{p}\left(\hat{f}^{i}\left(\hat{x}\right)\right)\right\|\right\rceil=\beta''_{i},\, i\in [m,n)\cap E,\\ & \left\lceil\log\left\|Df_{f^{i}(x)}\right\|\right\rceil=\beta'_{i}, \,\left\lceil\log\left\|Df\left(\hat{f}^{i}\left(\hat{x}\right)\right)\right\|\right\rceil=\beta''_{i},\, i\in [0,n)\setminus\left([m,n)\cap E\right)\Big \}.
\end{aligned}$$
We note that by definition,
\begin{equation}\label{comparison between O and Omega}
	O^{p,m,L,s}_{n, E,\beta}\subset\bigcup_{\gamma}	\Omega^{p,m,L,s}_{n, E,\gamma}
\end{equation} where the union is taken over all possible $\gamma:=\left\{\left(\gamma'_{i},\, \gamma''_{i}\right)\right\}_{i\in [0,n)}$ such that $$\gamma'_{i}=\beta'_{i},\quad \gamma''_{i}=\beta''_{i},\quad i\in [0,n)\setminus\left(\Delta\cup\Gamma_{l_{0}}^{2}\right).$$ There are at most $$\left(\log\left\|Df\right\|+\log\left\|Df^{-1}\right\|+2\right)^{\frac{2n}{L}}$$ such $\gamma$, since both $\Delta$ and $\Gamma_{l_0}^2$ have cardinality at most $n/L$. Indeed, the set $\Gamma_{l_0}^2$ records the right endpoints of geometric intervals, and $\Delta$ has the same cardinality as $\Gamma_{l_0}^2$. The number of geometric intervals (including isolated ones) is relatively small-at most $n/L$-because the complement of $E$ in $[0,n)$ is a disjoint union of intervals whose lengths are all larger than $L$.

As a consequence of (\ref{comparison between O and Omega}) and the claim above, there is a family $\widetilde{\Theta}_{l_{0}}$ of affine maps from $[0,1]$ to itself such that

\begin{enumerate}
	\item  $O^{p,m,L,s}_{n, E,\beta}\subset \cup_{\theta\in \widetilde{\Theta}_{l_{0}}} \sigma\circ\theta$,
	\item for any $\theta\in\widetilde{\Theta}_{l_{0}}$, $f^{n_{E}}\circ\sigma\circ\theta$ is $\varepsilon_{f}$-bounded,
	\item recalling $a_{l_{0}}=n_{E}$,
	\begin{equation}\label{log widetilde{Theta}{l_{0}}}
		\begin{aligned}
		\log\# \widetilde{\Theta}_{l_{0}}
		&\leq\#\left([m,n)\cap E\right)\cdot h_{\rm top}(f)+ \sum_{k\in[0, n_{E})}\frac{\beta'_{k}-\beta''_{k}}{r-1}+ \frac{n}{L}\cdot 3B_{p}+ n\cdot 2A_{r}\\
		&\quad + \frac{2n}{L}\cdot\log \left(\log\left\|Df\right\|+\log\left\|Df^{-1}\right\|+2\right)+\frac{2n}{L}\cdot \frac{\log\left\|Df\right\|+\log\left\|Df^{-1}\right\|+2}{r-1}.
	\end{aligned}
\end{equation}	Note that the last term above $$\frac{2n}{L}\cdot \frac{\log\left\|Df\right\|+\log\left\|Df^{-1}\right\|+2}{r-1}$$ is introduced to compensate for the discrepancy between (at the times $\Delta\cup\Gamma_{l_{0}}^{2}$) $$\sum_{k\in[0, n_{E})}\frac{\gamma'_{k}-\gamma''_{k}}{r-1} \quad\text{ and} \quad\sum_{k\in[0, n_{E})}\frac{\beta'_{k}-\beta''_{k}}{r-1}.$$

\end{enumerate}

\noindent\textbf{\large Step 2: Estimates along the time decomposition}

Next, we estimate from above the quantity: $$\sum_{k\in[0, n_{E})}\frac{\beta'_{k}-\beta''_{k}}{r-1}.$$
\begin{itemize}

	\item {\bf \small Geometric Bound:}  
	
	The estimation of the geometric component $\sum_{k\in [m,n)\cap E}\frac{\beta'_{k}-\beta''_{k}}{r-1}$ is rather delicate. We will show in  Proposition \ref{volume growth for time n} that this term is negligible, using a measure-theoretic argument. We retain this term in  Proposition~\ref{grow rate for major component in finite step n}.

	\item {\bf \small Neutral Bound:}

	Recall that $n_{E}$ is the maximal geometric time in $[0,n)$ (if no such geometric time exists, we set $n_{E}=0$). We consider the times in $[0,n_{E})\setminus E$ \footnote{This part contains the neutral component (i.e., the non-geometric times $[0,m)\setminus E$) of the trapping times $[0,m)$ and is treated here for technical convenience. The estimation of the remaining, geometric  component (i.e., $[0,m)\cap E$) will be carried out in {\bf Trapping Bound} below.}. We decompose $[0,n_{E})\setminus E$  into maximal disjoint intervals: $$[0,n_{E})\setminus E=\bigcup[c_{i}, d_{i}).$$ We note that $d_{i}$ must be a geometric time by maximality and consequently, for any $x\in O^{p,m,L,s}_{n, E,\beta}$, $$0<d_{i}-c_{i}\leq \sum_{k\in [c_{i},d_{i})}\rho'\left(\hat{f}^{k}(\hat{x})\right)\leq 10(d_{i}-c_{i})+\sum_{k\in [c_{i},d_{i})}\beta''_{k}-\frac{\beta'_{k}}{r}.$$ Here the term "$10(d_{i}-c_{i})$" is introduced to compensate for the discrepancy between the ceiling operation and the actual value. The above inequalities imply $$\sum_{k\in [c_{i},d_{i})}-\beta''_{k}\leq 10(d_{i}-c_{i})-\sum_{k\in [c_{i},d_{i})}\frac{\beta'_{k}}{r}.$$ Adding up all the intervals, we get $$\sum_{k\in [0,n_{E})\setminus E}-\beta''_{k}\leq 10\# \left([0,n_{E})\setminus E\right)-\sum_{k\in [0,n_{E})\setminus E}\frac{\beta'_{k}}{r}.$$ Hence, 

	 $$\begin{aligned}
	\sum_{k\in [0,n_{E})\setminus E}\frac{\beta'_{k}-\beta''_{k}}{r-1}
		&\leq \frac{10\# \left([0,n_{E})\setminus E\right)}{r-1}+\sum_{k\in [0,n_{E})\setminus E}\frac{\beta'_{k}}{r}\\
		&\leq \frac{10\# \left([0,n_{E})\setminus E\right)}{r-1}+\# \left([0,n_{E})\setminus E\right)\cdot\frac{\log\left\|Df\right\|}{r}\\
		&=\# \left([0,n_{E})\setminus E\right)\cdot\left(\frac{10}{r-1}+\frac{\log\left\|Df\right\|}{r}\right).
	\end{aligned}$$

\item {\bf \small Trapping Bound:} 

The estimate of $\frac{\beta'_{k}-\beta''_{k}}{r-1}$ for all $k\in [0,m)\setminus E$ has already been addressed in the {\bf Neutral Bound} above. It remains to treat the case where $k\in [0,m)\cap E$. We decompose $[0,m)\cap E$ into maximal disjoint intervals: $$[0,m)\cap E=\left(\bigcup [c_{i}, d_{i}]\right)\bigcup [m_{E}, m-1]$$ where $[m_{E}, m-1]$ denotes the last interval if $m-1, m\in E$ (which means that the geometric interval is truncated at $m-1$).
For each interval $[c_{i}, d_{i})$, by an almost identical argument as in the {\bf Neutral Bound}, we have $$\sum_{k\in [c_{i}, d_{i})}-\beta''_{k}\leq 10(d_{i}-c_{i})-\sum_{k\in [c_{i}, d_{i})}\frac{\beta'_{k}}{r}.$$
Adding up all the intervals, we get $$\sum_{i}\sum_{k\in [c_{i}, d_{i})}-\beta''_{k}\leq 10m - \sum_{i}\sum_{k\in [c_{i}, d_{i})}\frac{\beta'_{k}}{r}.$$

For the last interval $[m_{E}, m-1]$, there is some geometric time $j\in (m-1-L,m-1]$. As a consequence, \begin{equation}\label{tail bound for m}
	\sum_{k\in [m_{E}, m-1]}-\beta''_{k}\leq 10m-\sum_{k\in [m_{E}, m-1]}\frac{\beta'_{k}}{r}+L\cdot \left(\log\left\|Df\right\|+\log\left\|Df^{-1}\right\|+2\right).
\end{equation}

Hence,
	 $$\begin{aligned}
	&\quad\sum_{k\in [0,m)\cap E}\frac{\beta'_{k}-\beta''_{k}}{r-1}\\
	&= \sum_{i}\sum_{k\in [c_{i}, d_{i})}\frac{\beta'_{k}-\beta''_{k}}{r-1}+\sum_{d_{i}}\frac{\beta'_{d_{i}}-\beta''_{d_{i}}}{r-1}+\sum_{k\in [m_{E}, m-1]}\frac{\beta'_{k}-\beta''_{k}}{r-1}\\
	&\leq \frac{20m}{r-1}+\sum_{i}\sum_{k\in [c_{i}, d_{i})}\frac{\beta'_{k}}{r}+\sum_{d_{i}}\frac{\beta'_{d_{i}}-\beta''_{d_{i}}}{r-1}+\sum_{k\in [m_{E}, m-1]}\frac{\beta'_{k}}{r}+L\cdot\frac{ \log\left\|Df\right\|+\log\left\|Df^{-1}\right\|+2}{r-1}\\
	&\leq \#\left([0,m)\cap E\right)\cdot\frac{\log\left\|Df\right\|}{r}\\
	&\quad+\frac{20m}{r-1}+\frac{m}{L}\cdot\frac{\log\left\|Df\right\|+\log\left\|Df^{-1}\right\|+2}{r-1}+L\cdot\frac{ \log\left\|Df\right\|+\log\left\|Df^{-1}\right\|+2}{r-1}.
\end{aligned}$$
Note that, in the last step, we use the fact that there are relatively few (at most $m/L$) right endpoints $d_{i}$ among the maximal intervals in $[0,m)\cap E$, since the complement of $E$ in $[0,m)$ is a disjoint union of intervals whose lengths exceed $L$ (except possibly for the first and the last ones).
\item {\bf \small Tail Bound:}

For the times in $[n_{E},n)$, if $n_{E}<n-1$ which implies that $n-1$ is not a geometric time,  for any $x\in O^{p,m,L,s}_{n, E,\beta}$, by Lemma \ref{property of geometric time}, $$\sum_{k\in [n_{E},n-1)}\rho'\left(\hat{f}^{k}(\hat{x})\right)< n-n_{E}.$$ 

Taking the ceiling operation into consideration, we then have
\begin{equation}\label{direct assumption}
	\sum_{k\in [n_{E},n)}\beta''_{k}\leq \left(n-n_{E}\right)\cdot\frac{\log\left\|Df\right\|}{r} + 2n +10\left(\log\left\|Df\right\|+\log\left\|Df^{-1}\right\|+2\right).
\end{equation} Note that the extra term $10\left(\log\left\|Df\right\|+\log\left\|Df^{-1}\right\|+2\right)$ accounts for the inclusion of the time $n-1$ in the sum and for the mismatch between $\beta''_{n_{E}}-\frac{\beta'_{n_{E}}}{r}$ and $\rho'\!\left(\hat{f}^{n_{E}}(\hat{x})\right)$.

For the tail times, instead of estimating $\sum_{k\in [n_{E},n)}\frac{\beta'_{k}-\beta''_{k}}{r-1}$, we directly consider the volume growth. Given any reparametrization $\theta\in \widetilde{\Theta}_{l_{0}}$, recalling the last paragraph in {\bf Step 1},  $f^{n_{E}}\circ\sigma\circ\theta$ is an $\varepsilon_{f}$-bounded curve and consequently, its volume (length) is less than 1. Define \footnote{We ignore the value at time $n_E$ in the definition of $\Omega_\theta$ for technical convenience, which results in a slightly larger set. Since this omission concerns only a single time instance, it does not affect the asymptotic regime.}
$$\begin{aligned}
	\Omega_{\theta}:=\Big\{t\in [0,1]:~& x:=\sigma\circ\theta (t),\quad\left\lceil\log\left\|Df_{f^{k}(x)}\right\|\right\rceil=\beta'_{k},\\ &  \left\lceil\log\left\|Df\left(\hat{f}^{k}\left(\hat{x}\right)\right)\right\|\right\rceil=\beta''_{k},\quad k\in (n_{E},n)\Big \}.
\end{aligned}$$ 
Note that our focus is solely on the restriction $\sigma\circ\theta|_{\Omega_{\theta}}$. By (\ref{direct assumption}), we have
 $$\begin{aligned}
	&\quad\log{\rm Vol}\left(f^{n}\circ\sigma\circ\theta|_{\Omega_{\theta}}\right)\\
	&\leq\log{\rm Vol}\left(f^{n_{E}}\circ\sigma\circ\theta\right)+\sum_{k\in [n_{E},n)}\beta''_{k}+10\left(\log\left\|Df\right\|+\log\left\|Df^{-1}\right\|+2\right)\\
	&\leq \left(n-n_{E}\right)\cdot\frac{\log\left\|Df\right\|}{r}+2n+20\left(\log\left\|Df\right\|+\log\left\|Df^{-1}\right\|+2\right).
\end{aligned}$$ Note that the term "$10\left(\log\left\|Df\right\|+\log\left\|Df^{-1}\right\|+2\right)$" in the first inequality is introduced to compensate for the mismatch at time $n_{E}$.

\end{itemize}

	\medskip
	\noindent\textbf{\large Step 3: Conclusion.}
	\medskip

By the estimates in {\bf Step 2},
 $$\begin{aligned}
	\sum_{k\in [0,n_{E})}\frac{\beta'_{k}-\beta''_{k}}{r-1}
	&=\sum_{k\in [m,n_{E})\cap E}\frac{\beta'_{k}-\beta''_{k}}{r-1}+\sum_{k\in [0,n_{E})\setminus E}\frac{\beta'_{k}-\beta''_{k}}{r-1} + \sum_{k\in [0,m)\cap E}\frac{\beta'_{k}-\beta''_{k}}{r-1}\\
	&\leq\sum_{k\in [m,n_{E})\cap E}\frac{\beta'_{k}-\beta''_{k}}{r-1}+\left(m + \#\left([m,n_{E})\setminus E\right)\right)\cdot\frac{\log\left\|Df\right\|}{r}\\
	&\quad + \frac{30n}{r-1} +\frac{n}{L}\cdot\frac{\log\left\|Df\right\|+\log\left\|Df^{-1}\right\|+2}{r-1}+L\cdot\frac{ \log\left\|Df\right\|+\log\left\|Df^{-1}\right\|+2}{r-1}.
\end{aligned}$$

Since $\varepsilon_{f}$-bounded curve has volume (length) less than 1, by the {\bf Tail Bound} in {\bf Step 2}, $$\log{\rm Vol}\left(f^{n}\left(O^{p,m,L,s}_{n, E,\beta}\right)\right)\leq \log\# \widetilde{\Theta}_{l_{0}}+ \left(n-n_{E}\right)\cdot\frac{\log\left\|Df\right\|}{r}+2n+20\left(\log\left\|Df\right\|+\log\left\|Df^{-1}\right\|+2\right).$$

By the definition of $n_{E}$,  $$ m + \#\left([m,n_{E})\setminus E\right)+n-n_{E}= m + \#\left([m,n)\setminus E\right).$$
Hence, by (\ref{log widetilde{Theta}{l_{0}}}) we have
		 $$\begin{aligned}
	\log{\rm Vol}\left(f^{n}\left(O^{p,m,L}_{n, E,\beta}\right)\right)
	&\leq \log n + \max_{1\leq s\leq n}\log{\rm Vol}\left(f^{n}\left(O^{p,m,L,s}_{n, E,\beta}\right)\right)\\
	&\leq \log n + \log\# \widetilde{\Theta}_{l_{0}} + \left(n-n_{E}\right)\cdot\frac{\log\left\|Df\right\|}{r}+2n+20\left(\log\left\|Df\right\|+\log\left\|Df^{-1}\right\|+2\right)\\
	&\leq \#\left([m,n)\cap E\right)\cdot h_{\rm top}(f)+\left(m + \#\left([m,n)\setminus E\right)\right)\cdot\frac{\log\left\|Df\right\|}{r}+\sum_{k\in [m,n)\cap E}\frac{\beta'_{k}-\beta''_{k}}{r-1}\\
	&\quad + \frac{n}{L}\cdot 3B_{p} + n\cdot \left(2A_{r}+3\right)+ \frac{30n}{r-1}\\
	&\quad+\left(\frac{3n}{L}\cdot\frac{1}{r-1}+\frac{L}{r-1}+\frac{2n}{L}+20\right)\cdot\left(\log\left\|Df\right\|+\log\left\|Df^{-1}\right\|+2\right).
\end{aligned}$$

Define $$C_{n,L}:=\frac{n}{L}\cdot 3B_{p} + n\cdot \left(2A_{r}+3\right)+ \frac{30n}{r-1} +\left(\frac{3n}{L}\cdot\frac{1}{r-1}+\frac{L}{r-1}+\frac{2n}{L}+20\right)\cdot\left(\log\left\|Df\right\|+\log\left\|Df^{-1}\right\|+2\right)$$ and $$ C_{r}:=2A_{r}+3+\frac{30}{r-1}$$ and note that $$\limsup_{L \to +\infty}\limsup_{n \to +\infty}\frac{C_{n,L}}{n}=C_{r}.$$

The proof of Proposition~\ref{grow rate for major component in finite step n} is then complete.

\end{proof}

\subsection{Estimation on the geometric component and proof of Theorem \ref{main theorem of decomposition of volume growth of submanifolds}}

Recall that given a subset $E\subset \mathbb{N}$ and a number $L\in\mathbb{N}$, we defined
$$E^{L}:=\bigcup [i,j]$$ where the union runs over all $i,j\in E$ with $0\leq j-i\leq L$.

Given a subset $E\subset[0,n)$ (assume that $n$ is given in advance), we write $$\delta_{E}\left(\hat{x}\right):=\frac{1}{n}\sum_{i\in E}\delta_{\hat{f}^{i}(\hat{x})}$$ and we let $\delta_{E}\left(x\right)$ denote the corresponding projection, i.e., $$\delta_{E}\left(x\right):=\frac{1}{n}\sum_{i\in E}\delta_{f^{i}(x)}.$$

The following proposition shows that limits of empirical measures along geometric times are $\hat f$-invariant and project to invariant measures with at least one positive Lyapunov exponent.

Recall the geometric times in Definition \ref{Definition of geometric time} and the maximal Lyapunov exponent of $x$: $$\lambda^{+}(f,x):=\limsup_{n\to+\infty}\frac{1}{n}\log\left\|Df^{n}_{x}\right\| .$$

\begin{Proposition}\label{limit measures are hyperbolic}
		Let $f$ be a $C^1$ diffeomorphism on a compact surface $M$. We consider a doubly indexed family of points $\hat{x}_{n,p}\in \hat{M}$ together with a family of geometric time sets $E_{n}\subset [0,n)$ (common to all indices $p$). Let $0\leq m_{n}\leq n$ be a sequence of integers and assume $$\delta_{[m_{n},n)\cap E_{n}^{L}}\left(\hat{x}_{n,p}\right)\xrightarrow{\text{$\quad n \quad$}}\hat{\mu}^{L}_{p},\quad\quad \hat{\mu}^{L}_{p}\xrightarrow{\text{$\quad L \quad$}}\hat{\mu}_{p},\quad\quad \hat{\mu}_{p}\xrightarrow{\text{$\quad p \quad$}}\hat{\mu}$$ for some Borel measures (not necessarily probability measures) $\hat{\mu}^{L}_{p}, \hat{\mu}_{p}$ and $\hat{\mu}$. Then $\hat{\mu}_{p}$ and $\hat{\mu}$ are $\hat{f}$-invariant. Moreover, for $\mu$-almost every $x$, where $\mu$ is the projection of $\hat{\mu}$, we have $\lambda^{+}(f,x)>0$ (assuming $\mu$ is not the zero measure) and $\hat{\mu}$ is an unstable lift of $\mu$.
\end{Proposition}
\begin{proof}
Given $L\in\mathbb{N}$, we decompose $[m_{n},n)\cap E_{n}^{L}$ into maximal disjoint intervals: $$[m_{n},n)\cap E_{n}^{L}=\bigcup_{i\geq 1}[a_{i}, b_{i}).$$ Note that there are relatively few (at most $n/L$) such intervals $[a_{i}, b_{i})$, as the complement of $E^{L}_{n}$ in $[0, n)$ consists of a disjoint union of intervals with lengths exceeding $L$, except possibly the first and the last intervals. Consequently, given any continuous function $h:\hat{M}\to\mathbb{R}$ and any $n,p\in\mathbb{N}$, $$\big|\delta_{[m_{n},n)\cap E_{n}^{L}}\left(\hat{x}_{n,p}\right)\left(h\right)-\delta_{[m_{n},n)\cap E_{n}^{L}}\left(\hat{x}_{n,p}\right)\left(h\circ\hat{f}\right)\big|\leq\frac{1}{n}\sum_{i}\big|h\left(\hat{f}^{a_{i}}\left(\hat{x}\right)\right)\big|+\big|h\left(\hat{f}^{b_{i}}\left(\hat{x}\right)\right)\big|\leq \frac{2}{L}\cdot\max_{\hat{y}}|h(\hat{y})|.$$ Passing to the limit in $n$, we have $$\big|\hat{\mu}^{L}_{p}\left(h\right)-\hat{\mu}^{L}_{p}\left(h\circ\hat{f}\right)\big|\leq \frac{2}{L}\cdot\max_{\hat{y}}|h(\hat{y})|.$$ Then passing to the limit in $L$ gives the $\hat{f}$-invariance of $\hat{\mu}_{p}$ and the $f$-invariance of $\mu_{p}$. As a consequence, $\hat{\mu}$ and $\mu$ are also invariant.

Next we show that for $\mu$-almost every $x$, $\lambda^{+}(f,x)>0$. We first note the following two facts:
\begin{itemize}
	\item Since $E_{n}^{L}\subset E_{n}^{L+1}$, the above convergence w.r.t. $L$ (i.e., $\hat{\mu}^{L}_{p}\xrightarrow{\text{$\quad L \quad$}}\hat{\mu}_{p}$) is an increasing convergence where \emph{increasing} means for any subset $A$, $$\hat{\mu}^{L}_{p}(A)\leq \hat{\mu}^{L+1}_{p}(A).$$
	\item By the convergences of the measures, we have $$\frac{\#\left([m_{n},n)\cap E_{n}^{L}\right)}{n} \xrightarrow{\text{$\quad n \quad$}} \alpha^{L}\quad \text{ and } \quad \alpha^{L}\xrightarrow{\text{$\quad L \quad$}} \alpha$$ where $\alpha^{L},\alpha\in[0,1]$ are defined (independent of $p$) by $$\alpha^{L}:= \hat{\mu}^{L}_{p}(\hat{M})=\mu^{L}_{p}(M),\quad\alpha:=\hat{\mu}_{p}(\hat{M})=\hat{\mu}(\hat{M})=\mu_{p}(M)=\mu(M).$$ Since we assume that $\mu$ is not the zero measure, we may assume $\alpha^{L}, \alpha>0$. Also note that $\alpha^{L}$ converges increasingly to $\alpha$.
\end{itemize}

Note that by definition, for any point $\hat{x}_{n,p}$ and any integer $i$ with $i, i+1\in E_{n}^{L}$ \footnote{This additional condition $i+1\in E_n^{L}$ is a purely technical adjustment. It compensates for a minor asymmetry in the definition of geometric times: although $n$ itself may be declared a geometric time, the defining inequality only involves sums up to time $n-1$ (see the inequality (\ref{definition of geometric time inequality})). Requiring $i+1\in E_n^{L}$ avoids boundary issues caused by this convention and does not affect any asymptotic estimates.}, there is $1\leq l\leq L$ such that $l+i$ is a geometric time of $\hat{x}$. We then have (recalling $\rho(x, [v]):=\log\left\|Df_{x}\left(v\right)\right\|$) $$\hat{f}^{i}(\hat{x}_{n,p})\in \Omega^{L}:=\left\{\hat{x}\in\hat{M}:~ \exists 1\leq l\leq L \text{ such that } \sum_{m=0}^{l-1}\rho\left(\hat{f}^{m}\left(\hat{x}\right)\right)\geq l\right\}.$$ Note that in the definition of $\Omega^{L}$ above, we use the larger function $\rho$ rather than $\rho'$ for simplicity (since $\rho\geq\rho'$). Note that the number of indices $i$ satisfying $i\in E_n^{L}$ and $i+1\notin E_n^{L}$ is at most $\frac{n}{L}$, since each such $i$ corresponds to the right endpoint of a maximal geometric interval and the complement of $E_n^{L}$ consists of intervals of length at least $L$. As a consequence, for any $n,p\in\mathbb{N}$, we have $$\big|\delta_{[m_{n},n)\cap E_{n}^{L}}\left(\hat{x}_{n,p}\right)\left(\Omega^{L}\right)-\delta_{[m_{n},n)\cap E_{n}^{L}}\left(\hat{x}_{n,p}\right)\left(\hat{M}\right)\big |\leq \frac{1}{L}.$$ By the compactness of $\Omega^{L}$, passing to the limit in $n$, for any $p\in\mathbb{N}$,  $$\big|\hat{\mu}_{p}^{L}\left(\Omega^{L}\right)-\hat{\mu}^{L}_{p}\left(\hat{M}\right)\big|=\big|\hat{\mu}_{p}^{L}\left(\Omega^{L}\right)-\alpha^{L}\big|\leq \frac{1}{L}.$$ Given $\varepsilon>0$, choose $L_{0}\in\mathbb{N}$ large enough such that for any $p\in\mathbb{N}$, $$\hat{\mu}_{p}^{L_{0}}\left(\Omega^{L_{0}}\right)\geq \alpha-\varepsilon.$$ Recalling that $\hat{\mu}^{L}_{p}\xrightarrow{\text{$\quad L \quad$}}\hat{\mu}_{p}$ is an increasing convergence, for any $L\geq L_{0}$ and any $p\in\mathbb{N}$, we have $$\hat{\mu}_{p}^{L}\left(\Omega^{L_{0}}\right)\geq \hat{\mu}_{p}^{L_{0}}\left(\Omega^{L_{0}}\right)\geq \alpha-\varepsilon.$$ By the compactness of $\Omega^{L_{0}}$, passing to the limit in $L$, then in $p$, we have $$\hat{\mu}\left(\Omega^{L_{0}}\right)\geq \alpha-\varepsilon.$$ By the arbitrariness of $\varepsilon$, we have  $$\hat{\mu}\left(\bigcup_{L}\Omega^{L}\right)=\hat{\mu}\left(\hat{M}\right)=\alpha.$$ Therefore, the $\hat{f}$-invariant set $\bigcap_{i\in\mathbb{Z}}\hat{f}^{i}\left(\bigcup_{L}\Omega^{L}\right)$ has $\hat{\mu}$ full measure in which for any $\hat{x}$, by definition, there is an increasing sequence $0=m_{0}<m_{1}<m_{2}<\cdots$ such that $$\sum_{m=m_{i}}^{m_{i+1}-1}\rho\left(\hat{f}^{m}\left(\hat{x}\right)\right)\geq m_{i+1}-m_{i}.$$ Hence, for $\hat{\mu}$-almost every $\hat{x}$, $$\lim_{n\to+\infty}\frac{1}{n}\sum_{i=0}^{n-1}\rho\left(\hat{f}^{i}\left(\hat{x}\right)\right)=\limsup_{n\to+\infty}\frac{1}{n}\sum_{i=0}^{n-1}\rho\left(\hat{f}^{i}\left(\hat{x}\right)\right)\geq 1>0.$$ This implies for $\mu$-almost every $x$, $\lambda^{+}(f,x)>0$ and $\hat{\mu}$ is an unstable lift of $\mu$ (see Section \ref{tangent dynamics}).
\end{proof}

The following proposition provides an upper bound for the volume growth of a sequence of curves, which can be viewed as a prototype of Theorem~\ref{main theorem of decomposition of volume growth of submanifolds}.

Recall the small constant $\varepsilon_{f}$ in Proposition \ref{Burguet-Yomdin reparametrization}.

\begin{Proposition}\label{volume growth for time n}
	Let $f$ be a $C^r$ ($r>1$) diffeomorphism on a compact surface $M$ and let $\sigma_{n}$ be a sequence of $\varepsilon_{f}$-bounded $C^r$ curves. We have $$\limsup_{n \to +\infty}\frac{1}{n}\log{\rm Vol}\left(f^{n}\left(\sigma_{n}\right)\right)\leq \max \left\{h_{\rm top}(f),\,\, \frac{\log\left\|Df\right\|}{r}\right\}+4\log\left(\log\left\|Df\right\|+\log\left\|Df^{-1}\right\|+2\right)+C_{r}$$ where  $C_{r}$ is a universal constant depending only on $r$.
\end{Proposition}
\begin{Remark*}
	We prove Proposition~\ref{volume growth for time n} by reformulating it in a convex combination form. More precisely, it suffices to show that there exists $\alpha \in [0,1]$ such that
	$$\limsup_{n \to +\infty}\frac{1}{n}\log{\rm Vol}\left(f^{n}\left(\sigma_{n}\right)\right)\leq \alpha\cdot h_{\rm top}(f)+\left(1-\alpha\right)\cdot\frac{\log\left\|Df\right\|}{r}+4\log\left(\log\left\|Df\right\|+\log\left\|Df^{-1}\right\|+2\right)+C_{r}.$$ 
	The parameter $\alpha$ arises from the relative proportions of different dynamical time regimes that interpolate between topological entropy and the Yomdin term. One can see from the proof that $\alpha$ corresponds, for some point $x$, to the proportion of geometric times (see Definition~\ref{Definition of geometric time}) after excluding the times spent near periodic sources (which are referred to as \emph{trapping times} in the proof of Proposition \ref{volume growth for time n}).
\end{Remark*}
\medskip
\medskip
\begin{proof}[Proof of Proposition \ref{volume growth for time n}]

	Recall that the periodic sources are denoted by $z_{1}, z_{2}, z_{3}, \cdots$ with the corresponding radii $r_{z_{i}}$ (see Definition \ref{Definition of periodic sources}).

	Given $n\in\mathbb{N}$ and a point $x\in\sigma_{n}$, we define the \emph{trapping time} of $x$ by $$t_{n}(x):=\max\left\{0< k\leq n:~ \exists 1\leq s\leq n, \text{ s.t. } x\in B\left(z_{s}, k,r_{z_{s}}\right)\right\}$$ and set $t_{n}(x):=0$ if for any $1\leq s\leq n$, $$x\notin B\left(z_{s}, r_{z_{s}}\right).$$ Given $0\leq m\leq n$, define $$\sigma^{m}_{n}:=\{x\in\sigma_{n}:~ t_{n}(x)=m\}$$ and choose any $m_{n}$ (may not be unique) such that $${\rm Vol}\left(f^{n}\left(\sigma^{m_{n}}_{n}\right)\right)=\max_{0\leq m\leq n}{\rm Vol}\left(f^{n}\left(\sigma^{m}_{n}\right)\right).$$

	We define two partitions:
	\begin{itemize}
		\item Given $n\in\mathbb{N}$ and $E_{n}\subset[0,n)$, define $$\Lambda_{E_{n}}:=\left\{x\in\sigma^{m_{n}}_{n}:~E(\hat{x})\cap [0,n)=E_{n}\right\}$$ where, as previously defined, $E(\hat{x})$ is the set of geometric times of $\hat{x}$. Let $\mathcal{E}_{n}$ be the partition of $\sigma^{m_{n}}_{n}$ with such elements $\Lambda_{E_{n}}$ for all possible $E_{n}$. Note that for the cardinality, we have $$\#\mathcal{E}_{n}\leq 2^{n}.$$ We fix an element $\Lambda_{E_{n}}\in\mathcal{E}_{n}$ that achieves the maximal volume growth among all possible $\widetilde{E}_{n}$, i.e., $${\rm Vol}\left(f^{n}\left(\Lambda_{E_{n}}\right)\right)=\max_{\widetilde{E}_{n}}{\rm Vol}\left(f^{n}\left(\Lambda_{\widetilde{E}_{n}}\right)\right).$$
		\item  Given $p\in\mathbb{N}$ and a family of integers
		$$\gamma:=\left\{\left(\gamma'_{i},\, \hat{\gamma}'_{i},\, \gamma''_{i},\, \hat{\gamma}''_{i}\right)\right\}_{i\in [0,n)},$$ define
		$$\begin{aligned}
			\Lambda^{p,m_{n}}_{n, E_{n},\gamma}:=\Big\{x\in\Lambda_{E_{n}}:~\,i\in [0,n),\,\,\,&  \frac{1}{p}\left\lceil\log\left\|Df^{p}_{f^{i}(x)}\right\|\right\rceil= \gamma'_{i}, \,\frac{1}{p}\left\lceil\log\left\|Df^{p}\left(\hat{f}^{i}\left(\hat{x}\right)\right)\right\|\right\rceil=\hat{\gamma}'_{i},\,\\ & \left\lceil\log\left\|Df_{f^{i}(x)}\right\|\right\rceil=\gamma''_{i}, \,\left\lceil\log\left\|Df\left(\hat{f}^{i}\left(\hat{x}\right)\right)\right\|\right\rceil=\hat{\gamma}''_{i}\, \Big \}.
		\end{aligned}$$
		Let $\mathcal{E}_{n, E_{n}}^{p,m_{n}}$ be the partition of $\Lambda_{E_{n}}$ with such elements  $\Lambda^{p,m_{n}}_{n, E_{n},\gamma}$ for all possible $\gamma$.  We note two facts:
		\begin{itemize}
			\item For the cardinality, we have $$\#\mathcal{E}_{n, E_{n}}^{p,m_{n}}\leq \left(\log\left\|Df\right\|+\log\left\|Df^{-1}\right\|+2\right)^{4n}.$$
			\item For any $L\in \mathbb{N}$, by definitions (see the definition of $O^{p,m_{n}, L}_{n, E^{L}_{n},\beta}$ in front of Proposition \ref{grow rate for major component in finite step n}), $$\Lambda^{p,m_{n}}_{n, E_{n},\gamma}\subset O^{p,m_{n}, L}_{n, E^{L}_{n},\beta}$$ where the corresponding $\beta=\beta(\gamma):=\left\{\left(\beta'_{i},\, \beta''_{i}\right)\right\}_{i\in [0,n)}$ is defined by
			\begin{equation}\label{gamma generates beta}
				\beta'_{i}:=\gamma'_{i},\,\,\beta''_{i}:=\hat{\gamma}'_{i}, \,\, i\in[m_{n},n)\cap E_{n}^{L} \,\,\text{ and }\,\, \beta'_{i}:=\gamma''_{i},\,\,\beta''_{i}:=\hat{\gamma}''_{i}, \,\, i\in[0,n)\setminus\left([m_{n},n)\cap E_{n}^{L}\right).
			\end{equation}
		\end{itemize}

	\end{itemize}

	For each $p$, we fix an element $\Lambda^{p,m_{n}}_{n, E_{n},\gamma}\in\mathcal{E}_{n, E_{n}}^{p,m_{n}}$ that achieves the maximal volume growth, i.e., $${\rm Vol}\left(f^{n}\left(\Lambda^{p,m_{n}}_{n, E_{n},\gamma}\right)\right)=\max_{\widetilde{\gamma}}\,{\rm Vol}\left(f^{n}\left(\Lambda^{p,m_{n}}_{n, E_{n},\widetilde{\gamma}}\right)\right).$$

	We fix any point $x_{n,p}\in \Lambda^{p,m_{n}}_{n, E_{n},\gamma}$. Up to considering sub-sequences in $n, L,p$, we assume
	\begin{equation}\label{definitions for measures}
		\delta_{[m_{n},n)\cap E_{n}^{L}}\left(\hat{x}_{n,p}\right)\xrightarrow{\text{$\quad n \quad$}}\hat{\mu}^{L}_{p},\quad\quad \hat{\mu}^{L}_{p}\xrightarrow{\text{$\quad L \quad$}}\hat{\mu}_{p},\quad\quad \hat{\mu}_{p}\xrightarrow{\text{$\quad p \quad$}}\hat{\mu}
	\end{equation} for some Borel measures (not necessarily probability measures) $\hat{\mu}^{L}_{p}, \hat{\mu}_{p}$ and $\hat{\mu}$ (denote by $\mu^{L}_{p}, \mu_{p}$ and $\mu$ the corresponding projections on $M$). By the convergences of these measures, we have $$\frac{\#\left([m_{n},n)\cap E_{n}^{L}\right)}{n} \xrightarrow{\text{$\quad n \quad$}} \alpha^{L}\quad \text{ and } \quad \alpha^{L}\xrightarrow{\text{$\quad L \quad$}} \alpha$$ where $\alpha^{L},\alpha\in[0,1]$ are defined (independent of $p$) by $$\alpha^{L}:= \hat{\mu}^{L}_{p}(\hat{M})=\mu^{L}_{p}(M),\quad\alpha:=\hat{\mu}_{p}(\hat{M})=\hat{\mu}(\hat{M})=\mu_{p}(M)=\mu(M).$$
	By Proposition \ref{grow rate for major component in finite step n}, for any $p,L\in\mathbb{N}$, recalling that $\gamma$ is determined by $\Lambda_{E_n}$ and $p$, and that $\beta$ is determined by $\gamma$ as (\ref{gamma generates beta}),
	\begin{equation}\label{f n sigma volume intermediate step}
		\begin{aligned}
		\log{\rm Vol}\left(f^{n}\left(\sigma_{n}\right)\right)
		&\leq \log \left(n+1\right)+ \log{\rm Vol}\left(f^{n}\left(\sigma^{m_{n}}_{n}\right)\right)\\
		&\leq \log \left(n+1\right)+ \log\#\mathcal{E}_{n}+ \log{\rm Vol}\left(f^{n}\left(\Lambda_{E_{n}}\right)\right)\\
		&\leq \log \left(n+1\right)+ \log\#\mathcal{E}_{n}+\log\#\mathcal{E}_{n,E_{n}}^{p,m_{n}}+\log{\rm Vol}\left(f^{n}\left(\Lambda^{p,m_{n}}_{n, E_{n},\gamma}\right)\right)\\
		&\leq \log \left(n+1\right)+ \log\#\mathcal{E}_{n}+\log\#\mathcal{E}_{n,E_{n}}^{p,m_{n}}+\log{\rm Vol}\left(f^{n}\left(O^{p,m_{n}, L}_{n, E^{L}_{n},\beta}\right)\right)\\
		&\leq \log \left(n+1\right) + n\log2+4n\log\left(\log\left\|Df\right\|+\log\left\|Df^{-1}\right\|+2\right)+\#\left([m_{n},n)\cap E^{L}_{n}\right)\cdot h_{\rm top}(f)\\
		&+\left(m_{n}+\#\left([m_{n},n)\setminus E^{L}_{n}\right)\right)\cdot\frac{\log\left\|Df\right\|}{r} + \sum_{k\in [m_{n},n)\cap E^{L}_{n}}\frac{\beta'_{k}-\beta''_{k}}{r-1}+ C_{n,L}
	\end{aligned}
	\end{equation} where the error term $C_{n,L}$ satisfies $$\limsup_{L \to +\infty}\limsup_{n \to +\infty}\frac{C_{n,L}}{n}=C_{r}$$ for some universal constant $C_{r}$ depending only on $r$.

	Note that by definition, recalling $\rho(x, [v]):=\log\left\|Df_{x}\left(v\right)\right\|$, $$\frac{1}{n}\sum_{k\in [m_{n},n)\cap E^{L}_{n}}\frac{\beta'_{k}-\beta''_{k}}{r-1}\leq \frac{1}{r-1}\int\frac{\log\left\|Df^{p}_{(\cdot)}\right\|-\sum_{i=0}^{p-1}\rho\circ\hat{f}^{i}}{p} d\,\left(\delta_{[m_{n},n)\cap E_{n}^{L}}\left(\hat{x}_{n,p}\right)\right)+10.$$ Here the term "$10$" is introduced to compensate for the average discrepancy between the ceiling operation and the actual value.

	Putting $\frac{1}{n}$ on both sides in (\ref{f n sigma volume intermediate step}), we have

\begin{equation}\label{frac 1 n f n sigma volume intermediate step}
		\frac{1}{n}\log{\rm Vol}\left(f^{n}\left(\sigma_{n}\right)\right)\leq \alpha_{n,L}\cdot h_{\rm top}(f)+\left(1-\alpha_{n,L}\right)\cdot\frac{\log\left\|Df\right\|}{r}+\gamma_{n,L,p}
\end{equation} where $$\alpha_{n,L}:=\frac{\#\left([m_{n},n)\cap E^{L}_{n}\right)}{n}$$ and
	$$\begin{aligned}
\gamma_{n,L,p}
	&:= \frac{\log \left(n+1\right)}{n} + \log2+4\log\left(\log\left\|Df\right\|+\log\left\|Df^{-1}\right\|+2\right)+\frac{C_{n,L}}{n}+10\\
	&\quad+\frac{1}{r-1}\int\frac{\log\left\|Df^{p}_{(\cdot)}\right\|-\sum_{i=0}^{p-1}\rho\circ\hat{f}^{i}}{p} d\,\left(\delta_{[m_{n},n)\cap E_{n}^{L}}\left(\hat{x}_{n,p}\right)\right).
	\end{aligned}$$

	We take $\limsup_{n \to +\infty}$, then take $\limsup_{L \to +\infty}$ on both sides in (\ref{frac 1 n f n sigma volume intermediate step}): $$\limsup_{n \to +\infty}\frac{1}{n}\log{\rm Vol}\left(f^{n}\left(\sigma_{n}\right)\right)\leq \alpha\cdot h_{\rm top}(f)+\left(1-\alpha\right)\cdot\frac{\log\left\|Df\right\|}{r}+\limsup_{L\to+\infty}\limsup_{n \to +\infty}\gamma_{n,L,p}$$ where $$\alpha:=\limsup_{L\to+\infty}\limsup_{n \to +\infty}\alpha_{n,L}.$$ Consequently, for any $p\in\mathbb{N}$, $$\limsup_{n \to +\infty}\frac{1}{n}\log{\rm Vol}\left(f^{n}\left(\sigma_{n}\right)\right)\leq \max \left\{h_{\rm top}(f),\,\, \frac{\log\left\|Df\right\|}{r}\right\}+\limsup_{L\to+\infty}\limsup_{n \to +\infty}\gamma_{n,L,p}.$$
	
	It remains to show the following claim.

	\begin{Claim} $\gamma_{n,L,p}$ satisfies $$\limsup_{p\to+\infty}\limsup_{L\to+\infty}\limsup_{n \to +\infty}\gamma_{n,L,p}\leq 4\log\left(\log\left\|Df\right\|+\log\left\|Df^{-1}\right\|+2\right)+C_{r}$$ for some universal constant $C_{r}$ \footnote{The constant $C_{r}$ was defined earlier. We reuse the same notation here for simplicity. This causes no ambiguity, since the present $C_{r}$ may differ from the previous one only by a constant depending only on $r$.} depending only on $r$.
	\end{Claim}
	\begin{proof}[Proof of Claim]
		Let us first show the following two properties:
		\begin{itemize}
			\item $\hat{\mu}_{p}$ and $\hat{\mu}$ defined in (\ref{definitions for measures}) are invariant measures,
			\item $\hat{\mu}$ is the unique unstable lift of $\mu$, provided that $\mu$ is not the zero measure.
		\end{itemize}
		To see this, by Proposition \ref{limit measures are hyperbolic},
		\begin{itemize}
		\item $\hat{\mu}_{p}, \hat{\mu}$ are invariant measures,
		\item $\mu$-almost every $x$ has at least one positive Lyapunov exponent (i.e., $\lambda^{+}(f,x)>0$),
		\item $\hat{\mu}$ is an unstable lift of $\mu$.
		\end{itemize} We claim that $\mu$-almost every $x$ has a unique positive Lyapunov exponent (hence $\hat{\mu}$ is the unique unstable lift of $\mu$), due to the choice of $m_{n}$.   Indeed, going by contradiction, if not, $\mu$ would have an ergodic component with two positive Lyapunov exponents and this ergodic component must be supported on the orbits of some periodic source $z$ with $\mu$ positive measure (see \cite[Proposition 4.4]{Pol93} or \cite[Fact in Claim 7.3]{BCS22}). 
		
		Let $r>0$ be small enough such that if $y\notin \cup_{i} B\left(f^{i}(z), r_{z}\right)$,  then $f^{k}(y)\notin B(z, r)$ for all $k\in\mathbb{N}$. Indeed, it is enough to choose $r>0$ small enough such that for $0\leq j \leq T_{z}-1$ ($T_{z}$ denotes the period of $z$), \begin{equation}\label{choice of r}
			f^{j}\left(B\left(z,r\right)\right)\subset B\left(f^{j}(z),r_{z}\right).
		\end{equation} To see this, we write $k=m\cdot T_{z}-j$ for some $0\leq j \leq T_{z}-1$. Since $y\notin B\left(f^{j}(z), r_{z}\right)$, by the choice of $r_{z}$ (see the second property in Definition \ref{Definition of periodic sources}), $$f^{j}\circ f^{k}(y)=f^{mT_{z}}(y)\notin B\left(f^{j}(z), r_{z}\right).$$ By (\ref{choice of r}), $$f^{k}(y)\notin B\left(z, r\right).$$

		Since $\mu$ is not the zero measure, we may assume for all large $n$, $$m_{n}+100<n.$$

		Since $x_{n,p}\in\sigma_{n}^{m_{n}}$, by the definition of $m_{n}$, there must exist some $0\leq i\leq m_{n}$ such that  $$f^{i-1}(x_{n,p})\in B\left(f^{i-1}(z), r_{z}\right),\quad f^{i}(x_{n,p})\notin B\left(f^{i}(z), r_{z}\right)\footnote{There is also a degenerate case where $x_{n,p}\notin B\left(z, r_{z}\right)$ which corresponds to the case $m_{n}=0$.}.$$ By the choice of $r_{z}$ (see the third property in Definition \ref{Definition of periodic sources}), we then have $$f^{i}(x_{n,p})\notin \bigcup_{j}B\left(f^{j}(z), r_{z}\right).$$ As a consequence, for any $k\in [m_{n},n)\cap E_{n}^{L}$, $$f^{k}(x_{n,p})\notin B(z, r).$$ This implies for large $n$, $$\delta_{[m_{n},n)\cap E_{n}^{L}}\left(x_{n,p}\right)\left(B(z,r)\right)=0$$ which is a contradiction. Hence, $\mu$-almost every $x$ must have a unique positive Lyapunov exponent and  $\hat{\mu}$ is the unique unstable lift of $\mu$.

		 By taking limit in $n$, then in $L$, we have

		\begin{equation*}
			\begin{aligned}
				&\quad\limsup_{L\to+\infty}\limsup_{n \to +\infty}\int\frac{\log\left\|Df^{p}_{(\cdot)}\right\|-\sum_{i=0}^{p-1}\rho\circ\hat{f}^{i}}{p} d\,\left(\delta_{[m_{n},n)\cap E_{n}^{L}}\left(\hat{x}_{n,p}\right)\right)\\
				&= \int\frac{\log\left\|Df^{p}_{(\cdot)}\right\|-\sum_{i=0}^{p-1}\rho\circ\hat{f}^{i}}{p} d\,\hat{\mu}_{p} .\\
				&=\int\frac{\log\left\|Df^{p}_{(\cdot)}\right\|}{p} d\,\mu_{p}- \int \rho d\,\hat{\mu}_{p}.
			\end{aligned}
		\end{equation*}
		Taking limit in $p$ on both sides, by Lemma \ref{upper continuity of subadditive functions}, we then have
			$$\begin{aligned}
				&\quad\limsup_{p\to+\infty}\limsup_{L\to+\infty}\limsup_{n \to +\infty}\int\frac{\log\left\|Df^{p}_{(\cdot)}\right\|-\sum_{i=0}^{p-1}\rho\circ\hat{f}^{i}}{p} d\,\left(\delta_{[m_{n},n)\cap E_{n}^{L}}\left(\hat{x}_{n,p}\right)\right)\\
			&\leq \int\limsup_{p\to+\infty}\frac{\log\left\|Df^{p}_{(\cdot)}\right\|}{p} d\,\mu- \int \rho d\,\hat{\mu}.
		\end{aligned}$$
		Since $\hat{\mu}$ is the unique unstable lift of $\mu$, we have (see Section \ref{tangent dynamics}) $$\int\limsup_{p\to+\infty}\frac{\log\left\|Df^{p}_{(\cdot)}\right\|}{p} d\,\mu- \int \rho d\,\hat{\mu}=0 \footnote{If $\mu$ is the zero measure (e.g., if $\frac{m_n}{n}\to 1$), then the equality holds trivially.}.$$ Hence, $$\limsup_{p\to+\infty}\limsup_{L\to+\infty}\limsup_{n \to +\infty}\gamma_{n,L,p}\leq 4\log\left(\log\left\|Df\right\|+\log\left\|Df^{-1}\right\|+2\right)+\log 2+10+C_{r}.$$ We can put $\log 2+10$ into $C_{r}$ which completes the proof of the claim.

	\end{proof}
The proof of Proposition \ref{volume growth for time n} is then complete.
\end{proof}

Now we use Proposition \ref{volume growth for time n} to prove Theorem \ref{main theorem of decomposition of volume growth of submanifolds}.
\begin{proof}[Proof of Theorem \ref{main theorem of decomposition of volume growth of submanifolds}]

Write
$$\chi(f):=\max \left\{h_{\rm top}(f),\,\, \frac{\log\left\|Df\right\|}{r}\right\}+4\log\left(\log\left\|Df\right\|+\log\left\|Df^{-1}\right\|+2\right)+C_r$$ where $C_{r}$ is the universal constant depending only on $r$ from Proposition \ref{volume growth for time n}.

We claim \begin{equation}\label{uniform volume growth}
	\limsup_{n\to+\infty}\frac{1}{n}\log\sup_{\sigma}{\rm Vol}\left(f^n(\sigma)\right)\leq \chi(f)
\end{equation} where, for any $\varepsilon\in (0,1)$, the supremum can be taken over the family of all $C^{r}$ embedded curves  $\sigma: [0,1]\to M$ with $\left\|\sigma\right\|_{C^{r}}\leq 1$ and $\left\|d_{t}\sigma\right\| \geq \varepsilon$ for all $t\in[0,1]$. Assume by contradiction that
$$\limsup_{n\to+\infty}\frac{1}{n}\log\sup_{\sigma}{\rm Vol}\left(f^n(\sigma)\right)>\chi(f).$$ Then there exist $n_k\to+\infty$ and $C^r$ embedded curves
$\sigma_{n_k}$ with
$\left\|\sigma\right\|_{C^{r}}\leq 1$ and $\left\|d_{t}\sigma\right\| \geq \varepsilon, \,t\in[0,1]$ such that $$\limsup_{k\to+\infty}\frac{1}{n_k}\log{\rm Vol}\left(f^{n_k}(\sigma_{n_k})\right)>\chi(f).$$
Since each $\sigma_{n_k}$ can be decomposed into a uniformly bounded number of $\varepsilon_f$-bounded curves (see Lemma~\ref{Decomposition into varepsilon bounded curves}), we may, without loss of generality, assume that the sequence $\{\sigma_{n_k}\}$ itself consists of $\varepsilon_f$-bounded curves.
 Hence, by Proposition \ref{volume growth for time n},
$$\limsup_{k\to+\infty}\frac{1}{n_k}\log{\rm Vol}\left(f^{n_k}(\sigma_{n_k})\right)\leq \chi(f),$$
which is a contradiction.

\begin{Claim}
For any $q\geq 1$, $$\limsup_{n\to+\infty}\frac{1}{n}\log\sup_{\sigma}{\rm Vol}\left(f^n(\sigma)\right)=\limsup_{n\to+\infty}\frac{1}{qn}\log\sup_{\sigma}{\rm Vol}\left(f^{qn}(\sigma)\right).$$
\end{Claim}
\begin{proof}[Proof of Claim]
	Set
	$$ a_n :=\sup_{\sigma}{\rm Vol}\left(f^n(\sigma)\right).$$
	The inequality "$\geq$" follows since $\{qn\}$ is a sub-sequence of $\mathbb{N}$.
	For the reverse inequality, write $n=qk+m$ with $0\leq m<q$. Since $f^m$ has
	uniformly bounded $C^1$ norm for $0\leq m<q$, there exists $C>0$ such that
	$$a_n \leq C\, a_{qk}.$$
	Hence,
	$$\frac{1}{n}\log a_n\le\frac{qk}{n}\cdot\frac{1}{qk}\log a_{qk}+\frac{1}{n}\log C,$$
	and letting $n\to+\infty$ yields the result.
\end{proof}

Replacing $f$ by $f^{q}$ and putting $\frac{1}{q}$ on both sides in (\ref{uniform volume growth}), by the arbitrariness of $q$, we get $$\limsup_{n\to+\infty}\frac{1}{n}\log\sup_{\sigma}{\rm Vol}\left(f^n(\sigma)\right)\leq \max \left\{h_{\rm top}(f),\,\, \frac{\lambda^{+}(f)}{r}\right\}.$$
The proof of Theorem \ref{main theorem of decomposition of volume growth of submanifolds} is then complete.
\end{proof}

	\section{Volume growth of the tangent cocycle}\label{Volume growth of the tangent cocycle}
	In this section, we study the relationship between the volume growth of the tangent cocycle and entropy. More precisely, we establish the chain of inequalities:
	$$h_{\rm top}(f)\leq\liminf_{n\to+\infty}\frac{1}{n}\log\int_{M}\left\|Df^{n}_{x}\right\|\,dx\leq  \limsup_{n\to+\infty}\frac{1}{n}\log\int_{M}\left\|Df^{n}_{x}\right\|\,dx\leq \max \left\{h_{\rm top}(f),\,\, \frac{\lambda^{+}(f)}{r}\right\}.$$
	The rightmost inequality is Theorem~\ref{main theorem of decomposition of volume growth of derivatives}, while the leftmost inequality will be proved in Theorem~\ref{volume of tangent cocycles bounds entropy}.

\subsection{Upper bound of the volume growth rate}\label{Upper bound of the volume growth rate}
The following lemma asserts the existence of a point $z$ at which the exponential growth rate of $\int_{M}\left\|Df^{n}_{x}\right\|$ can be detected on arbitrarily small neighborhoods of $z$.
\begin{Lemma}\label{existence of a maximal point}
	Let $f$ be a $C^{1}$ diffeomorphism on a compact manifold $M$. There is a point $z$ such that for any neighborhood $U$ of $z$, $$\limsup_{n\to+\infty}\frac{1}{n}\log\int_{U}\left\|Df^{n}_{x}\right\|\,dx=\limsup_{n\to+\infty}\frac{1}{n}\log\int_{M}\left\|Df^{n}_{x}\right\|\,dx.$$
\end{Lemma}
\begin{proof}
	Write $$\lambda(U) := \limsup_{n\to+\infty} \frac{1}{n} \log \int_{U} \left\|Df^n_x\right\| \, dx.$$ Suppose for contradiction that for every $z \in M$, there exists an open neighborhood $U_z$ such that $\lambda(U_z) < \lambda(M)$.

	By the compactness of $M$, there exists a finite sub-cover $\{U_{z_i}\}_{i=1}^k$ such that $M = \bigcup_{i=1}^k U_{z_i}$. Using the property that for non-negative sequences $a_{n,i}$,
	$$ \limsup_{n\to+\infty} \frac{1}{n} \log \left( \sum_{i=1}^k a_{n,i} \right) = \max_{1 \leq i \leq k} \left( \limsup_{n\to+\infty} \frac{1}{n} \log a_{n,i} \right), $$
	we have:
	$$ \lambda(M) = \lambda\left( \bigcup_{i=1}^k U_{z_i} \right) = \max_{1 \leq i \leq k} \lambda(U_{z_i})<\lambda(M)$$ which is a contradiction.
\end{proof}

The main idea to prove Theorem~\ref{main theorem of decomposition of volume growth of derivatives} is to reduce the estimation of the volume growth of the tangent cocycle to that of one-dimensional submanifolds via Fubini's theorem, which allows the two-dimensional integral of $\left\|Df^{n}_{x}\right\|$ to be decomposed into integrals along horizontal and vertical lines. The conclusion then follows by applying the curve volume growth estimate in Theorem \ref{main theorem of decomposition of volume growth of submanifolds}.
\begin{proof}[Proof of Theorem \ref{main theorem of decomposition of volume growth of derivatives}]

Let $\mathcal{A}:=\left\{\left(U_{i},\varphi_{i}\right):~\varphi_{i}:U_{i}\to[0,1]^{2}\right\}_{0\leq i<N}$ be a finite family of charts that cover $M$. Let $z$ be the point in Lemma \ref{existence of a maximal point}. Without loss of generality, we assume $z\in U_{0}$  and $U_{0}$ is a ball at $z$ with sufficiently small diameter such that for any line $\gamma$ in $[0,1]^{2}$, $\varphi^{-1}_{0}\left(\gamma\right)$ is contained in some $C^{r}$ embedded curve $\sigma: [0,1]\to M$ with $\left\|\sigma\right\|_{C^{r}}\leq 1$ and $\left\|d_{t}\sigma\right\| \geq 1/2, \,t\in[0,1]$. Let $K$ be a sufficiently large number (independent of $n$) such that

\begin{itemize}
	\item for any $0\leq i< N$ and any curve $\sigma$ on $U_{i}$, $$K^{-1}\cdot {\rm Vol}\left(\varphi_{i}\circ\sigma\right)\leq {\rm Vol}\left(\sigma\right)\leq K\cdot {\rm Vol}\left(\varphi_{i}\circ\sigma\right).$$
	
	\item  for any $0\leq i< N$ and any $n\geq 1$,
	\begin{equation}\label{volume growth in local coordinates}
		\int_{U_{0}^{i}}\left\|Df^{n}_{x}\right\|\,dx\leq K \int_{\varphi_{0}\left( U_{0}^{i}\right)}\left\|DF^{n}_{\left(x_{1},x_{2}\right)}\right\|\,dx_{1}dx_{2}
	\end{equation} where $\left(x_{1}, x_{2}\right)$ denotes the standard coordinates in $\mathbb{R}^{2}$ and $$F^{n}:=\varphi_{i}\circ f^{n}\circ\varphi^{-1}_{0},\quad U_{0}^{i}:=\left\{x\in U_{0}:~ f^{n}\left(x\right)\in U_{i}\right\}.$$
	
\end{itemize}

Given $n\in\mathbb{N}$, we fix any $0\leq i<N$ such that
\begin{equation}\label{major U0 i}
\int_{U_{0}}\left\|Df^{n}_{x}\right\|\,dx \leq N\cdot \int_{U_{0}^{i}}\left\|Df^{n}_{x}\right\|\,dx.
\end{equation} 

Given a $C^{1}$ map $G:=(G_{1}, G_{2})$ on $\mathbb{R}^{2}$, denote partial derivative vectors of $G$ by $$\partial_{1}G_{(x_{1},x_{2})}:=\left(\frac{\partial G_{1}}{\partial x_{1}}\left(x_{1},x_{2}\right), \frac{\partial G_{2}}{\partial x_{1}}\left(x_{1},x_{2}\right)\right),\quad \partial_{2}G_{(x_{1},x_{2})}:=\left(\frac{\partial G_{1}}{\partial x_{2}}\left(x_{1},x_{2}\right), \frac{\partial G_{2}}{\partial x_{2}}\left(x_{1},x_{2}\right)\right).$$  Note that, by standard linear algebra, $$\left\|DG_{\left(x_{1},x_{2}\right)}\right\|\leq\left\|\partial_{1}G_{(x_{1},x_{2})}\right\|+\left\|\partial_{2}G_{(x_{1},x_{2})}\right\|$$ where $\left\|\cdot\right\|$ denotes the usual operator norm for matrices and the usual Euclidean norm for vectors.  Consider the foliation of horizontal lines
$$
\left\{[0,1]\times x_{2}:\ x_{2}\in[0,1]\right\}.
$$

By Fubini's theorem and by the choice of $K$, we have

$$\begin{aligned}
	\int_{\varphi_{0}\left(U_{0}^{i}\right)}
	\left\|\partial_{1} F^{n}_{\left(x_{1},x_{2}\right)}\right\|\,dx_{1}dx_{2}
	&=\int_{0}^{1}
	\left(
	\int_{\varphi_{0}\left(U_{0}^{i}\right)\cap([0,1]\times x_{2})}
	\left\|\partial_{1} F^{n}_{\left(x_{1},x_{2}\right)}\right\|\,dx_{1}
	\right)dx_{2}\\
	&=\footnotemark\int_{0}^{1}
	{\rm Vol}\left(F^{n}\left(\varphi_{0}\left(U_{0}^{i}\right)\cap\left([0,1]\times x_{2}\right)\right)\right)dx_{2}\\
	&\leq \max_{0\leq x_{2}\leq 1}{\rm Vol}\left(F^{n}\left(\varphi_{0}\left(U_{0}^{i}\right)\cap\left([0,1]\times x_{2}\right)\right)\right)\\
	&\leq K\cdot \max_{0\leq x_{2}\leq 1}{\rm Vol}\left(f^{n}\left(U_{0}^{i}\cap\varphi_{0}^{-1}\left([0,1]\times x_{2}\right)\right)\right)\\
	&\leq K\cdot \max_{0\leq x_{2}\leq 1}{\rm Vol}\left(f^{n}\left(\varphi_{0}^{-1}\left([0,1]\times x_{2}\right)\right)\right).
\end{aligned}$$\footnotetext{Consider a $C^{1}$ diffeomorphism $G:=(G_{1}, G_{2})$ on $\mathbb{R}^{2}$. Fix $x_{2}\in[0,1]$ and consider the $C^{1}$ curve
$\gamma_{x_{2}}(x_{1}):=(x_{1},x_{2})$.
Then $G\circ\gamma_{x_{2}}(x_{1})=G(x_{1},x_{2})$ and
$$
d_{x_{1}}\big(G\circ\gamma_{x_{2}}\big)
=
\left(\frac{\partial G_{1}}{\partial x_{1}}(x_{1},x_{2}),
\frac{\partial G_{2}}{\partial x_{1}}(x_{1},x_{2})\right)
=:
\partial_{1}G_{(x_{1},x_{2})}.
$$
By the length formula for $C^{1}$ curves, the length of
$G(\gamma_{x_{2}})$ is therefore given by
$$
{\rm Vol}\big(G(\gamma_{x_{2}})\big)
=
\int \left\|\partial_{1}G_{(x_{1},x_{2})}\right\|\,dx_{1}.
$$}
Similarly, by considering vertical lines, $$\int_{\varphi_{0}\left(U_{0}^{i}\right)}
\left\|\partial_{2} F^{n}_{\left(x_{1},x_{2}\right)}\right\|\,dx_{1}dx_{2}\leq K\cdot \max_{0\leq x_{1}\leq 1}{\rm Vol}\left(f^{n}\left(\varphi_{0}^{-1}\left(x_{1}\times[0,1] \right)\right)\right).$$ We define $\sigma_n$ to be the pre-image under $\varphi_0$ of a horizontal or vertical line in $[0,1]^2$ whose image under $f^n$ attains the maximal volume among all such coordinate lines above, namely, $${\rm Vol}\left(f^{n}\left(\sigma_{n}\right)\right)=\max_{0\leq x_{1},x_{2}\leq 1}\left\{{\rm Vol}\left(f^{n}\left(\varphi_{0}^{-1}\left([0,1]\times x_{2}\right)\right)\right),\,{\rm Vol}\left(f^{n}\left(\varphi_{0}^{-1}\left(x_{1}\times[0,1] \right)\right)\right)\right\}.$$
Hence,
$$\begin{aligned}
	\int_{\varphi_{0}\left(U_{0}^{i}\right)}
	\left\| DF^{n}_{(x_{1},x_{2})} \right\| \, dx_{1} dx_{2}
	&\leq
	\int_{\varphi_{0}\left(U_{0}^{i}\right)}
	\left(
	\left\| \partial_{1} F^{n}_{(x_{1},x_{2})} \right\|
	+
	\left\| \partial_{2} F^{n}_{(x_{1},x_{2})} \right\|
	\right)
	\, dx_{1} dx_{2}\\
	&\leq
	2K\cdot{\rm Vol}\left(f^{n}\left(\sigma_{n}\right)\right).
\end{aligned}$$

	Applying Theorem \ref{main theorem of decomposition of volume growth of submanifolds} and recalling (\ref{volume growth in local coordinates}) and (\ref{major U0 i}) , we have
	\begin{equation*}
	\begin{aligned}
		\limsup_{n\to+\infty}\frac{1}{n}\log\int_{M}\left\|Df^{n}_{x}\right\|\,dx&=\limsup_{n\to+\infty}\frac{1}{n}\log\int_{U_{0}}\left\|Df^{n}_{x}\right\|\,dx\\
		&\leq\limsup_{n\to+\infty}\frac{1}{n}\log{\rm Vol}\left(f^{n}\left(\sigma_{n}\right)\right)\\
		&\leq \max \left\{h_{\rm top}(f),\,\, \frac{\lambda^{+}(f)}{r}\right\}.
	\end{aligned}
\end{equation*}
	The proof of Theorem \ref{main theorem of decomposition of volume growth of derivatives} is then complete.
\end{proof}
\subsection{Lower bound of the volume growth rate}
The goal of this section is to establish the following result, which can be viewed as a consequence of Pesin theory and is, in a certain sense, already known.
The distinction here lies in the formulation: Theorem~\ref{volume of tangent cocycles bounds entropy} is stated directly in terms of the derivative $Df^{n}_{x}$, rather than the induced map $(Df^{n}_{x})^{\wedge}$ on exterior algebras of tangent spaces, as considered by Przytycki~\cite{Prz80} and Kozlovski~\cite{Koz98}.

\begin{Theorem}\label{volume of tangent cocycles bounds entropy}
	Let $f$ be a $C^r$ ($r>1$) diffeomorphism on a compact surface $M$. $$ h_{\rm top}(f)\leq\liminf_{n\to+\infty}\frac{1}{n}\log\int_{M}\left\|Df^{n}_{x}\right\|\,d x.$$
\end{Theorem}

We first recall a result of Katok (Theorem 1.1 in \cite{Kat80}).

Let $\mu$ be an ergodic measure. Given $\lambda\in(0,1)$, we define $S_{\lambda}(n,\varepsilon)\subset M$ to be a finite set such that
$$\mu\left(\bigcup_{x\in S_{\lambda}(n,\varepsilon)} B(x,n,\varepsilon)\right)\ge \lambda,$$
and whose cardinality is minimal among all finite subsets of $M$ with this property. Given a subset $\Omega$, recall (see \cite{Wal82} for background in ergodic theory)  that a subset $S\subset \Omega$ is called a \emph{$(n,\varepsilon)$-separated set} if for any $x,y\in S$ with $x\neq y$, there is $0\leq k<n$ such that $d(f^{k}(x), f^{k}(y))>\varepsilon$. Moreover, a $(n,\varepsilon)$-separated set $S$ is called \emph{maximal} if its cardinality is larger than or equal to the cardinality of any other $(n,\varepsilon)$-separated set. Given a subset $\Omega$, let $S(n,\varepsilon,\Omega)$ denote a maximal $(n,\varepsilon)$-separated set of $\Omega$. By maximality, the union of the dynamical balls $\{B(x,n,\varepsilon)\}_{x\in S(n,\varepsilon,\Omega)}$ covers $\Omega$.

\begin{Lemma}[Katok, \cite{Kat80}]\label{Katok entropy formula}
	Let $f$ be a homeomorphism on a compact metric space $X$ and let $\mu$ be an ergodic measure.  Then for any $\lambda\in(0,1)$, $$h(f,\mu)=\lim_{\varepsilon\to 0}\liminf_{n\to +\infty}\frac{1}{n}\log\# S_{\lambda}(n,\varepsilon)=\lim_{\varepsilon\to 0}\limsup_{n\to +\infty}\frac{1}{n}\log\# S_{\lambda}(n,\varepsilon).$$
\end{Lemma}
The following lemma, due to Przytycki, as an application of Pesin theory, provides a uniform lower bound on the integral of the derivative over dynamical balls centered at typical points w.r.t. a hyperbolic ergodic measure.
\begin{Lemma}\cite[Lemma 2.1]{Prz80}\label{Przytycki Lemma}
		Let $f$ be a $C^r$ ($r>1$) diffeomorphism on a compact surface $M$ and let $\mu$ be a hyperbolic ergodic measure. There are a subset $\Omega\subset M$ with $\mu\left(\Omega\right)>\frac{1}{2}$ and a constant $C>0$ such that for any $z\in\Omega$, any $n\in\mathbb{N}$ and any $0<\varepsilon<2026$, $$\int_{B(z,n,\varepsilon)}\left\|Df^{n}_{x}\right\|\,dx\geq C\cdot\varepsilon^{2}\cdot\frac{1}{2^{n}}.$$
\end{Lemma}
\begin{Remark}
	The original statement of Lemma~\ref{Przytycki Lemma} is formulated for the induced map  $\left(Df^{n}_{x}\right)^{\wedge}$ acting on the exterior algebra of the tangent space.
	However, one can see from the proof that the same conclusion holds when
	$\left(Df^{n}_{x}\right)^{\wedge}$ is replaced by $\left\|Df^{n}_{x}\right\|$, since for a hyperbolic measure on a surface the unstable dimension is equal to one.
\end{Remark}

\begin{proof}[Proof of Theorem \ref{volume of tangent cocycles bounds entropy}]
	Given a hyperbolic ergodic measure $\mu$, by Lemma \ref{Przytycki Lemma}, there are a subset $\Omega\subset M$ with $\mu\left(\Omega\right)\geq\frac{1}{2}$ and a constant $C>0$ such that for any $z\in\Omega$, any $n\in\mathbb{N}$ and any $0<\varepsilon<2026$, $$\int_{B(z,n,\varepsilon)}\left\|Df^{n}_{x}\right\|\,dx\geq C\cdot\varepsilon^{2}\cdot\frac{1}{2^{n}}.$$ Let $S(n,2\varepsilon,\Omega)$ be a maximal $(n,2\varepsilon)$-separated set of $\Omega$. By definition, for any $z_{1}, z_{2}\in S(n,2\varepsilon,\Omega)$, $$B(z_{1}, n,\varepsilon)\cap B(z_{2}, n,\varepsilon)=\emptyset.$$
We have
$$\begin{aligned}
	\int_{M}\left\|Df^{n}_{x}\right\|\,d x
	&\geq \# S(n,2\varepsilon,\Omega)\cdot\min_{z\in S(n,2\varepsilon,\Omega)}\int_{B(z,n,\varepsilon)} \left\|Df^{n}_{x}\right\|\,dx\\
	&\geq \# S(n,2\varepsilon,\Omega)\cdot C\cdot\varepsilon^{2}\cdot\frac{1}{2^{n}}.
\end{aligned}$$
Then $$\liminf_{n\to+\infty}\frac{1}{n}\log\int_{M}\left\|Df^{n}_{x}\right\|\,d x\geq \liminf_{n\to+\infty}\frac{1}{n}\log \#S(n,2\varepsilon
,\Omega)-\log 2.$$

Since $\mu(\Omega)>\frac{1}{2}$, we have $\#S(n,2\varepsilon,
\Omega)\geq \#S_{\frac{1}{2}}(n,2\varepsilon)$. By Lemma \ref{Katok entropy formula} and the arbitrariness of $\varepsilon$, we then have $$\liminf_{n\to+\infty}\frac{1}{n}\log\int_{M}\left\|Df^{n}_{x}\right\|\,d x\geq h(f,\mu)-\log 2.$$ By the variational principle, we have
\begin{equation}\label{rough estimations of volume growth 3}
\liminf_{n\to+\infty}\frac{1}{n}\log\int_{M}\left\|Df^{n}_{x}\right\|\,d x\geq h_{\rm top}(f)-\log 2.
\end{equation}
Note that $$\liminf_{n\to+\infty}\frac{1}{n}\log\int_{M}\left\|Df^{n}_{x}\right\|\,dx=\liminf_{n\to+\infty}\frac{1}{qn}\log\int_{M}\left\|Df^{qn}_{x}\right\|\,dx,\quad h_{\rm top}(f)=\frac{h_{\rm top}(f^{q})}{q}.$$ Replacing $f$ by $f^{q}$ and putting $\frac{1}{q}$ on both sides in (\ref{rough estimations of volume growth 3}), by the arbitrariness of $q$, we finally get the conclusion: $$\liminf_{n\to+\infty}\frac{1}{n}\log\int_{M}\left\|Df^{n}_{x}\right\|\,dx\geq h_{\rm top}(f).$$

\end{proof}

\section{Appendix}

	\begin{Lemma}\label{volume growth geq entropy abstract prototype}
	Let $f$ be a homeomorphism on a compact metric space $(X, d)$. Given $n\in\mathbb{N}$ and $\varepsilon>0$, let $\left\{B(x_{i}, n,\varepsilon)\right\}_{1\leq i\leq N}$ be a collection of mutually disjoint $(n,\varepsilon)$ dynamical balls. There is $0\leq k_{n}\leq n-1$ such that  $$\sum_{i=1}^{N-1} d\left(f^{k_{n}}\left(x_{i}\right), f^{k_{n}}\left(x_{i+1}\right)\right)\geq \frac{\left(N-1\right)\varepsilon}{n}.$$
   \end{Lemma}
\begin{proof}
	Since the dynamical balls $\left\{B(x_{i}, n,\varepsilon)\right\}_{1\leq i\leq N}$ are mutually disjoint, 
	$$ \max_{0 \leq k \leq n-1} d\left(f^{k}(x_i), f^{k}(x_{i+1})\right) \geq \varepsilon. $$
	Hence,
	$$ \sum_{i=1}^{N-1} \sum_{k=0}^{n-1} d\left(f^{k}(x_i), f^{k}(x_{i+1})\right) \geq \sum_{i=1}^{N-1} \max_{0 \leq k \leq n-1} d\left(f^{k}(x_i), f^{k}(x_{i+1})\right) \geq (N-1)\varepsilon. $$
	
	Changing the order of summation yields:
	$$ \sum_{k=0}^{n-1}  \sum_{i=1}^{N-1} d\left(f^{k}(x_i), f^{k}(x_{i+1})\right) \geq (N-1)\varepsilon. $$
	
	Hence, there must exist at least one index $k_n \in \{0, 1, \dots, n-1\}$ such that
	$$\sum_{i=1}^{N-1} d\left(f^{k_{n}}\left(x_{i}\right), f^{k_{n}}\left(x_{i+1}\right)\right)\geq \frac{\left(N-1\right)\varepsilon}{n}.$$
\end{proof}

Recall that $h_{\rm top}(f,\Omega)$ denotes the topological entropy on a subset $\Omega$, defined in the usual way by $$h_{\rm top}\left(f,\Omega\right):=\lim_{\varepsilon\to 0}\limsup_{n \to +\infty}\frac{1}{n}\log\# S(n,\varepsilon,\Omega)$$ where $S(n,\varepsilon,\Omega)$ denotes a maximal $(n,\varepsilon)$-separated subset of $\Omega$.

\begin{Lemma}\label{volume growth geq entropy}
Let $f$ be a $C^{1}$ diffeomorphism on a compact manifold $M$ and let $\sigma :[0,1]\to M$ be a $C^{1}$ curve. We have $$h_{\rm top}\left(f,\sigma\right)\leq \limsup_{n\to+\infty}\frac{1}{n}\log^{+} {\rm Vol}\left(f^{n}\left(\sigma\right)\right).$$
\end{Lemma}
\begin{proof}
	We may assume $h_{\rm top}\left(f,\sigma\right)>0$, otherwise the inequality is trivially true.

	Given $0<\delta<h_{\rm top}\left(f,\sigma\right)/2$, let $\varepsilon>0$ be such that $$\limsup_{n \to +\infty}\frac{1}{n}\log \# S(n,2\varepsilon,\sigma)\geq h_{\rm top}\left(f,\sigma\right)-\delta>0$$ where $S(n,2\varepsilon,\sigma)$ denotes a maximal $(n,2\varepsilon)$-separated set of $\sigma$. Write $S_{n}:=\#S(n, 2\varepsilon,\sigma)$ for simplicity. We enumerate the points in $S(n,2\varepsilon,\sigma)$ as $\{x_{1}, x_{2}, \cdots,  x_{S_{n}}\}$ in such a way that $\sigma^{-1}\left(x_{i}\right)< \sigma^{-1}\left(x_{i+1}\right)$. Note that by definition, for any $z, z'\in S(n,2\varepsilon,\sigma)$, $$B(z, n,\varepsilon)\cap B(z', n,\varepsilon)=\emptyset.$$ By Lemma \ref{volume growth geq entropy abstract prototype}, there is $0\leq k_{n}\leq n-1$ such that  $$\sum_{i=1}^{S_{n}-1} d\left(f^{k_{n}}\left(x_{i}\right), f^{k_{n}}\left(x_{i+1}\right)\right)\geq \frac{\left(S_{n}-1\right)\varepsilon}{n}.$$ As a consequence, $${\rm Vol}\left(f^{k_{n}}\left(\sigma\right)\right)\geq \frac{\left(S_{n}-1\right)\varepsilon}{n}.$$ Hence, for all large $n$, $$\frac{1}{k_{n}}\log^{+} {\rm Vol}\left(f^{k_{n}}\left(\sigma\right)\right)\geq\frac{n}{k_{n}}\cdot\left(\frac{1}{n}\log\frac{\left(S_{n}-1\right)\varepsilon}{n}\right)\geq \frac{1}{n}\log\frac{\left(S_{n}-1\right)\varepsilon}{n}.$$ Note that $k_{n}\to +\infty$ as $n\to +\infty$. Putting "$\limsup_{n\to+\infty}$" on both sides, we then have $$\limsup_{n\to+\infty}\frac{1}{n}\log^{+} {\rm Vol}\left(f^{n}\left(\sigma\right)\right)\geq \limsup_{n\to+\infty}\frac{1}{k_{n}}\log^{+} {\rm Vol}\left(f^{k_{n}}\left(\sigma\right)\right)\geq h_{\rm top}\left(f,\sigma\right)-\delta.$$ By the arbitrariness of $\delta$, we then have $$\limsup_{n\to+\infty}\frac{1}{n}\log^{+} {\rm Vol}\left(f^{n}\left(\sigma\right)\right)\geq h_{\rm top}\left(f,\sigma\right).$$ 

\end{proof}

The following lemma establishes an upper semi-continuity property for integrals of sub-additive potentials along converging sequences of invariant measures. This result is a special case of \cite[Lemma~2.3]{CFH08}.
	\begin{Lemma}\label{upper continuity of subadditive functions}
	Let $f$ be a homeomorphism on a compact metric space $X$. Let $\mu_{n}$ be a sequence of invariant probability measures converging to some invariant probability measure $\mu$ and let $\varphi_{n}$ be a sequence of sub-additive continuous functions. We have $$\limsup_{n\to+\infty}\int\frac{\varphi_{n}}{n}d\,\mu_{n}\leq \int\limsup_{n\to+\infty}\frac{\varphi_{n}}{n}d\,\mu.$$
	\end{Lemma}

	\begin{proof}
		Fix $m \ge 1$. For any $n$, write $n = km + r$ with $1 \leq r \leq m$. By the sub-additivity of $\varphi_n$, we have
		\[
		\varphi_n \leq \sum_{j=0}^{k-1} \varphi_m \circ f^{jm} + \varphi_r \circ f^{km}.
		\]
		Integrating w.r.t. the $f$-invariant measure $\mu_n$, we obtain
		\[
		\int \varphi_n \, d\mu_n \leq k \int \varphi_m \, d\mu_n + C_m,
		\]
		where $C_m:= \max_{1 \leq r \leq m} \max_{x}\varphi_{r}(x)$. Dividing by $n$ yields
		\[
		\int \frac{\varphi_n}{n} \, d\mu_n \leq \frac{k}{n} \int \varphi_m \, d\mu_n + \frac{C_m}{n} \leq \frac{1}{m} \int \varphi_m \, d\mu_n + \frac{C_m}{n}.
		\]
		Since $\mu_n \to \mu$ and $\varphi_m$ is continuous, letting $n \to +\infty$ gives
		\[
		\limsup_{n \to +\infty} \int \frac{\varphi_n}{n} \, d\mu_n \leq \frac{1}{m} \int \varphi_m \, d\mu.
		\]
		Taking the infimum over $m \ge 1$, and applying Kingman's Sub-additive Ergodic Theorem (specifically, $\inf_m \frac{1}{m} \int \varphi_m \, d\mu = \int \lim_n \frac{\varphi_n}{n} \, d\mu$), we conclude
		\[
		\limsup_{n \to +\infty} \int \frac{\varphi_n}{n} \, d\mu_n \leq \inf_{m \ge 1} \frac{1}{m} \int \varphi_m \, d\mu = \int \limsup_{n \to +\infty} \frac{\varphi_n}{n} \, d\mu.
		\]
	\end{proof}
	\begin{Lemma}[Fatou Lemma]\label{Fatou Lemma}
		Let $X$ be a compact metric space and let $\mu$ be a Borel probability measure on $X$.
		Let $\{\varphi_{n}\}_{n\ge1}$ be a sequence of measurable functions on $X$ such that
		$\sup_{n\ge1}\left\|\varphi_{n}\right\|_\infty < +\infty$.
		Then $$	\int_X \liminf_{n\to+\infty} \varphi_{n}\,d\mu \leq\liminf_{n\to+\infty}\int_X \varphi_{n}\,d\mu,$$
		$$	\int_X \limsup_{n\to+\infty} \varphi_{n}\,d\mu \geq\limsup_{n\to+\infty}\int_X \varphi_{n}\,d\mu.$$
	\end{Lemma}
	
The following lemma (see the book \cite[Lemma 3.33]{Pat99}) gives an abstract comparison between the pointwise exponential growth rate and the exponential growth rate of the integral. We provide a proof for completeness.
\begin{Lemma}\label{exponent less than topological entropy abstract version}
	Let $(X,\mathcal{A},\mu)$ be a probability space and $\varphi_n:X\to(0,+\infty)$ a sequence of integrable functions. Then
	$$
	\limsup_{n\to+\infty}\frac{1}{n}\log \varphi_n(x)
	\le
	\limsup_{n\to+\infty}\frac{1}{n}\log\int_X \varphi_n\,d\mu,
	$$
	for $\mu$-a.e.\ $x\in X$.
\end{Lemma}
\begin{proof}
	Write $$\lambda := \limsup_{n \to +\infty} \frac{1}{n} \log \int_X \varphi_{n} \, d\mu.$$ Given $\varepsilon > 0$, we have
	$$\int_X \varphi_{n} \, d\mu \ge \int_{A_n} \varphi_{n} \, d\mu \ge \mu(A_n)\cdot e^{n(\lambda + \varepsilon)}$$ where $$A_n := \left\{x \in X:~ \varphi_{n}(x) \ge e^{n(\lambda + \varepsilon)}\right\}.$$
	Rearranging and taking the limit yields:
	$$\limsup_{n \to +\infty} \frac{1}{n} \log \mu(A_n) \leq \limsup_{n \to +\infty} \frac{1}{n} \log \int_X \varphi_{n} \, d\mu - (\lambda + \varepsilon) = -\varepsilon < 0.$$
	Hence, $$\sum_{n=1}^{+\infty} \mu(A_n) <+\infty.$$ By the Borel-Cantelli Lemma, for $\mu$-a.e. $x$, there exists $N(x)$ such that for all $n \ge N(x)$, $$\varphi_{n}(x) < e^{n(\lambda + \varepsilon)},$$ which proves the lemma by the arbitrariness of $\varepsilon$.
\end{proof}

\begin{Lemma}\label{simple property of ergodicity}
	Let $f$ be a homeomorphism on a compact metric space $X$. Let $\mu$ be an ergodic probability measure and let $A$ be a subset with $\mu(A)>0$. Let $\varphi: X \to \mathbb{R}$ be a measurable function. If for any $x \in A$, there exists $n_x \in \mathbb{N}$ such that $\varphi(f^{-n}(x))=0$ for all $n \geq n_x$, then $\varphi(x)=0$ for $\mu$-a.e. $x$.
\end{Lemma}

%
%

\begin{proof}
Define $$Z:= \left\{x \in X \mid \exists N \in \mathbb{N} \text{ s.t. } \forall n \geq N, \varphi(f^{-n}(x)) = 0\right\}.$$  
	By definition, $Z$ is $f$-invariant, and $A \subset Z$. Since $\mu(A) > 0$, ergodicity implies $\mu(Z) = 1$. 
	
	If $\mu(\{\varphi \neq 0\}) > 0$, the Poincaré recurrence theorem applied to $f^{-1}$ implies that for $\mu$-a.e. $x \in \{\varphi \neq 0\}$, the backward orbit $f^{-n}(x)$ returns to $\{\varphi \neq 0\}$ infinitely often. This contradicts $\mu(Z) = 1$. Thus, $\varphi = 0$ $\mu$-a.e.
\end{proof}

	\begin{Lemma}[Decomposition into $\varepsilon$-bounded curves]\label{Decomposition into varepsilon bounded curves}
		Let $M$ be a compact manifold and let $r>1$.
		Then for any $0<\varepsilon', \varepsilon<1$, any $C^{r}$ curve $\sigma:[0,1]\to M$ with
		$$
		\left\|\sigma\right\|_{C^{r}}:=\max_{1\leq s\leq r}\left\|d^{s}\sigma\right\|\leq 1 \,\text{ and }\, \left\|d_{t}\sigma\right\| \geq \varepsilon, \,t\in[0,1],
		$$ there are $\varepsilon'$-bounded $C^{r}$ curves $\{\sigma_{j}\}_{1\leq j \leq \left\lceil \frac{10}{\varepsilon\cdot\varepsilon'} \right\rceil}$ such that
		$$
		\sigma\left([0,1]\right)=\bigcup_{j=1}^{\left\lceil \frac{10}{\varepsilon\cdot\varepsilon'} \right\rceil}\sigma_{j}\left([0,1]\right).
		$$
	\end{Lemma}

	\begin{proof}
		Take $N = \left\lceil \frac{10}{\varepsilon\cdot\varepsilon'} \right\rceil$ and define the step size $\delta = \frac{1}{N}$. Note that $\delta \leq \varepsilon'$ and $\frac{\delta}{\varepsilon}\leq \frac{1}{6}$.

		Define the translation constants $b_j=(j-1)\delta$ for $j=1, \dots, N$. 
		The affine maps $\theta_j: [0,1] \to [0,1]$ are given by $\theta_j(t) = \delta t + b_j$. By construction, $\theta_j([0,1]) = [b_j, b_j+\delta] \subset [0,1]$, and the union of these images covers $[0,1]$.

		Let $\sigma_j = \sigma \circ \theta_j$. We verify the $\varepsilon'$-boundedness for each $1 \leq j \leq N$:

		\begin{enumerate}
			\item \textbf{First derivative bound.}

			Since $\left\|d_t\sigma\right\| \leq 1$ and $\theta_j' = \delta$, we have
			$$ \left\|d\sigma_j\right\| = \sup_{t \in [0,1]} \left\|d_t\sigma \cdot \theta_j'\right\| \leq 1 \cdot \delta \leq \varepsilon'. $$

			\item \textbf{Higher order derivative bound.}

			For $2 \leq s \leq r$, the chain rule gives $\left\|d^s\sigma_j\right\| \leq \delta^s \left\|d^s\sigma\right\| \leq \delta^s$. Recalling the lower bound $\left\|d_t\sigma\right\| \ge \varepsilon$, we have $\left\|d\sigma_j\right\| \ge \varepsilon\cdot\delta$. Consequently,
			$$ \frac{\left\|d^s\sigma_j\right\|}{\left\|d\sigma_j\right\|} \leq \frac{\delta^s}{\varepsilon\cdot\delta} = \frac{\delta^{s-1}}{\varepsilon} \leq \frac{\delta}{\varepsilon}\leq\frac{1}{6}, $$
			which satisfies the boundedness condition for all $2 \leq s \leq r$.
		\end{enumerate}
		Therefore, each $\sigma_j$ is $\varepsilon'$-bounded, and $\sigma([0,1]) = \bigcup_{j=1}^N \sigma_j([0,1])$.
	\end{proof}

	The following example shows that the local volume growth rate can still be positive in a small neighborhood of a repeller. Here the positivity of the local volume growth is a purely local phenomenon: although the curve is exponentially restricted to an increasingly small neighborhood of the repeller, the rapidly growing oscillation frequency produces a strong transverse amplification of derivatives. The resulting growth reflects a subtle balance between geometric restriction (uniform expansion) and oscillatory stretching.
	\begin{Example}\label{example repeller positive volume growth}
Let $A = \operatorname{diag}(a, 3)$ with $a > 1$, and let $\sigma: [0, 1] \to \mathbb{R}^2$ be a $C^2$ curve defined by $\sigma(x) = (x, x^5 \sin(1/x))$. The volume (i.e., arc length) growth rate restricted to $[-1,1]^2$ \footnote{The restriction to $[-1,1]^2$ is meant to mimic the restriction to a dynamical ball, which is justified here by the uniform expansion of the map.}, defined by
\[ {\rm Vol}(A, \sigma)_{[-1,1]^2} := \limsup_{n \to +\infty} \frac{1}{n} \log \operatorname{Length}(A^n \sigma \cap [-1,1]^2), \]
is given by:
\[ {\rm Vol}(A, \sigma)_{[-1,1]^2} =
\begin{cases}
	\frac{1}{5} \log 3, & 1 < a \leq 3^{1/5} \\
	\log 3 - 4 \log a, & 3^{1/5} < a < 3^{1/4} \\
	0, & a \ge 3^{1/4}
\end{cases} \]
In particular, the growth rate is positive if and only if $1 < a < 3^{1/4}$.
	\end{Example}

	\begin{proof}
		\textbf{Step 1. Parametrization and derivatives.}

		Inside $[-1,1]^2$, the image curve $A^n\sigma$ admits the parametrization
		$$
		\gamma_n(u)=(u,g_n(u)), \qquad u\in[0,1],
		$$
		where
		$$
		g_n(u)=3^n a^{-5n}u^5\sin\!\left(\frac{a^n}{u}\right).
		$$

		The admissible domain to have $\gamma_n(u)=(u,g_n(u))\in[-1,1]^2$ is determined by
		$$
		u\in[0,1],\qquad |g_n(u)|\le1.
		$$
		Denote by $I_n$ the set of $u$ satisfying these conditions. The length of
		$A^n\sigma\cap[-1,1]^2$ is
		$$
		L_n:=\int_{I_n}\sqrt{1+|g_n'(u)|^2}\,du.
		$$
		A direct computation gives
		$$
		g_n'(u)
		=5\cdot3^n a^{-5n}u^4\sin\!\left(\frac{a^n}{u}\right)
		-3^n a^{-4n}u^3\cos\!\left(\frac{a^n}{u}\right).
		$$

		\medskip
		\textbf{Step 2. Case analysis.}

	The figure \ref{example} below illustrates how the local volume growth transitions from bounded oscillations to amplified and eventually explosive behavior as the parameter $a$ varies across different regimes.
		\begin{figure}[htbp]
			\centering
			\includegraphics[width=1\textwidth]{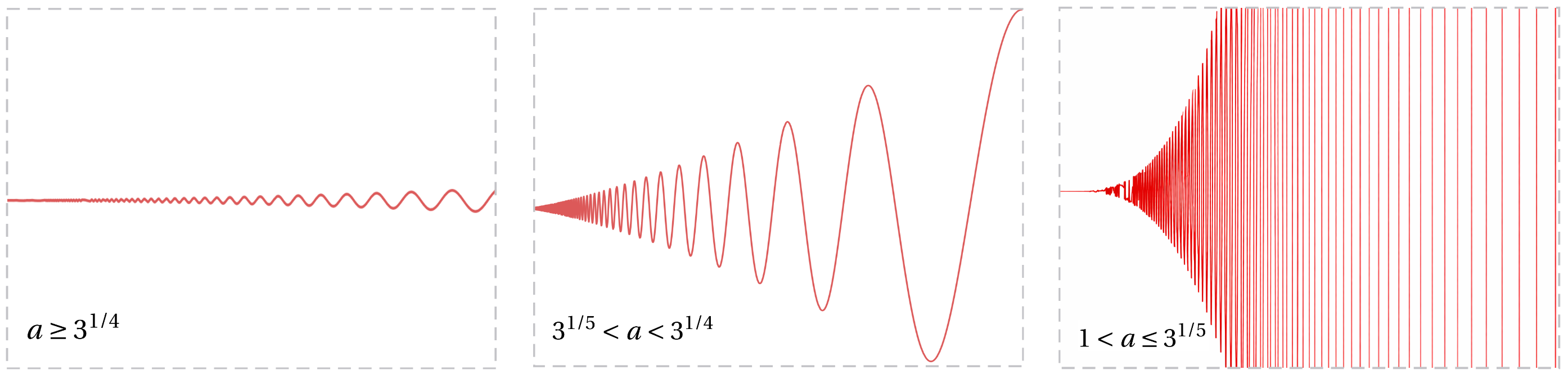}
			\caption{Local volume growth for different values of $a$}
			\label{example}
		\end{figure}

		\smallskip

		\begin{itemize}
			\item 	\emph{Case 1: $a\ge 3^{1/4}$.}

			In this range, $3a^{-4}\le1$. For $u\in[0,1]$, we have
			$$
			|g_n(u)|\leq (3a^{-5})^n\leq 1,
			$$
			which implies $I_{n}=[0,1]$. Also note
			$$
			|g_n'(u)|
			\leq 5\cdot3^n a^{-5n}+3^n a^{-4n}
			=5(3a^{-5})^n+(3a^{-4})^n.
			$$
			As $a>1$ and $3a^{-4}\le1$, there exists $C>0$ such that
			$|g_n'(u)|\leq C$ for all $n$ and $u\in I_n$. Hence,
			$$
			1\leq \sqrt{1+|g_n'(u)|^2}\le\sqrt{1+C^2},
			$$
			which implies
			$$
			{\rm Vol}(A, \sigma)_{[-1,1]^2}=\limsup_{n\to+\infty}\frac1n\log L_n\leq0.
			$$
			Since $I_{n}=[0,1]$, the image curve has length at least one. Hence, $${\rm Vol}(A, \sigma)_{[-1,1]^2}=0$$

			\item  \emph{Case 2: $3^{1/5}<a<3^{1/4}$.}

			In this range, $3a^{-5}<1$ and $3a^{-4}>1$. Since
			$$
			|g_n(u)|\leq (3a^{-5})^n\leq 1,
			$$
			we have $I_n=[0,1]$ as in the previous case.

			Note $$5\cdot3^n a^{-5n}u^4\sin\!\left(\frac{a^n}{u}\right)\leq 5\cdot3^n a^{-5n}\to 0.$$
			As a consequence, for all large $n$, $$\sqrt{1+|g_n'(u)|^2}\le1+|g_n'(u)|\leq 2+ 3^n a^{-4n}.$$ Hence, $${\rm Vol}(A, \sigma)_{[-1,1]^2}\leq \log3-4\log a.$$ It remains to prove the reverse inequality.

			 For all large $n$, by Lemma \ref{Calculus results} with $b=1$, $$L_{n}\geq \int_{0}^{1}|g_n'(u)|du\geq \int_{0}^{1}\Big|3^n a^{-4n}u^3\cos\!\left(\frac{a^n}{u}\right)\Big|du-1\geq 3^n a^{-4n}\cdot \frac{1}{2\pi(1+2\pi)^{4}} -1.$$ Hence, $${\rm Vol}(A, \sigma)_{[-1,1]^2}\geq \log3-4\log a.$$

			\item 	\emph{Case 3: $1<a\leq 3^{1/5}$.}

			In this range, $3a^{-5}\ge1$.
			For a lower bound, consider $[0,u_n]$ with
			$$
			u_n:=(3^{n}a^{-5n})^{-1/5}=3^{-n/5}a^n.
			$$
			Then $|g_n(u)|\le1$ on $[0,u_n]$, hence $[0,u_n]\subset I_n$. Also note $$\sqrt{1+|g_n'(u)|^2}\geq |g_n'(u)|\geq  \left|3^n a^{-4n}u^3\cos\!\left(\frac{a^n}{u}\right)\right|-5\cdot3^n a^{-5n}u^4.$$ Applying Lemma \ref{Calculus results} with $b=u_n=3^{-n/5}a^n$, we have $$\int_{0}^{u_{n}}\left|3^n a^{-4n}u^3\cos\!\left(\frac{a^n}{u}\right)\right|du\geq \frac{1}{2\pi\left(1+2\pi\right)^{4}}\cdot 3^{n/5}.$$ Hence, $$L_{n}\geq \int_{0}^{u_{n}}\left|3^n a^{-4n}u^3\cos\!\left(\frac{a^n}{u}\right)\right|du-\int_{0}^{u_{n}}5\cdot3^n a^{-5n}u^4du\geq \frac{1}{2\pi\left(1+2\pi\right)^{4}}\cdot 3^{n/5}-1$$ which gives $${\rm Vol}(A, \sigma)_{[-1,1]^2}\geq \frac15\log3.$$

			It remains to prove the reverse inequality. On $[0, u_{n}]$, $$\sqrt{1+|g_n'(u)|^2}\leq |g_n'(u)|+1\leq 3^n a^{-4n}u^3+5\cdot3^n a^{-5n}u^4 +1.$$ By a direct computation,  $$\limsup_{n \to +\infty}\frac{1}{n}\log\int_{0}^{u_{n}}\left(3^n a^{-4n}u^3+5\cdot3^n a^{-5n}u^4 +1\right)du \leq \frac15\log3.$$ So we only have to focus on the volume growth on the interval $[u_{n},1]$.

			We next estimate the number of intervals of monotonicity for $g_n$ on $[u_{n},1]$. Note that the critical points (i.e., $g'_{n}(u)=0$) are determined by
			\[
			5u\sin\!\frac{a^n}{u}
			=
			a^n\cos\!\frac{a^n}{u}.
			\]
			Introducing the variable $t=a^n/u$, this equation is equivalent to
			\[
			\tan t=\frac{t}{5}.
			\]
			For $u\in [u_{n},1]$, we have
			\[
			t\in\Bigl[a^n,\frac{a^n}{u_n}\Bigr]=[a^n,3^{n/5}].
			\]
			On each interval $(k\pi-\frac{\pi}{2},k\pi+\frac{\pi}{2})$, the function
			$\tan t - t/5$ is strictly increasing, hence the above equation admits at most one solution.
			Therefore, the number of solutions in $[a^n,3^{n/5}]$ is bounded by
			\[
			\frac{3^{n/5}-a^n}{\pi}+1 \leq \frac{3^{n/5}}{\pi}+1.
			\]
			Each critical point contributes at most one change of monotonicity, so the number of intervals of monotonicity on $[u_{n},1]$ is bounded by
			\[
			\frac{3^{n/5}}{\pi}+2.
			\]

			Restricted on $[-1,1]^2$, the length of the curve on each interval of monotonicity is at most $4$. Hence, the volume growth rate on the interval $[u_{n},1]$ restricted to $[-1,1]^2$ is also bounded above by $\frac{1}{5}\log3.$ This proves the reverse inequality.

		\end{itemize}

	\end{proof}

		\begin{Lemma}\label{Calculus results}
		For any $a,b>0$ and $n\in\mathbb{N}$,
		$$
		\int_{0}^{b} u^{3}\left|\cos\frac{a^{n}}{u}\right|\,du
		\ge
		\frac{a^{4n} b^{4}}{2\pi\left(a^{n}+2\pi b\right)^{4}}.
		$$
	\end{Lemma}

	\begin{proof}
		By the change of variables $t=a^{n}/u$, we obtain
		$$
		\int_{0}^{b} u^{3}\left|\cos\frac{a^{n}}{u}\right|\,du
		=
		a^{4n}\int_{a^{n}/b}^{+\infty} t^{-5}|\cos t|\,dt .
		$$

		Let
		$$
		k_{0}:=\left\lceil \frac{a^{n}}{b\pi}\right\rceil .
		$$
		Since $t^{-5}|\cos t|\ge 0$, we have
		$$
		\int_{a^{n}/b}^{\infty} t^{-5}|\cos t|\,dt
		\ge
		\sum_{k\ge k_{0}}
		\int_{k\pi}^{(k+1)\pi} t^{-5}|\cos t|\,dt .
		$$
		For each $k\ge k_{0}$, using that $t^{-5}$ is decreasing on $[k\pi,(k+1)\pi]$ and
		$$
		\int_{k\pi}^{(k+1)\pi}|\cos t|\,dt=2,
		$$
		we obtain
		$$
		\int_{k\pi}^{(k+1)\pi} t^{-5}|\cos t|\,dt
		\ge
		2\,((k+1)\pi)^{-5}.
		$$
		Therefore,
		$$
		\int_{a^{n}/b}^{\infty} t^{-5}|\cos t|\,dt
		\ge
		2\sum_{k\ge k_{0}}((k+1)\pi)^{-5}
		\ge
		\frac{2}{\pi}\int_{(k_{0}+1)\pi}^{\infty} t^{-5}\,dt
		=
		\frac{1}{2\pi}((k_{0}+1)\pi)^{-4}.
		$$
		Since $(k_{0}+1)\pi\leq a^{n}/b+2\pi$, it follows that
		$$
		\int_{a^{n}/b}^{\infty} t^{-5}|\cos t|\,dt
		\ge
		\frac{1}{2\pi}(a^{n}/b+2\pi)^{-4}.
		$$
		Consequently,
		$$
		\int_{0}^{b} u^{3}\left|\cos\frac{a^{n}}{u}\right|\,du
		\ge
		\frac{a^{4n} b^{4}}{2\pi\left(a^{n}+2\pi b\right)^{4}}
		.$$
	\end{proof}

\section*{Acknowledgements}
		
		We are grateful to Mingyang Xia from Dalian University of Technology for many helpful and stimulating discussions during the early stages of this work. We also thank J\'{e}r\^{o}me Buzzi from Universit\'{e} Paris-Saclay for many valuable comments.
		
		ChatGPT (free version) and DeepSeek were used to assist with language-related tasks.
	
		Yuntao Zang is supported by National Key R\&D Program of China (2022YFA1005802) and NSFC (12201445, 12471186).

\vskip 5pt

\flushleft{\bf Yuntao Zang} \\
\small School of Mathematical Sciences,  Soochow University, Suzhou, 215006, P.R. China\\
\textit{E-mail:} \texttt{ytzang@suda.edu.cn}\\

\end{document}